\documentclass[a4paper,oneside,11pt]{article} 

\usepackage[utf8]{inputenc} 
\usepackage[T1]{fontenc} 
\usepackage{textcomp} 
\usepackage[english]{babel} 
\usepackage{indentfirst}
\usepackage[autolanguage]{numprint} 
\usepackage{hyperref} 
\hypersetup{colorlinks=true,linkcolor=blue,citecolor=red,urlcolor=blue}
\usepackage{graphicx} 
\usepackage{verbatim} 
\usepackage{makeidx} 
\usepackage{enumitem}
\usepackage{tikz}
\usetikzlibrary{matrix,arrows,patterns}

\usepackage{amsmath,amssymb,amsthm} 
\usepackage[all]{xy} 

\usepackage{fancyhdr} 
\newcommand\gts[1]{``#1''} 
\swapnumbers 
\theoremstyle{definition} 
\newtheorem{Def}{Definition}[subsection] 
\newtheorem{Def,Thm}[Def]{Theorem-definition} 
\newtheorem{Def,Prop}[Def]{Proposition-definition} 
\newtheorem{Not}[Def]{Notation} 
\newtheorem{Hyp}[Def]{Hypothesis} 

\theoremstyle{plain} 
\newtheorem{Prop}[Def]{Proposition} 
\newtheorem{Lem}[Def]{Lemma} 
\newtheorem{Thm}[Def]{Theorem} 
\newtheorem{Cor}[Def]{Corollary} 

\theoremstyle{remark} 
\newtheorem{Ex}[Def]{Example} 
\newtheorem{Rq}[Def]{Remark} 

\title{Explicit generators of some pro-$p$ groups via Bruhat-Tits theory} 
\author{Benoit Loisel} 
\date{\today}

\makeindex

\begin{document} 

\maketitle 

\abstract{
Given a semisimple group over a local field of residual characteristic $p$, its topological group of rational points admits maximal pro-$p$ subgroups.
Quasi-split simply-connected semisimple groups can be described in the combinatorial terms of valued root groups, thanks to Bruhat-Tits theory.
In this context, it becomes possible to compute explicitly a minimal generating set of the (all conjugated) maximal pro-$p$ subgroups thanks to parametrizations of a suitable maximal torus and of corresponding root groups.
We show that the minimal number of generators is then linear with respect to the rank of a suitable root system.
}

\tableofcontents

\section{Introduction}

In this paper, a smooth connected affine group scheme of finite type over a field $K$ will be called a $K$-group.
Given a base field $K$ and an $K$-group denoted by $G$, we get an abstract group called the group of rational points, denoted by $G(K)$.
When $K$ is a non-Archimedean local field, this group inherits a topology from the field.
In particular, the topological group $G(K)$ is totally disconnected and locally compact.
The maximal compact or pro-$p$ subgroups of such a group $G(K)$, when they exist, provide a lot of examples of profinite groups.
Thus, one can investigate maximal pro-$p$ subgroups from the profinite group theory point of view.

\subsection{Minimal number of generators}
\label{sec:intro:minimal:number}

When $H$ is a profinite group, we say that $H$ is \textbf{topologically generated} by a subset $X$ if $H$ is equal to its smallest closed subgroup containing $X$; such a set $X$ is called a \textbf{generating set}.
We investigate the minimal number of generators of a maximal pro-$p$ subgroup of the group of rational points of an algebraic group over a local field.

Suppose that $K = \mathbb{F}_q((t))$ is a nonzero characteristic local field, where $q = p^m$ and $G$ is a simple $K$-split simply-connected $K$-group of rank $l$.
By a recent result of Capdeboscq and R\'emy \cite[2.5]{CapdeboscqRemy}, we know that any maximal pro-$p$ subgroup of $G(K)$ admits a finite generating set $X$;
moreover, the minimal number of elements of such a $X$ is $m(l+1)$.

In the general situation of a smooth algebraic $K$-group scheme $G$, we know by \cite[1.4.3]{Loisel-maximaux} that an algebraic group over a local field admits maximal pro-$p$ subgroups (called pro-$p$ Sylows) if, and only if, it is quasi-reductive (the split unipotent radical is trivial).
When $K$ is of characteristic $0$, this corresponds to reductive groups because a unipotent group is always split over a perfect field.
To provide explicit descriptions of a pro-$p$ Sylow thanks to Bruhat-Tits theory, we restrict the study to the case of a semisimple group $G$ over a local field $K$.

Such a group $G$ can be decomposed as an almost direct product of almost-$K$-simple groups.
Moreover, by \cite[6.21]{BorelTits}, we know that for any almost-$K$-simple simply connected group $H$, there exists a finite extension of local fields $K'/K$ and an absolutely simple $K'$-group $H'$ such that $H$ is isomorphic to the Weil restriction $R_{K'/K}(H')$, that means $H'$ seen as a $K$-group.
Since $H(K) = H'(K')$ by definition of the Weil restriction, we can assume that $G$ is absolutely simple.

In the Bruhat-Tits theory, given a reductive $K$-group $G$, we define a poly-simplicial complex $X(G,K)$ (a Euclidean affine building), called the Bruhat-Tits building of $G$ over $K$ together with a suitable action of $G(K)$ onto $X(G,K)$.
There exists a non-ramified extension $K'/K$ such that the $K$-group $G$ quasi-splits over $K'$.
There are two steps in the theory.
The first part, corresponding to chapter 4 of \cite{BruhatTits2}, provides the building $X(G',K')$ of $G_{K'}$ by gluing together affine spaces, called apartments.
The second part, corresponding to chapter 5 of \cite{BruhatTits2}, applies a Galois descent to the base field $K$, using fixed point theorems.

In the non quasi-split case, the geometry of the building does not faithfully reflect the structure of the group.
There is an anisotropic kernel of the action of $G(K)$ on $X(G,K)$.
As an example, when $G$ is anisotropic over $K$, its Bruhat-Tits building is a point; the Bruhat-Tits theory completely fails to be explicit in combinatorial terms for anisotropic groups.
Thus, the general case may require, moreover, arithmetical methods.
Hence, to do explicit computations with a combinatorial method based on Lie theory, we have to assume that $G$ contains a torus with enough characters over $K$.
More precisely, we say that a reductive group $G$ is \textbf{quasi-split} if it admits a Borel subgroup defined over $K$ or, equivalently, if the centralizer of any maximal $K$-split torus is a torus \cite[4.1.1]{BruhatTits2}.

Now, assume that $K$ is any non-Archimedean local field of residual characteristic $p \neq 2$ and residue field $\kappa \simeq \mathbb{F}_q$ where $q = p^m$.
Let $G$ be an absolutely-simple simply-connected quasi-split $K$-group.
\begin{Thm}
Denote by $l$ the rank of the relative root system of $G$, and by $n$ the rank of the absolute root system of $G$.
Assume that $l \geq 2$.
If $G$ has a relative root system $\Phi$ of type $G_2$ or $BC_l$, assume that $p \neq 3$.
Let $P$ be a maximal pro-$p$ subgroup of $G(K)$.
Denote by $d(P)$ the minimal number of generators of $P$.
Then, we have:
$$d(P) = m(l+1) \text{ or } m(n+1)$$
depending on whether the minimal splitting field extension of short roots is ramified or not.
\end{Thm}

This theorem is formulated more precisely and proven in Corollary \ref{cor:minimal:number:topological:generators}.
According to \cite[4.2]{SerreCohomologieGaloisienne}, we know that $d(P)$ can also be computed via cohomological methods: $d(P) = \mathrm{dim}_{\mathbb{Z} / p \mathbb{Z} } H^1(P, \mathbb{Z} / p \mathbb{Z})=\mathrm{dim}_{\mathbb{Z} / p \mathbb{Z} } \mathrm{Hom}(P,\mathbb{Z} / p \mathbb{Z})$.

From now on, we need to be more explicit.
In the following, given a local field $L$, we denote by $\omega_L$ the discrete valuation on $L$, by $\mathcal{O}_L$ the ring of integers, by $\mathfrak{m}_L$ its maximal ideal, by $\varpi_L$ a uniformizer, and by $\kappa_L = \mathcal{O}_L / \mathfrak{m}_L$ the residue field.
Because we have to compare valuations of elements in $L^*$, we will normalize the discrete valuation $\omega_L : L^* \rightarrow \mathbb{Q}$ so that $\omega_L(L^*) = \mathbb{Z}$.
When $l \in \mathbb{R}$, we denote by $\lfloor l \rfloor$ the largest integer less than or equal to $l$ and by $\lceil l \rceil$ the smallest integer greater than or equal to $l$. 

If it is clear in the context, we can omit the index $L$ in these notations.
When $L/K$ is an extension, we denote by $G_L$ the extension of scalars of $G$ from $K$ to $L$.
When $H$ is an algebraic $L$-group, we denote by $R_{L/K}(G)$ the $K$-group obtained by the Weil restriction functor $R_{L/K}$ defined in \cite[I§1 6.6]{DemazureGabriel}.

\subsection{Pro-\texorpdfstring{$p$}{p} Sylows and their Frattini subgroups}
\label{sec:intro:profinite:theory}

In a general context, let $K$ be a global field and $\mathcal{V}$ its set of places (i.e. valuations of $K$). Let $R \leq K$ be a Dedekind domain bounded except over a finite set $\mathcal{S} \subset \mathcal{V}$ of places.
For any $v \in \mathcal{V} \setminus \mathcal{S}$, we consider the $v$-completion $R_v$ of $R$.
We get a first completion $\widetilde{G(R)} = \prod_{v \in \mathcal{V} \setminus \mathcal{S}} G(R_v)$.
We get a second completion of $G(R)$ by considering its profinite completion denoted by $\widehat{G(R)}$.
The \textit{congruence subgroup problem} is to know when the natural map $\widehat{G(R)} \rightarrow \widetilde{G(R)}$ is surjective with finite kernel.
For example, when $G=\mathrm{SL_n}$ with $n \geq 2$ and $R=\mathbb{Z}$, by a theorem of Matsumoto \cite{Matsumoto-ProblemeSGC}, the surjective map $\widehat{\mathrm{SL_n}(\mathbb{Z})} \rightarrow \prod_{p} \mathrm{SL}_n(\mathbb{Z}_p)$ has finite kernel if, and only if, $n \geq 3$.

Here, we focus on a single factor and, more precisely, on a pro-$p$ Sylow of a factor $G(R_v)$.
More precisely, $K$ is a non-Archimedean local field and $G$ is a semisimple $K$-group.
We consider a maximal pro-$p$ subgroup $P$ of $G(K)$.
When $G$ is simply connected, we know by \cite[1.5.3]{Loisel-maximaux}, that there exists a model $\mathcal{G}$ provided by Bruhat-Tits theory, that means a $\mathcal{O}_K$-group with generic fiber $\mathcal{G}_{K} = G$, such that we can identifies $P$ with the kernel of the natural surjective quotient morphism $\mathcal{G}(\mathcal{O}_K) \rightarrow \Big( \mathcal{G}_{\kappa} / \mathcal{R}_u \big( \mathcal{G}_\kappa \big) \Big) (\kappa)$.
In another words, the pro-$p$ Sylow $P$ is the inverse image of a $p$-Sylow among the surjective homomorpshism $\mathcal{G}(\mathcal{O}_K) \rightarrow \mathcal{G}(\kappa)$.

To compute the minimal number of generators, the theory of profinite groups provides a method consisting of computing the Frattini subgroup.
The Frattini subgroup of a pro-$p$ group $P$ consists of non-generating elements and can be written as $\mathrm{Frat}(P) = \overline{[P,P] P^p}$, the smallest closed subgroup generated by $p$-powers and commutators of elements of $P$ \cite[1.13]{DixonDuSautoyMannSegal}.
Once the group $\mathrm{Frat}(P)$ has been determined, it becomes immediate to provide a minimal topologically generating set $X$ of $P$, arising from finite generating set of $P / \mathrm{Frat}(P)$.

From this writing, we observe that the computation of the Frattini subgroup of $P$ is mostly the computation of its derived subgroup.
Despite $P$ is close to be an Iwahori subgroup $I$ of $G(K)$ (in fact, $I = \mathcal{N}_{G(K)}(P)$ is an Iwahori subgroup and $P$ has finite index in $I$), we cannot use the results of \cite[§6]{PrasadRaghunathan1} because there are less toric elements in $P$ than in $I$.
However, computations of Section \ref{sec:commutation:relations:general} have some similarities with compurations of Prasad and Raghunathan.

We say that $P$ is finitely presented as pro-$p$ group if there exists a closed normal subgroup $R$ of the free pro-$p$ group $\widehat{F_n}^p$ generated by $n$ elements such that $P \simeq \widehat{F_n}^p / R$ and $R$ is finitely generated as a pro-$p$ group.
Let $r(P)$ be the minimum of all the $d(R)$ among the $R$ and $n \geq d(P)$.
According to \cite[4.3]{SerreCohomologieGaloisienne}, $P$ is finitely presented as pro-$p$ group if, and only if $H^2(G,\mathbb{Z}/p \mathbb{Z})$ is finite.
In this case, we get $r(P) =\mathrm{dim}_{\mathbb{Z} / p \mathbb{Z} } H^2(G,\mathbb{Z}/p \mathbb{Z})$ and, for any $R$, we have $d(R) = n - d(P) + r(P)$.
Note that $r(P)$ does not depend on the choice of a generating set and we can choose simultaneously a minimal generating set and a minimal family of relations.
More generally, Lubotzky has shown \cite[2.5]{Lubotzky-ProfinitePresentations} that any finitely presented profinite group $P$ can be presented by a minimal presentation as a profinite group.
If we can show that $H^2(G,\mathbb{Z}/p \mathbb{Z})$ is finite, then, by \cite[12.5.8]{Wilson}, we get the Golod-Shafarevich inequality $r(P) \geq \frac{d(G)^2}{4}$.
This has to be the case according to study of $\mathcal{O}_K$-standard groups of Lubotzky and Shalev \cite{LubotzkyShalev}.

Here, the main result is a description of the Frattini subgroup of $P$, denoted by $\mathrm{Frat}(P)$, in terms of valued root groups datum.
We assume that $K$ is a non-Archimedean local field of residue characteristic $p$ and that $G$ is a semisimple and simply-connected $K$-group.
To simplify the statements, we assume, moreover, that $G$ is absolutely almost simple;
this is equivalent to assuming that the absolute root system $\widetilde{\Phi}$ is irreducible.
We know that it is possible to describe a maximal pro-$p$ subgroup $P$ of $G(K)$ in terms of the valued root groups datum \cite[3.2.9]{Loisel-maximaux}.
A maximal poly-simplex of the building $X(G,K)$, seen as poly-simplicial complex, is called an alcove.
We denote by $\mathbf{c}_{\mathrm{af}}$ a well-chosen alcove to be a fundamental domain of the action of $G(K)$ on $X(G,K)$.
Any maximal pro-$p$ subgroup of$G(K)$ fixes a unique alcove.
Up to conjugation, we can assume that $\mathbf{c} = \mathbf{c}_{\mathrm{af}}$ is the only alcove fixed by $P$.
It is then possible to describe the Frattini subgroup in terms of the valued root groups datum, as stated in the following two theorems:

\begin{Thm}
We assume that $p \neq 2$ and,
if $\Phi$ is of type $G_2$ or $BC_l$, we assume that $p \geq 5$.

Then the pro-$p$ group $P$ is topologically of finite type and, in particular, $\mathrm{Frat}(P) = P^p [P,P]$.
Moreover, when $K$ is of characteristic $p > 0$, we have $P^p \subset [P,P]$.

The Frattini subgroup $\mathrm{Frat}(P)$ can be written as a directly generated product in terms of the valued root groups datum.

When $\Phi$ is reduced (that means is not of type $BC_l$), then $\mathrm{Frat}(P)$ is the maximal pro-$p$ subgroup of the pointwise stabilizer in $G(K)$ of the combinatorial ball centered at $\mathbf{c}$ of radius $1$.
\end{Thm}

For a more precise version, see Theorems \ref{thm:frattini:descriptions:reduced:case} and \ref{thm:frattini:descriptions:non:reduced:case}.

\subsection{Structure of the paper}
\label{sec:intro:structure}

We assume that $G$ is a simply-connected quasi-split semisimpe $K$-group.
We fix a maximal Borel subgroup $B$ of $G$ defined over $K$.
In particular, this choice determines an order $\Phi^+$ of the root system and a basis $\Delta$.
By \cite[20.5, 20.6 (iii)]{Borel},
there exists a maximal $K$-split torus $S$ in $G$ such that its centralizer, denoted by $T = \mathcal{Z}_G(S)$, is a maximal $K$-torus of $G$ contained in $B$.
We fix a separable closure $K_s$ of $K$;
by \cite[8.11]{Borel}, there exists a unique smallest Galois extension of $K$, denoted by $\widetilde{K}$, splitting $T$, hence also splitting $G$ by \cite[18.7]{Borel}.
We call \textbf{the relative root system}, denoted by $\Phi$, the root system of $G$ relatively to $S$.
We call \textbf{the absolute root system}, denoted by $\widetilde{\Phi}$, the root system of $G_{\widetilde{K}}$ relatively to $T_{\widetilde{K}}$.
In Section \ref{sec:star:action}, we define a $\mathrm{Gal}(K_s / K)$-action on $\widetilde{\Phi}$ which preserves the Dynkin diagram structure of $\mathrm{Dyn}(\widetilde{\Delta})$ and on its basis $\widetilde{\Delta}$ corresponding to the Borel subgroup $B$.
According to \cite[4.2.23]{BruhatTits2}, when $G$ is absolutely simple (hence $\mathrm{Dyn}(\widetilde{\Delta})$ is connected), the group $\mathrm{Aut}\big(\mathrm{Dyn}(\widetilde{\Delta})\big)$ is a finite group of order $d \leq 6$.
As a consequence, the degree of each splitting field extension is small and does not interact a lot with Lie theory.
One can note that a major part of proofs in this paper is taken by the non-reduced $BC_l$ cases and the trialitarian $D_4$ cases.

From this action and thanks to a rank $1$ consideration, we define, according to \cite[§4.2]{BruhatTits2}, a coherent system of parametrizations of root groups in Section \ref{sec:parametrization} together with a filtration of the root groups in Section \ref{sec:valuation:of:root:groups:datum}.
This provides us a generating valued root groups datum $\Big( T(K), \big( U_a(K), \varphi_a \big)_{a\in \Phi} \Big)$ built from $(G,S,K,\widetilde{K})$.
This filtration corresponds to a prescribed affinisation of the spherical root system $\Phi$.
From this, we compute, in Sections \ref{sec:reduced:case} and \ref{sec:non:reduced:case}, various commutation relations between unipotent and semisimple elements in rank $1$.
This will be useful to describe, in Section \ref{sec:action:on:ball}, the action of $P$ onto a combinatorial ball centered at $\mathbf{c}$ of radius $1$.
This will also be useful in Section \ref{sec:frattini:subgroup} to generate semisimple elements of $\mathrm{Frat}(P)$.

We denote by $\mathbb{A} = A(G,S,K)$ the \gts{standard} apartment and we choose a fundamental alcove $\mathbf{c}_{\mathrm{af}} \subset \mathbb{A}$,
to be a fundamental domain of the action of $G(K)$ on $X(G,K)$.
Those objects will be described in Section \ref{sec:walls} and \ref{sec:description:alcove} respectively thanks to the sets of values, defined in Section \ref{sec:set:of:values}, which measure where the gaps between two terms of the filtration are and, in the non-reduced case, what kind of gaps we must deal with.
From this, we deduce, in Section \ref{sec:counting:alcoves}, the geometrical description of the combinatorial ball centered at $\mathbf{c}$ of radius $1$.
Consequently, the geometric situation provides, in Section \ref{sec:action:on:ball}, an upper bound for $\mathrm{Frat}(P)$, that means a group $Q$ containing $\mathrm{Frat}(P)$.

Thus, we seek a generating set of $Q$ contained in $\mathrm{Frat}(P)$.
From the writing $\displaystyle \mathrm{Frat}(P) = \overline{P^p [P,P]}$, we seek such a generating set by commuting elements of $P$.
In Section \ref{sec:commutation:relations}, we invert the commutation relations provided by \cite[A]{BruhatTits2} in the quasi-split case from which we deduce, in Section \ref{sec:explicit:computation:with:commutation:relations}, a list of unipotent elements contained in $[P,P]$.

From these unipotent elements and from semisimple elements obtained by the rank $1$ case, we obtain, in Section \ref{sec:frattini:subgroup}, a generating set and a description of the Frattini subgroup as a directly generated product.
In Section \ref{sec:counting:alcoves}, we go a bit further than Bruhat-Tits in the study of quotient subgroups of filtered root groups.
From this, we can compute the finite quotient $P / \mathrm{Frat}(P)$ and deduce, in Section \ref{sec:minimal:number}, a minimal generating set of $P$.
The minimal number of elements of such a family is stated in Corollary \ref{cor:minimal:number:topological:generators}.

We summarize this in the following graph:

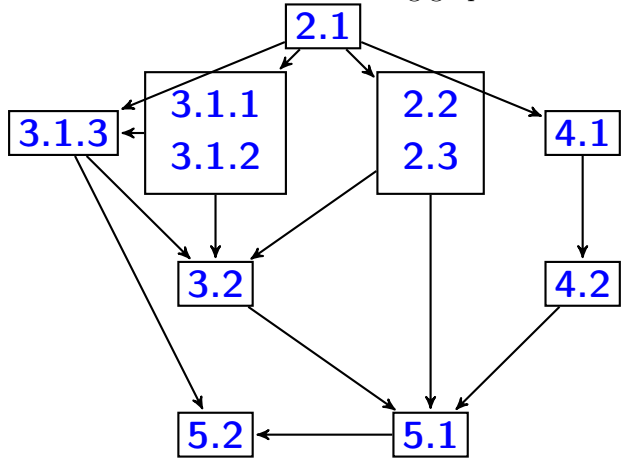
\begin{tikzpicture}[->,>=stealth',shorten >=1pt,auto,node distance=2cm,
  thick,main node/.style={rectangle,fill=white!100,draw,font=\sffamily\Large\bfseries}]

  \node[main node] (1) {\ref{sec:valued:root:groups:datum}};
  \node[main node] (3) [below left of=1] {\begin{tabular}{c}\ref{sec:walls} \\ \ref{sec:description:alcove}\end{tabular}};
  \node[main node] (4) [below right of=1] {\begin{tabular}{c}\ref{sec:reduced:case} \\ \ref{sec:non:reduced:case}\end{tabular}};
  \node[main node] (2) [left of=3] {\ref{sec:counting:alcoves}};
  \node[main node] (5) [right of=4] {\ref{sec:commutation:relations}};
  \node[main node] (6) [below of=3] {\ref{sec:action:on:ball}};
  \node[main node] (7) [below of=5] {\ref{sec:explicit:computation:with:commutation:relations}};
  \node (0) [below of=4] {};
  \node[main node] (8) [below of=0] {\ref{sec:frattini:subgroup}};
  \node[main node] (9) [below of=6] {\ref{sec:minimal:number}};

  \path[every node/.style={font=\sffamily\small}]
	(1)
		edge (2)
		edge (3)
		edge (4)
		edge (5)
	   
	(2)
		edge (6)
		edge (9)
	(3)
		edge (2)
		edge (6)
	(4)
		edge (6)
		edge (8)
	(5)
		edge (7)
	(6)
		edge (8)
	(7)
		edge (8)
	(8)
		edge (9);
	
\end{tikzpicture}

\section{Rank \texorpdfstring{$1$}{1} subgroups inside a valued root group datum} 
\label{sec:rank:one:case}

We keep notations of Section \ref{sec:intro:structure}.
In particular, we always denote by $K$ a field and by $G$ a semisimple $K$-group.
From Section \ref{sec:valuation:of:root:groups:datum}, we will assume that $K$ is a non-Archimedean local field, and we will assume that $G$ is simply-connected, almost-$K$-simple.
In order to compute the Frattini subgroup of a maximal pro-$p$ subgroup of $G(K)$, we adopt the point of view of valued root groups datum. In Section \ref{sec:valued:root:groups:datum}, we recall how we define a valuation on root groups, and how these groups can be parametrized. Thanks to these parametrizations, given in Section \ref{sec:parametrization}, we compute explicitly, in Sections \ref{sec:reduced:case} and \ref{sec:non:reduced:case}, the various possible commutators, and the $p$-powers of elements in a rank $1$ subgroup corresponding to a given root.
The rank $1$ case is not only useful to define filtrations of root groups, but also useful to compute elements in the Frattini subgroup corresponding to elements of the maximal torus $T$.
There are exactly two root systems, up to isomorphism, whose types are named $A_1$ and $BC_1$, corresponding to groups $\mathrm{SL}_{2}$ (Section \ref{sec:reduced:case}) and $\mathrm{SU}(h) \subset \mathrm{SL}_3$ (Section \ref{sec:non:reduced:case}) respectively.

We denote by $T(K)_b$ the maximal bounded subgroup of $T(K)$, defined in \cite[4.4.1]{BruhatTits2}.
We denote by $T(K)_b^+$ the (unique) maximal pro-$p$ subgroup of $T(K)_b$.

\subsection{Valued root groups datum}
\label{sec:valued:root:groups:datum}

We want to describe precisely the derived group of a maximal pro-$p$ subgroup.
We do it in combinatorial terms, thanks to a filtration of root groups.
Because we have to deal with non-reduced root systems, we recall the following definitions:
\begin{Def}
\label{def:multipliable:divisible:root}
Let $\Phi$ be a root system.
A root $a \in \Phi$ is said to be \textbf{multipliable} if $2a \in \Phi$; otherwise, it is said to be \textbf{non-multipliable}.
A root $a \in \Phi$ is said to be \textbf{divisible} if $\frac{1}{2}a \in \Phi$; otherwise, it is said to be \textbf{non-divisible}.
\end{Def}

The set of non-divisible roots, denoted by $\Phi_{\mathrm{nd}}$, is a root system; the set of non-multipliable roots, denoted by $\Phi_{\mathrm{nm}}$, is a root system.

\subsubsection{Root groups datum}
\label{sec:root:groups:datum}

For each root $a \in \Phi$, there is a unique unipotent subgroup $U_a$ of $G$ whose Lie algebra is a weight subspace with respect to $a$.
In order to define an action of $G(K)$ on a spherical building with suitable properties, it suffices to have suitable relations of the various root groups $U_a(K)$.
These required relations are the axioms given in the definition of a root groups datum. More precisely:

\begin{Def}
\label{def:root:groups:datum} \cite[6.1.1]{BruhatTits1}
Let $G$ be an abstract group and $\Phi$ be a root system.
A \textbf{root groups datum} of $G$ of type $\Phi$ is a system
$(T, (U_a, M_a)_{a \in \Phi})$ satisfying the following axioms:
\begin{description}
\item[(RGD 1)] $T$ is a subgroup of $G$ and, for any $a \in \Phi$, the set $U_a$ is a non-trivial subgroup of $G$, called the root group of $G$ associated to $a$.
\item[(RGD 2)] For any $a,b \in \Phi$, the group of commutators $[U_a, U_b]$ is contained in the group generated by the groups $U_{ra+sb}$ where $r,s \in \mathbb{N}^*$ and $ra+sb \in \Phi$.
\item[(RGD 3)] If $a$ is a multipliable root, we have $U_{2a} \subset U_a$ and $U_{2a} \neq U_a$.
\item[(RGD 4)] For any $a \in \Phi$, the set $M_a$ is a right coset of $T$ in $G$ and we have $U_{-a} \setminus \{ 1 \} \subset U_a M_a U_a$.
\item[(RGD 5)] For any $a,b \in \Phi$ and $n \in M_a$, we have $n U_b n^{-1} = U_{r_a(b)}$ where $r_a \in W(\Phi)$ is the orthogonal reflection with respect to $a^\perp$ and $W(\Phi)$ is the Weyl group of $\Phi$.
\item[(RGD 6)] We have $T U_{\Phi^+} \cap U_{\Phi^-} = \{1\}$ where $\Phi^+$ is an order of the root system $\Phi$ and $\Phi^- = - \Phi^+ = \Phi \setminus \Phi^+$.
\end{description}

A root groups datum is said to be \textbf{generating} if the groups $U_a$ and $T$ generate $G$.
\end{Def}

Now, given a reductive group $G$ over a field $K$, with a relative root system $\Phi$, we provide a root groups datum of $G(K)$.
Let $a \in \Phi$. By \cite[14.5 and 21.9]{Borel}, there exists a unique closed $K$-subgroup of $G$, denoted by $U_a$, which is connected, unipotent, normalized by $T$ and whose Lie algebra is $\mathfrak{g}_a + \mathfrak{g}_{2a}$.
This group $U_a$ is called the \textbf{root group} of $G$ associated to $a$.
By \cite[4.1.19]{BruhatTits2}, there exists cosets $M_a$ such that $\Big(T(K),\big(U_a(K),M_a\big)_{a\in\Phi}\Big)$ is a generating root groups datum of $G(K)$ of type $\Phi$.

\subsubsection{The \texorpdfstring{$*$}{*}-action on the absolute root system and splitting extension fields of root groups}
\label{sec:star:action}

From now on, $G$ is a quasi-split semisimple group. As in Section \ref{sec:intro:structure}, we denote by $\widetilde{K}$ the minimal splitting field of $G$ over $K$ (uniquely defined in a given separable closure $K_s$ of $K$).

In a general context, there is a canonical action of the absolute Galois group $\Sigma = \mathrm{Gal}(K_s / K)$ on the algebraic group $G$.
When $G$ is quasi-split, we can choose a maximal $K$-split torus $S$ and we get a maximal torus $T =\mathcal{Z}_G(S)$ of $G$ defined over $K$.
Thus, we define an action of $\Sigma$ on $X^*(T_{K_s})$ by:
$$\forall \sigma \in \Sigma,\ \forall \chi \in X^*(T_{K_s}),\ \sigma \cdot \chi = t \mapsto \sigma\Big( \chi\big( \sigma^{-1}(t)\big)\Big)$$
In the same way, thanks to conjugacy of minimal parabolic subgroups (which are Borel subgroups when $G$ is quasi-split), we define an action of $\Sigma$ on the type of parabolic subgroups, from which we deduce an action on the (simple) absolute roots.

\begin{Not}[The $*$-action on the absolute root system]
\label{not:star:action}
This is a summary of \cite[§6]{BorelTits} for a quasi-split group $G$.
In particular, there exists a Borel subgroup $B$ of $G$ defined over $K$.
Denote by $\widetilde{\Delta}$ the set of absolute simple roots and by $\mathrm{Dyn}(\widetilde{\Delta})$ its associated Dynkin diagram.
There exists an action of the Galois group $\Sigma = \mathrm{Gal}(\widetilde{K} / K)$ on $\mathrm{Dyn}(\widetilde{\Delta})$ which preserves the diagram structure.
This action is called the $*$-action and it can be extended, by linearity, to an action of $\Sigma$ on $\widetilde{V}^* = X^*(T_{\widetilde{K}}) \otimes_\mathbb{Z} \mathbb{R}$, and on $\widetilde{\Phi}$.
The restriction morphism $j = \iota^* : X^*(T) \rightarrow X^*(S)$, where $\iota : S \subset T$ is the inclusion morphism, can be extended to an endomorphism of the Euclidean space $\rho : \widetilde{V}^* \rightarrow\ \widetilde{V}^*$.
This morphism $\rho$ is the orthogonal projection onto the subspace $V^*$ of fixed points by the action of $\Sigma$ on $\widetilde{V}^*$.
From a geometric realization of $\widetilde{\Phi}$ in the Euclidean space $\widetilde{V}^*$, we deduce a geometric realization of $\Phi = \rho(\widetilde{\Phi})$ in $V^*$.
The orbits of the action of $\Sigma$ on $\widetilde{\Phi}$ are the fibers of the map $\rho : \widetilde{\Phi} \rightarrow \Phi$.
\end{Not}

\begin{Not}[Some field extensions]
\label{not:L:prime}
According to \cite[4.1.2]{BruhatTits2}, by definition of $\widetilde{K}$ as minimal splitting extension, the $*$-action of $\Sigma = \mathrm{Gal}(\widetilde{K} / K)$ on $\mathrm{Dyn}(\widetilde{\Delta})$ is faithful.
Assume that $G$ is almost-$K$-simple, so that the relative root system $\Phi$ is irreducible.
Consider a connected component of $\mathrm{Dyn}(\widetilde{\Delta})$.
Denote by $\Sigma_0$ its pointwise stabilizer in $\Sigma$.
Denote by $\Sigma_d$ its setwise stabilizer, where $d \in \mathbb{N}^*$ is defined by $d=[\Sigma_d : \Sigma_0]$.
We denote $L_d = \widetilde{K}^{\Sigma_d}$ and $L_0 = \widetilde{K}^{\Sigma_0}$, so that $L_0 / L_d$ is a Galois extension of degree $d$.
Because of the classification of root systems, the index $d$ is an element of $\{1,2,3,6\}$.

If $d = 2$, we let $L' = L_0$; we fix $\tau \in \mathrm{Gal}(L_0 / L_d)$ to be the non-trivial element.

If $d \geq 3$, we let $L'$ be a separable sub-extension of $L_0$ (possibly non-Galois) of degree $3$ over $L_d$;
we fix $\tau \in \mathrm{Gal}(L_0 / L_d)$ to be an element of order $3$.

Thus, we denote $d' = [L':L_d] \in \{1,2,3\}$.
In practice, $d' = \min(d,3)$.
\end{Not}

\begin{Rq}
\label{rq:reduction:to:absolutely:simple:groups}
According to \cite[6.21]{BorelTits}, we can write $G = R_{L_d / K}(G')$ where $G'$ is an absolutely simple $L_d$-group.
Hence $G(K) \simeq G'(L_d)$.
Because, in this paper, we prove some results on rational points, we could assume that $G$ is absolutely simple.
Under this assumption, the root system $\widetilde{\Phi}$ is irreducible;
$\widetilde{K} = L_0$ and $L_d = K$.
Despite this, we will only assume that $G$ is $K$-simple in order to have more intrinsic statements.
\end{Rq}

\begin{Def}
\label{def:splitting:field}

Let $\alpha \in \widetilde{\Phi}$ be an absolute root.
Denote by $\Sigma_\alpha$ be the stabilizer of $\alpha$ for the $*$-action.
The \textbf{field of definition} of the root $\alpha$ is the subfield of $\widetilde{K}$ fixed by $\Sigma_\alpha$, denoted by $L_{\alpha} = \widetilde{K}^{\Sigma_\alpha}$.

Let $a = \alpha|_S$.
The \textbf{splitting field extension class} of $a$ is the isomorphism class of the field extension $L_\alpha / K$, denoted by $L_a/K$.
\end{Def}

\begin{proof}[Proof that this definition makes sense]
\label{rq:splitting:field:defined:up:to:isomorphism}
We know, by \cite[§6]{BorelTits}, that the set $\{\alpha \in \widetilde{\Phi},\ a = \alpha|_S \}$ is a non-empty orbit of the $*$-action on $\widetilde{\Phi}$.
Hence, by abuse of notation, we denote $a = \{\alpha \in \widetilde{\Phi},\ a = \alpha|_S \}$.
Thus, given any relative root $a \in \Phi$, the field extension class $L_{\alpha}/K$ does not depend of the choice of $\alpha \in a$.
\end{proof}

\begin{Rq}
\label{rq:splitting:field:multipliable:root}
If $a \in \Phi$ is a multipliable root,
then there exists $\alpha, \alpha' \in a$ such that $\alpha + \alpha' \in \widetilde{\Phi}$.
Because $a$ is an orbit, we can write $\alpha' = \sigma(\alpha)$ where $\sigma \in \Sigma$.
As a consequence, the extension of fields $L_{\alpha} / L_{\alpha+\alpha'}$ is quadratic.
By abuse of notation, we denote this extension class by $L_a / L_{2a}$;
the ramification of this extension will be considered later.
\end{Rq}

\subsubsection{Parametrization of root groups}
\label{sec:parametrization}

In order to value the root groups (we do it in Section \ref{sec:valuation:of:root:groups:datum}) thanks to the valuation of the local field, we have to define a parametrization of each root group.
Moreover, these valuations have to be compatible.
That is why we furthermore have to get relations between the parametrizations.

Let $(\widetilde{x}_{\alpha})_{\alpha \in \widetilde{\Phi}}$ be a Chevalley-Steinberg system of $G_{\widetilde{K}}$.
This is a parametrization of the absolute root groups $\widetilde{x}_{\alpha} : U_a \rightarrow \mathbb{G}_a$ satisfying some compatibility relations, that will be exploited to get commutation relations in Section \ref{sec:commutation:relations}.
We recall the precise definition and that such a system exists in Section \ref{sec:commutation:relations}).

Let $a\in \Phi$ be a relative root.
To compute commutators between elements of opposite root groups, or between elements of a torus and of a root group,
it is sufficient to compute inside the simply-connected semisimple $K$-group $\langle U_{-a} , U_a \rangle$ generated by the two opposite root groups $U_{-a}$ and $U_a$.
Let $\pi : G^a \rightarrow \langle U_{-a} , U_a \rangle$ be the universal covering of the quasi-split semisimple $K$-subgroup of relative rank $1$ generated by $U_a$ and $U_{-a}$.
The group $G^a$ splits over $L_a$ (this explains the terminology of \gts{splitting field} of a root).
A parametrization of the simply-connected group $G^a$ is given by \cite[4.1.1 to 4.1.9]{BruhatTits2}.
We now recall notations and the matrix realization that we will use later.

\paragraph{The non-multipliable case:}

Let $a\in\Phi$ be a relative root such that $2a \not\in \Phi$.
By \cite[4.1.4]{BruhatTits2}, the rank $1$ group $G^a$ is isomorphic to $R_{L_\alpha / K}(\mathrm{SL}_{2,L_\alpha})$.
It can be written as $G^a = R_{L_\alpha / K}(\widetilde{G}^{\alpha})$ with an isomorphism $\xi_\alpha : \mathrm{SL}_{2,L_\alpha} \stackrel{\simeq}{\longrightarrow} \widetilde{G}^\alpha$.

Inside the classical group $\mathrm{SL}_{2,L_\alpha}$, a maximal $L_\alpha$-split torus of $\mathrm{SL}_{2,L_\alpha}$ can be parametrized by the following homomorphism:
$$\begin{array}{cccc}z :&\mathbb{G}_{m,L_\alpha} & \rightarrow & \mathrm{SL}_{2,L_\alpha}\\&t&\mapsto&\begin{pmatrix}t&0\\0&t^{-1}\end{pmatrix}\end{array}$$
The corresponding root groups can be parametrized by the following homomorphisms:
$$\begin{array}{cccc}y_- :&\mathbb{G}_{a,L_\alpha} & \rightarrow & \mathrm{SL}_{2,L_\alpha}\\&v&\mapsto&\begin{pmatrix}1&0\\-v&1\end{pmatrix}\end{array} \text{ and }
\begin{array}{cccc}y_+ :&\mathbb{G}_{a,L_\alpha} & \rightarrow & \mathrm{SL}_{2,L_\alpha}\\&u&\mapsto&\begin{pmatrix}1&u\\0&1\end{pmatrix}\end{array}$$
According to \cite[4.1.5]{BruhatTits2}, there exists a unique $L_\alpha$-group homomorphism, denoted by $\xi_\alpha : \mathrm{SL}_{2,L_\alpha} \rightarrow \widetilde{G}^\alpha$, satisfying $\widetilde{x}_{\pm \alpha} = \pi \circ \xi_\alpha \circ y_{\pm}$.

Thus, we define a $K$-homomorphism $x_a = \pi \circ R_{L_\alpha / K}(\xi_\alpha)$ which is a $K$-group isomorphism between $R_{L_a/K}(\mathbb{G}_{a,L_a})$ and $U_a$.
We also define the following $K$-group isomorphism:
$$\widetilde{a} = \pi \circ R_{L_\alpha / K}(\xi_\alpha \circ z) :
 R_{L_\alpha / K}(\mathbb{G}_{m,L_\alpha}) \rightarrow T^a$$
 where $T^a = T \cap G^a$.

\paragraph{The multipliable case:}

Let $a\in\Phi$ be a relative root such that $2a \in \Phi$.
Let $\alpha \in a$ be an absolute root from which $a$ arises, and let $\tau \in \Sigma$ be an element of the Galois group such that $\alpha + \tau(\alpha)$ is again an absolute root.
To simplify notations, we let (up to compatible isomorphisms in $\Sigma$) $L = L_a = L_\alpha$ and $L_2 = L_{2a} = L_{\alpha + \tau(\alpha)}$.
By \cite[4.1.4]{BruhatTits2}, the $K$-group $G^a$ is isomorphic to $R_{L_2 / K}(\mathrm{SU}(h))$, where $h$ denotes the hermitian form on $L\times L \times L$ given by the formula:
$$h : (x_{-1},x_0,x_1) \mapsto \sum_{i=-1}^{1} x_i {^\tau}x_{-i}$$
The group $G^a_{L_2}$ can be written as $G^a_{L_2} = \prod_{\sigma \in \mathrm{Gal}(L_2 / K)} \widetilde{G}^{\sigma(\alpha),\sigma(\tau(\alpha))}$ where each $\widetilde{G}^{\sigma(\alpha),\sigma(\tau(\alpha))}$ denotes a simple factor isomorphic to $\mathrm{SU}(h)$, so that $\mathrm{SU}(h)_{L} \simeq \mathrm{SL}_{3,L}$.

We define a connected unipotent $L_{2}$-group scheme by providing the $L_{2}$-subvariety
$H_0(L,L_{2}) = \left\{ (u,v) \in L \times L,\ u {^\tau}u = v + {^\tau}v \right\}$ of $L_a \times L_a$ with the following group law:
$$(u,v),(u',v') \mapsto (u+u',v+v'+u {^\tau}u')$$
Then, we let $H(L,L_{2}) = R_{L_{2}/K}(H_0(L,L_{2}))$.
For the rational points, we get $H(L,L_{2})(K) = \{ (u,v) \in L \times L,\ u {^\tau}u = v + {^\tau}v \}$ and the group law is given by $x_a(u,v) x_a(u',v') = x_a(u+u',v+v'+u {^\tau}u')$.

\begin{Not}
\label{not:trace:subspace}
For any multipliable root $a \in \Phi$, in \cite[4.2.20]{BruhatTits2} are furthermore defined the following notations:
\begin{itemize}
\item $L^0 = \{ y \in L,\ y+{^\tau}y = 0 \}$, this is an $L_{2}$-vector space of dimension $1$;
\item $L^1 = \{ y \in L,\ y+{^\tau}y = 1 \}$, this is an $L^0$-affine space.
\end{itemize}
\end{Not}

Indeed, if $K$ is not of characteristic $2$, then $L^0 = \ker ( \tau + \mathrm{id})$ is of dimension $1$ because $L_{2} = \ker ( \tau - \mathrm{id})$ is of dimension $1$ and $\pm 1$ are the eigenvalues of $\tau \in \mathrm{GL}(L_a)$.
Moreover, the affine space $L^1$ is non-empty because it contains $\frac{1}{2}$.
If $K$ is of characteristic $2$, then $L^0 = \ker ( \tau + \mathrm{id}) = L_{2}$.

\begin{Rq}[Interest of such notations]
For any $\lambda \in L^0$ so that $\lambda \neq 0$, we have an isomorphism of abelian groups given by the relation $\begin{array}{ccc}L_{2} & \rightarrow & L^0 \\ y & \mapsto & \lambda y\end{array}$, so that $x_a(0,\lambda y) = x_{2a}(y)$.
This constitutes an additional uncertainty when we want to perform computations in $G(K)$.
Because of valuation considerations, we will have to choose a $\lambda$ whose valuation is equal to zero; in fact, this is always possible.
To avoid confusion, it is better to work with the isomorphism of abelian groups $\begin{array}{ccc} L_a^0 & \rightarrow & U_{2a}(K) \\ y & \mapsto & x_a(0,y)\end{array}$ in order to realize this group as a subgroup of $U_a(K)$.

The affine space $L^1$ has an interest in the context of a valued field.
In particular, as soon as we will know that $L^1$ is non-empty, we will write $L = L_{2} \lambda \oplus L^0$ with a suitable $\lambda \in L^1$.
\end{Rq}

We parametrize a maximal torus of $\mathrm{SU}(h)$ by the isomorphism
$$\begin{array}{cccc}z :&R_{L/L_2}(\mathbb{G}_{m,L}) & \rightarrow & \mathrm{SU}(h)\\&t&\mapsto&\begin{pmatrix}t&0&0\\0&t^{-1}{^\tau}t&0\\0&0&{^\tau}t^{-1}\end{pmatrix}\end{array}$$

We parametrize the corresponding root groups of $\mathrm{SU}(h)$ by the homomorphisms:
$$\begin{array}{cccc}y_- : &H_0(L,L_2) & \rightarrow & \mathrm{SU}(h)\\&(u,v)&\mapsto&\begin{pmatrix}1&0&0\\u&1&0\\-v&-{^\tau}u&1\end{pmatrix}\end{array}$$
and
$$
\begin{array}{cccc}y_+ :&H_0(L,L_2) & \rightarrow & \mathrm{SU}(h)\\&(u,v)&\mapsto&\begin{pmatrix}1&-{^\tau}u&-v\\0&1&u\\0&0&1\end{pmatrix}\end{array}$$
By \cite[4.1.9]{BruhatTits2}, there exists a unique $L_2$-group isomorphism, denoted by $\xi_\alpha : \mathrm{SU}(h) \rightarrow \widetilde{G}^{\alpha,\tau(\alpha)}$, satisfying $\widetilde{x}_{\pm \alpha} = \pi \circ \xi_\alpha \circ y_{\pm}$.
From this, we define a $K$-homomorphism $x_a = \pi \circ R_{L_2 / K}(\xi_\alpha)$ which is a $K$-group isomorphism between the $K$-group $H(L,L_{2})$ and the root group $U_a$.
We also define the following $K$-group isomorphism:
$$\widetilde{a} = \pi \circ R_{L_2 / K}(\xi_\alpha \circ z) :
 R_{L_\alpha / K}(\mathbb{G}_{m,L_\alpha}) \rightarrow T^a$$
 where $T^a = T \cap G^a$.

\subsubsection{Valuation of a root groups datum} 
\label{sec:valuation:of:root:groups:datum}

For each root group, we now use its parametrization to define an exhaustion by subgroups.
In order to define an action of $G(K)$ on an affine building with suitable properties, it suffices to have suitable relations between the terms of filtration of root groups.
More precisely:

\begin{Def}
\label{def:valued:root:groups:datum} \cite[6.2.1]{BruhatTits1}
Let $G$ be an abstract group, let $\Phi$ be a root system and let $\Big(T,\big(U_a,M_a\big)_{a\in\Phi}\Big)$ be a root groups datum of $G$ of type $\Phi$.
A \textbf{valued root groups datum} is a system $\Big(T,\big(U_a,M_a,\varphi_a\big)_{a\in\Phi}\Big)$, where each $\varphi_a$ is a map from $U_a$ to $\mathbb{R} \cup \{ \infty \}$, satisfying the following axioms:
\begin{description}
\item[(VRGD 0)] for any $a \in \Phi$, the image of $\varphi_a$ contains at least $3$ elements;
\item[(VRGD 1)] for any $a \in \Phi$ and any $l \in \mathbb{R} \cup \{ \infty \}$, the set $U_{a,l} = \varphi_a^{-1}([l;\infty])$ is a subgroup of $U_a$ and the group $U_{a,\infty}$ is $\{ \mathbf{1} \}$;
\item[(VRGD 2)] for any $a \in \Phi$ and any $m \in M_a$, the map $u \mapsto \varphi_{-a}(u) - \varphi_a(mum^{-1})$ is constant over $U_{-a} \setminus \{ \mathbf{1} \}$;
\item[(VRGD 3)] for any $a,b \in \Phi$ such that $b \not\in -\mathbb{R}_+ a$ and any $l,l' \in \mathbb{R}$, the group of commutators $[U_{a,l},U_{b,l'}]$ is contained is the group generated by the groups $U_{ra+sb,rl+sl'}$ where $r,s \in \mathbb{N}^*$ and $ra+sb \in \Phi$;
\item[(VRGD 4)] for any multipliable root $a \in \Phi$, the map $\varphi_{2a}$ is the restriction of the map $2 \varphi_a$ to the subgroup $U_{2a}$;
\item[(VRGD 5)] for any $a \in \Phi$, any $u \in U_a$ and any $u',u'' \in U_{-a}$ such that $u'uu'' \in M_a$, we have $\varphi_{-a}(u') = - \varphi_a(u)$.
\end{description}
\end{Def}

Now, given a reductive group $G$ over a non-Archimedean local field $K$, with a relative root system $\Phi$, we provide a valued root groups datum of $G(K)$.
We define a filtration $(\varphi_a)_{a \in \Phi}$ of the rational points $U_a(K)$ of each root group by:
\begin{itemize}
\item $\varphi_a(x_a(y)) = \omega(y)$ if $a$ is a non-multipliable and non-divisible root, and if $y \in L_a$;
\item $\varphi_{a}(x_a(y,y')) = \frac{1}{2} \omega(y')$ if $a$ is a multipliable root and if $(y,y') \in H(L_a,L_{2a})$;
\item $\varphi_{2a}(x_a(0,y')) = \omega(y')$ if $a$ is a multipliable root and if $y' \in L_a^0$.
\end{itemize}
By \cite[§4.2]{BruhatTits2}, the family $\Big(T,\big(U_a(K),M_a,\varphi_a\big)_{a\in\Phi}\Big)$ is a valued root groups datum.

\subsubsection{Set of values}
\label{sec:set:of:values}

If $L / K$ is a finite extension of local fields, the valuation $\omega$ over $K^\times$ can be extended uniquely to a valuation over $L^\times$, still denoted by $\omega$ because of its uniqueness.
We let $\Gamma_L = \omega(L^\times)$.

Because we considered a discrete valuation $\omega$, the terms of filtration indexed by $\mathbb{R}$ can, in fact, be indexed by discrete subsets.
These subsets will be used in Section \ref{sec:description:apartment}, to provide an \gts{affinisation} of the spherical root system.

Let $a \in \Phi$ be a root. We define the following sets of values:
\begin{itemize}
\item $\Gamma_a = \varphi_a(U_a(K) \setminus \{\mathbf{1}\})$;
\item $\Gamma'_a = \{\varphi_a(u),\ u \in U_a(K) \setminus\{\mathbf{1}\} \text{ and }\varphi_a(u) = \sup \varphi_a(u U_{2a}(K)) \}$.
\end{itemize}
Furthermore, for any value $l \in \mathbb{R}$, we denote $l^+ = \min \{l' \in \Gamma_a,\ l'> l \}$.
This is the lowest value, greater than $l$, for which we detect a change in the valued root groups $(U_{a,l'})_{l' > l}$.
In order to characterize $\Gamma'_a$, we complete the notations of \ref{not:trace:subspace} introducing the following $L^{1}_{a,\max} = \{z \in L_a^1,\ \omega(z) = \sup \{\omega(y),\ y\in L_a^1 \} \}$.
It is the subset of $L^1_a$ whose elements reach the maximum of the valuation.

\begin{Rq}
Be careful that the value $l^+$ also depends on $a$.

The sense of $\Gamma'_a$ will be provided by Lemma \ref{lem:equivalence:affine:root:systems}, as the set of values parametrizing the affine roots.
\end{Rq}

\begin{Lem}
\label{lem:sets:of:values:non:multipliable:root}
If $a$ is a non-multipliable non-divisible root, then we have $\Gamma_{a} = \Gamma'_a = \Gamma_{L_a}$.
\end{Lem}

\begin{proof}
This is obvious by the isomorphism between $U_a(K)$ and $L_a$.
\end{proof}

Now, we assume that $a \in \Phi$ is a multipliable root.

Let $p$ be the residue characteristic of $K$.
Even if the sets of values can be computed for any $p$, we assume here that $p \neq 2$.
This assumption provides a short cut in the computation of sets of values (mostly because $\frac{1}{2} \in L^1_{a,\max}$ in this case), and will be necessary later for more algebraic reasons.

Since $\omega$ is a discrete valuation and since for any $y \in L^1_a$, we have $\omega(y) \leq 0$, it is clear that $L^1_{a,\max}$ is non-empty.
Moreover, when $p \neq 2$, we have $\frac{1}{2} \in L^1_{a,\max}$.
Hence, by \cite[4.2.21 (4)]{BruhatTits2}, we know that $\Gamma_a = \frac{1}{2} \Gamma_{L_a}$ and that $\Gamma'_a = \Gamma_{L_a}$.

By \cite[4.3.4]{BruhatTits2}, we know that:
\begin{itemize}
\item when the quadratic extension $L_a / L_{2a}$ is unramified, we have the equalities $\Gamma_{2a} = \Gamma'_{2a} = \omega({L^0}^\times) = \Gamma_{L_a} = \Gamma_{L_{2a}}$;
 \item when the quadratic extension $L_a / L_{2a}$ is ramified, we have the equalities  $\Gamma_{2a} = \Gamma'_{2a} = \omega({L^0}^\times) = \omega(\varpi_{L_a}) + \Gamma_{L_{2a}}$.
 \end{itemize}

\begin{Lem}[Summary]
\label{lem:sets:of:values:multipliable:root}
Let $a \in \Phi$ be a multipliable root.
If we normalize the valuation $\omega$ so that $\Gamma_{L_a} = \mathbb{Z}$, then we get:

\begin{tabular}{|c|c|c|}
\hline
$L_a/L_{2a}$ & unramified & ramified\\
\hline
$\Gamma_{L_a}$ & $\mathbb{Z}$ & $\mathbb{Z}$ \\
\hline
$\Gamma_{L_{2a}}$ & $\mathbb{Z}$ & $2 \mathbb{Z}$ \\
\hline
$\Gamma_a$ & $\frac{1}{2} \mathbb{Z}$ & $\frac{1}{2} \mathbb{Z}$ \\
\hline
$\Gamma_{2a}$ & $\mathbb{Z}$ & $1 + 2 \mathbb{Z}$\\
\hline
$\Gamma'_{a}$ & $\mathbb{Z}$ & $\mathbb{Z}$\\
\hline
\end{tabular}
\end{Lem}

\begin{Rq}
The case of a divisible root has been treated. It is the case $2a$ of a multipliable root $a$.
\end{Rq}

\begin{Rq}[The case of residue characteristic $2$]
When the residue characteristic is any prime number (and in particular if $p = 2$), it can be seen via further investigations, that the set $L^1_{a,\max}$ is non-empty and we let $\{ \delta \} = \omega(L^1_{a,\max})$.
We can compute the sets of values, depending on $\delta$ and on the ramification of $L_a / L_{2a}$.
We get the following results:
\begin{itemize}
\item $\Gamma'_a = \frac{1}{2} \delta + \Gamma_{L_a}$;
\item $\Gamma_a = \Gamma'_a \cup \frac{1}{2} \Gamma_{2a} = \frac{1}{2} \Gamma_{L_a}$;
\item if $L_a / L_{2a}$ is ramified, then $\Gamma'_a \cap \frac{1}{2} \Gamma_{2a} = \emptyset$ and $\Gamma_{2a} = \delta + \omega(\varpi_{L_a}) + \Gamma_{L_{2a}}$;
\item if $L_a / L_{2a}$ is unramified, then $\Gamma'_a \cap \frac{1}{2} \Gamma_{2a} \neq \emptyset$ and $\Gamma_{2a} = \Gamma_{L_{2a}} = \Gamma_{L_{a}}$.

Because $\delta = 0$ when $p \neq 2$, this is, in fact, the generalisation to any residue characteristic.
\end{itemize}

\end{Rq}

\subsection{The reduced case}
\label{sec:reduced:case}

Let $a\in \Phi$ be a non-multipliable root of $\Phi$ arising from an absolute root $\alpha \in \widetilde{\Phi}$.
In this section, in order to simplify notation, we denote $L = L_{\alpha} = L_a$.
Denote by $G^a = \langle U_{-a}, U_a \rangle$ the $K$-subgroup of $G$ generated by $U_{-a}$ and $U_a$.
The universal covering $\pi : R_{L/K}(SL_{2,L}) \rightarrow G^a$ is a central $K$-isogeny, which allows us to compute relations between the elements of $U_a$, $U_{-a}$ and $T$ by the parametrizations $x_a$, $x_{-a}$ and $\widetilde{a}$ thanks to matrix realizations in $\mathrm{SL}_2$.

We denote by $T^a = T \cap G^a$ the maximal torus of $G^a$ and by $T^a(K)_b^+ = T(K)_b^+ \cap T^a(K)$ the maximal pro-$p$ subgroup of $T^a(K)$.
By \cite[3.2.10]{Loisel-maximaux} (because $G^a$ is simply-connected, the torus $T^a$ is an induced torus), we know that $\widetilde{a} : 1 + \mathfrak{m}_{L_a} \rightarrow T^a(K)_b^+$ is a group isomorphism.

\begin{Lem}[{Commutation relation $[T,U_a]$ in the reduced case}]~
\label{lem:commutator:relation:reduced:case:torus:unipotent}

(1) Let $t \in T(K)$.
Then, for any $x \in L_\alpha$, we have
$$\Big[ x_a(x) , t \Big] = x_a\Big( \big( 1 - \alpha(t)\big) x \Big)$$

(2) Normalize the valuation $\omega$ by $\Gamma_a = \Gamma_{L_a} = \mathbb{Z}$.
For any $l \in \Gamma_a$, we have:
$$\left[T(K)_b^+,U_{a,l}\right] \leq U_{a,l+1}$$
and this is an equality if $p \neq 2$.
\end{Lem}

\begin{proof}
(1) By definitions, $t x_a(x) t^{-1} = x_a\big( \alpha(t) x \big)$.
Hence $\big[x_a(x), t \big] = x_a\big(x\big) x_a\big( - \alpha(t) x\big) =x_a\Big( \big(1 - \alpha(t) \big) x \Big)$.

(2) Let $t \in T(K)_b^+$ and $u \in U_{a,l}$.
Write $u = x_a(x)$ with $x \in L_a$ such that $\omega(x) \geq l$.
Write $t = \widetilde{a}(1+z)$ with $z \in \mathfrak{m}_{L_\alpha}$ so that $\alpha(t) = (1 + z)^2$.
In particular, $\omega\big(1 - \alpha(t)\big) \geq 1$.
Applying (1), we get
$\varphi_a\big([t,u]\big) = \omega\big( (1 - \alpha(t)) x \big)$.
Hence $\varphi_a\big([t,u]\big) \geq \omega(x) + 1 \geq l+1$.
This gives the inclusion $\left[T(K)_b^+,U_{a,l}\right] \subset U_{a,l+1}$.

Conversely, let $y \in L_\alpha$ be such that $\omega(y) \geq l+1$.
Let $\varpi$ be a uniformizer of $\mathcal{O}_{L_a}$.
Assume $p \neq 2$.
We have $\omega(2 \varpi + \varpi^2)=1$.
Set $t = \widetilde{a}(1 + \varpi)$ and $x = (2 \varpi + \varpi^2)^{-1} y$.
Then $\big[t,x_a(x)\big] = x_a(y)$ and $t \in T(K)_b^+$.
Hence $\omega(x) = \omega(y) - 1 \geq l$.
\end{proof}

\begin{Lem}[{Commutation relation $[U_{-a,l},U_{a,l'}]$ in the reduced case}]~\\
\label{lem:commutator:relation:reduced:case:opposite:unipotent}
Normalize $\omega$ by $\Gamma_a = \Gamma_{L_a} = \mathbb{Z}$.
Let $l,l' \in \Gamma_a = \mathbb{Z}$ such that $l+l' \geq 1$.
Then, for any $x,y \in L_\alpha$ such that $\omega(x) \geq l'$ and $\omega(y) \geq l$, we have:
$$\Big[ x_{-a}(y) , x_a(x) \Big] = x_{-a}\left( \frac{xy^2}{1+xy}\right) \widetilde{a}\left(1 + xy\right) x_a\left(\frac{-x^2y}{1+xy}\right)$$

In particular, $[U_{-a,l},U_{a,l'}] \subset U_{-a,l+1} T(K)_b^+ U_{a,l'+1}$.
\end{Lem}

\begin{proof}
We have $\omega(xy) = \omega(x) + \omega(y)> 0$, hence $xy \in \mathfrak{m}_{L_a}$.
Thus, $1 + xy \in \mathcal{O}_{L_a}^\times$ and in $\mathrm{SL}_2(L_a)$, we have:
$$\left[ \begin{pmatrix}1&0\\-y&1\end{pmatrix} , \begin{pmatrix}1&x\\0&1\end{pmatrix} \right] = 
\begin{pmatrix}1&0\\-\frac{xy^2}{1+xy}&1\end{pmatrix} 
\begin{pmatrix} 1 + xy & 0\\0& \frac{1}{1+xy}\end{pmatrix}
\begin{pmatrix} 1 & \frac{-x^2y}{1+xy}\\0&1\end{pmatrix}
$$
Applying $\pi$ to this equality, we get the desired equality.

We have $1 + xy \in 1 + \mathfrak{m}_{L_a}$,
hence $\widetilde{a}(1-xy) \in T(K)_b^+$.
Moreover,
$\omega\left( \frac{xy^2}{1+xy} \right) = \omega(x) + 2 \omega(y) \geq 1 + \omega(y)$ and $\omega\left( \frac{x^2y}{1+xy} \right) = 2 \omega(x) + \omega(y) \geq 1 + \omega(x)$.
Hence $x_{-a}\left( \frac{xy^2}{1+xy}\right) \in U_{-a, l+1}$ and $x_a\left(\frac{-x^2y}{1+xy}\right) \in U_{a,l'+1}$.
\end{proof}

\begin{Prop}
\label{prop:frattini:computation:reduced:case}
Assume that $p\neq 2$ and $\Gamma_a = \Gamma_{L_a} = \mathbb{Z}$.
Let $l \in \mathbb{Z} = \Gamma_a$.
Let $H$ be a compact open subgroup of $G^a(K)$ containing $U_{a,l}$, $T^a(K)_b^+$ and $U_{-a,-l+1}$.

Then the group $H^p [H,H]$ contains the subgroups $U_{a,l+1}$, $U_{-a, -l+2}$ and $T^a(K)_b^+$.

Moreover, in the case of equal characteristic $\mathrm{char}(K) = p$, we have the inclusion $H^p \subset [H,H]$.
\end{Prop}

\begin{proof}
Denote by $\varpi$ a uniformizer of $L_a$.
We firstly show that $T^a(K)_b^+$ is contained in $H^p[H,H]$.
For any $t \in 1 + \mathfrak{m}_{L_a},\ t \neq 1$ and any $u \in L_a$, one can check the following equalities inside $\mathrm{SL}_2$:

\begin{equation}\label{eqn1}
\left[
\begin{pmatrix} t & \frac{t u}{1-t^2}\\0&\frac{1}{t} \end{pmatrix} , 
\begin{pmatrix} 1 & 0 \\ -\frac{(1-t^2)^2}{t^2 u} & 1 \end{pmatrix} 
\right] =
\begin{pmatrix} t^2 & u \\ 0 & \frac{1}{t^2} \end{pmatrix}
\end{equation}

\begin{equation}\label{eqn2}
\left[
\begin{pmatrix} 1 & \frac{t^2-1}{t^2 v}\\0&1 \end{pmatrix} , 
\begin{pmatrix} \frac{1}{t} & 0 \\ -\frac{t v}{(t^2-1)} & t \end{pmatrix} 
\right] =
\begin{pmatrix} t^2 & 0 \\ v & \frac{1}{t^2} \end{pmatrix}
\end{equation}

We have $\omega(1+t) = \omega(2+s) = 0$ because $p \neq 2$.
Hence, for any $u \in \varpi^{l+1} \mathcal{O}_{L_a}$ and for any $t-1 = s \in \varpi \mathcal{O}_{L_a}$, we have the following:
$$\begin{array}{rcl}
\omega\left(\frac{t u}{1-t^2}\right) & = & \omega(t) + \omega(u) - \omega(1+t) - \omega(1-t) \\
& = & \omega(u) - \omega(s)\\
\omega\left(-\frac{(1-t^2)^2}{t^2 u}\right) & = & 2\omega(s) - \omega(u)\end{array}$$
Moreover, we have:
\begin{equation} \label{eqn3}
\begin{pmatrix}t^2&u\\0&t^{-2}\end{pmatrix}
\begin{pmatrix}t^2&-t^{-4}u\\0&t^{-2}\end{pmatrix} =
\begin{pmatrix}t^4&0\\0&t^{-4}\end{pmatrix}
\end{equation}

Let $t = 1 + s \in 1 + \varpi \mathcal{O}_L$.
Set $u = \varpi^{l+\omega(s)}$ so that $\omega\left(\frac{t u}{1-t^2}\right) \geq l$ and $\omega\left(-\frac{(1-t^2)^2}{t^2 u}\right) \geq -l + 1$.
Hence, $\pi \begin{pmatrix} t & \frac{t u}{1-t^2}\\0&\frac{1}{t} \end{pmatrix}  \in H$ and $\pi \begin{pmatrix} 1 & 0 \\ -\frac{(1-t^2)^2}{t^2 u} & 1 \end{pmatrix} \in H$.
Thus, according to the equation (\ref{eqn1}), we get $\pi \begin{pmatrix}t^2&u\\0&t^{-2}\end{pmatrix} \in [H,H]$ .
Similarly, substituting $u$ by $-t^4 u$, we get $\pi \begin{pmatrix}t^2&-t^{-4}u\\0&t^{-2}\end{pmatrix} \in [H,H]$.
As a consequence, for any $t \in 1 + \varpi \mathcal{O}_L$, we have $\widetilde{a}(t^4) \in [H,H]$ according to the equation (\ref{eqn3}).

Moreover, the elements $\widetilde{a}(t^p)$ where $t \in 1 + \mathfrak{m}_L$ are in $H^p$ because we assumed $H \supset T(K)_b^+$.
Since $4$ and $p$ are coprime, we have $\widetilde{a}(t) \in H^p [H,H]$.

In the case of equal characteristic $\mathrm{char}(K) = p >2$,
the group homomorphism $\left\{\begin{array}{ccc}
	1 + \mathfrak{m}_L & \rightarrow & 1 + \mathfrak{m}_L\\
	t						&	\mapsto		& t^2
\end{array}\right.$ is surjective.
Hence $\widetilde{a}(t)\in [H,H]$.

As a consequence, the elements:
$$x_a(u) = \widetilde{a}(t^{-2}) \cdot \left[\widetilde{a}(t) x_a\left(\frac{t^2 u}{1-t^2}\right),x_{-a}\left(-\frac{(1-t^2)^2}{t^4 u}\right)\right]$$
where $u \in \varpi^{l+1} \mathcal{O}_L$ and $t = 1 + \varpi^{\omega(u)}$, are in $H^p [H,H]$ (resp. in $[H,H]$ if $\mathrm{char}(K) = p$).
Hence, the group $H^p [H,H]$ (resp. $[H,H]$) contains $U_{a,l+1}$.

Similarly, it contains $U_{-a, (-l+1)+1} = U_{a,-l+2}$, using the equation (\ref{eqn2}) instead of (\ref{eqn1}).

It remains to prove that $H^p \subset [H,H]$ when $\mathrm{char}(K) = p > 2$.
Let $g \in H$ and write $g=x_{-a}(v) \widetilde{a}(t) x_a(u)$.
Consider the quotient morphism $\pi : H \rightarrow H / [H,H]$.
Then $\displaystyle \pi(g^p) = \pi(g)^p =
\Big(\pi\big(x_{-a}(v)\big) \pi\big(\widetilde{a}(t)\big) \pi\big(x_a(u)\big) \Big)^p$.
Since $H/[H,H]$ is commutative, we have $\pi(g^p) = \pi\big(x_{-a}(v)\big)^p \pi\big(\widetilde{a}(t)\big)^p \pi\big(x_a(u)\big)^p = \pi\big(x_{-a}(pv)\big) \pi\big(\widetilde{a}(t^p)\big) \pi\big(x_a(pu)\big) = \pi\big(\widetilde{a}(t^p)\big) = 1$ because we got $\widetilde{a}(t^p) \in [H,H]$.
Hence $g^p \in [H,H]$.
\end{proof}

\subsection{The non-reduced case}
\label{sec:non:reduced:case}

Let $a\in \Phi$ be a multipliable root of $\Phi$ arising from an absolute root $\alpha \in \widetilde{\Phi}$.

In this paragraph, we denote by $L = L_{\alpha} = L_a$ and $L_2 = L_{\alpha+{^{\tau}}\alpha} = L_{2a}$, where $\tau = \tau_a$ is the non trivial element of $\mathrm{Gal}(L/L_2)$.
In order to simplify notations, for any $x \in L$, we denote ${^\tau}x = \tau(x)$.
Denote by $h$ the $L_2$-Hermitian form:
$$\begin{array}{cccc}h : & L \times L \times L & \rightarrow & L \\ & (x_{-1},x_0,x_1) & \mapsto & \sum_{i=-1}^1 x_{-i} {^\tau}x_i\end{array}$$

Recall that the universal covering is a central $K$-isogeny $\pi : R_{L/k}(SU(h)) \rightarrow G^a$, from which we compute, inside $\mathrm{SU}(h)$, relations between elements of $U_a$, $U_{-a}$ and $T$ thanks to parametrizations $x_a$, $x_{-a}$ and $\widetilde{a}$.

Denote by $T^a = T \cap G^a$ and $T^a(K)_b^+ = T(K)_b^+ \cap T^a(K)$, so that $T^a(K)_b^+ = \widetilde{a}\left(1+\mathfrak{m}_{L_a}\right)$.
For any $l \in \mathbb{N}^*$, set $T^a(K)_b^l = \widetilde{a}\left(1+\mathfrak{m}_{L_a}^l\right)$.
Normalize $\omega$ by $\Gamma_a = \Gamma_{-a} = \frac{1}{2}\mathbb{Z}$, so that $\Gamma_L = \mathbb{Z}$ and $\Gamma_{L_2}=2\mathbb{Z}\text{ or }\mathbb{Z}$ depending on whether the extension $L/L_2$ is ramified or not.
The analogue to Proposition \ref{prop:frattini:computation:reduced:case}, in the non-reduced case, is the following:

\begin{Prop}
\label{prop:non:reduced:frattini:rank:one}
Assume that $p \geq 5$.
Let $l \in \Gamma_a = \frac{1}{2}\mathbb{Z}$.
Let $H$ be a compact open subgroup of $G(K)$ containing the following subgroups $T(K)_b^+$, $U_{-a,-l}$ and $U_{a,l+\frac{1}{2}}$.

If $L/L_2$ is not ramified, then there exists $l'' \in \mathbb{N}^*$ such that $H^p [H,H]$ contains the following subgroups $T^a(K)_b^{l''}$, $U_{-a,-l+1}$ and $U_{a,l+\frac{3}{2}}$.

If $L/L_2$ is ramified, then there exists $l'' \in \mathbb{N}^*$ such that $H^p [H,H]$ contains the following subgroups $T^a(K)_b^{l''}$, $U_{-a,-l+\frac{3}{2}}$ and $U_{a,l+2}$.

Precisely, up to exchanging $a$ with $-a$, we can assume $l \in \Gamma'_a = \mathbb{Z}$ and, in this case, we get $l'' = 3 + \varepsilon$
where
$$\varepsilon =
 \left\{\begin{array}{cl}
 1 & \text{ if } L/L_2 \text{ is ramified and } l \in 2 \mathbb{Z} + 1 = \Gamma_{L} \setminus \Gamma_{L_2}\\
 0 & \text{ otherwise}
 \end{array}\right.
$$

Moreover, when $\mathrm{char}(K) = p > 0$, we have $H^p \subset [H,H]$.
\end{Prop}

\begin{Rq}
Since the maximal pro-$p$ subgroups are pairwise conjugated by \cite[1.2.1]{Loisel-maximaux}, by the choice of a maximal pro-$p$ subgroup corresponding to a suitable alcove, we can assume later that $\varepsilon = 0$.
Such a choice will be done in Section \ref{sec:description:alcove}.
Moreover, because of the lack of rigidity, computations in the rank $1$ case gives large inequalities for the commutator group.
In fact, when the rank is $\geq 2$, we can make a stronger assumption, to get a more precise computation of the Frattini subgroup, as stated in Proposition \ref{prop:improvement:upper:bound:bounded:torus:non:reduced:case}.
\end{Rq}

In order to simplify notation, denote by $H(L,L_2)$ the rational points of the $K$-group $H(L,L_2)$, instead of $H(L,L_2)(K)$.
For any $(x,y),(u,v) \in H(L,L_2)$ and for any $t \in 1 + \varpi_L \mathcal{O}_L$,
up to precomposing by $\pi$, we have the following matrix realization:

$$
\widetilde{a}(t) = \begin{pmatrix}t&0&0\\0&t^{-1} {^\tau}t&0\\0&0&{^\tau}t^{-1}\end{pmatrix} $$

$\displaystyle x_a(x,y) = \begin{pmatrix}1 & -{^\tau}x & - y \\ 0 & 1 & x \\ 0 & 0 & 1 \end{pmatrix} $
\hfill $\displaystyle x_{-a}(u,v) = \begin{pmatrix}1 & 0 & 0 \\ u & 1 & 0 \\ -v & -{^\tau}u & 1 \end{pmatrix} $

We want to obtain some unipotent elements, and some semisimple elements, by multiplying suitable commutators and $p$-powers of elements in $H$, as we have done, previously, in the reduced case.
In particular, in Lemma \ref{lem:commutation:non:reduced:torus:root:group} we bound explicitely, thanks to these parametrizations, the group generated by commutators of toric elements and unipotent elements in a given root group.
In Lemma \ref{lem:commutation:non:reduced:opposite:root:groups}, we provide an explicit formula for the commutators of unipotent elements taken in opposite root groups, in terms of the parametrizations.
Finally, thanks to Lemma \ref{lem:inversion:relation:non:reduced:torus}, we invert such a commutation relation.
At last, we prove Proposition \ref{prop:non:reduced:frattini:rank:one} thanks to these lemmas.

The following lemma provides the existence of elements with minimal valuation, used in the parametrization of coroots.

\begin{Lem}
\label{lem:good:element:t}
Let $L / K$ be a quadratic Galois extension of local fields and $\tau \in \mathrm{Gal}(L/K)$ be the non-trivial element.
Let $\varpi_L$ be a uniformizer of the local ring $\mathcal{O}_L$.
Denote by $p$ the residue characteristic and assume that $p \neq 2$.
\begin{description}
\item[(1)]
For any $\forall t \in 1 + \mathfrak{m}_L$, we have $\omega\left(t^2 - {^\tau}t \right) \geq \omega\left( \varpi_L \right)$
and $\omega\left(t {^\tau}t - 1 \right) \geq \omega\left( \varpi_L \right)$.

\item[(2)]
If the extension $L/K$ is unramified, then there exists $t \in 1 + \mathfrak{m}_L$ such that $ \omega\left(t {^\tau}t - 1 \right) =\omega\left(t^2 - {^\tau}t \right) = \omega\left( \varpi_L \right)$.

\item[(3)]
If the extension $L/K$ is ramified, then for any $t \in 1 + \mathfrak{m}_L$, we have the inequality $\omega\left(t {^\tau}t - 1 \right) \geq 2 \omega\left( \varpi_L \right)$.
If $p \geq 5$, then there exists $t \in 1 + \mathfrak{m}_L$ such that $\omega\left(t {^\tau}t - 1 \right) = 2 \omega\left(t^2 - {^\tau}t \right) = 2 \omega\left( \varpi_L \right)$.

\end{description}
\end{Lem}

\begin{proof}
(1) Write $t = 1 + s$ with $\omega(s) \geq \omega(\varpi_L)$.
Then $\omega(t^2 - {^\tau}t) = \omega(2s + s^2 - {^\tau}s) \geq \omega(s)$
and $\omega(t {^\tau}t - 1) = \omega(s + {^\tau}s + s {^\tau}s) \geq \omega(s)$.

(2) If $L/K$ is unramified, one can choose a uniformizer $\varpi_L \in \mathcal{O}_L \cap K$.
Let $t = 1 + \varpi_L$, so that $t^2 - {^\tau}t = \varpi_L + \varpi_L^2 $.
Since $p \neq 2$, then $\omega(2) = 0$.
Hence $\omega\left(t {^\tau}t - 1\right) = \omega\left(2 \varpi_L + \varpi_L^2\right) = \omega\left(\varpi_L\right)$.

(3) If $L/K$ is ramified, the inequality $\omega\left(t {^\tau}t - 1 \right) \geq \omega\left( \varpi_L \right)$ is never an equality because $t {^\tau}t - 1 \in K$.
Consequently, $\omega\left(t {^\tau}t - 1 \right) \geq 2 \omega\left( \varpi_L \right)$.
Remark that $\omega\left( \varpi_L + {^\tau}\varpi_L \right) \geq 2 \omega \left( \varpi_L \right) = \omega \left( \varpi_L {^\tau} \varpi_L \right)$.
Define $t = 1 + \varpi_L$, so that $t^2 - {^\tau}t = 2 \varpi_L - {^\tau} \varpi_L + \varpi_L^2$.

By contradiction, if we had $\omega\left(2 \varpi_L - {^\tau} \varpi_L \right) \geq 2 \omega \left( \varpi_L \right)$,
then, by triangle inequality, we would get $\omega \left( 3 \varpi_L \right) \geq \min \Big( \omega \left( \varpi_L + {^\tau} \varpi_L \right) , \omega \left( 2 \varpi_L - {^\tau}\varpi_L \right) \Big) \geq 2 \omega \left( \varpi_L \right)$.
When $p \neq 3$, we have $\omega\left(3 \varpi_L\right) = \omega \left( \varpi_L \right)$.
Hence, there is a contradiction with $\omega\left(\varpi_L\right) > 0$.
As a consequence, $\omega\left(2 \varpi_L - {^\tau} \varpi_L \right) = \omega \left( \varpi_L \right)$, for any uniformizer $\varpi_L \in \mathcal{O}_L$.

Define $\varpi_L' = \varpi_L + \varpi_L {^\tau}\varpi_L$.
This element $\varpi_L' \in \mathcal{O}_L$ is also a uniformizer.
Define $t' = 1 + \varpi_L'$.
We have seen that $\omega\left(t'^2 - {^\tau}t' \right) = \omega\left(\varpi_L\right)$.

\textbf{Claim:} Either $t$ or $t'$ satisfies the desired equalities.

Indeed, we have $t{^\tau}t - 1 = \varpi_L + {^\tau}\varpi_L + \varpi_L {^\tau} \varpi_L$ 
and $t'{^\tau}t' - 1 = \varpi_L + {^\tau}\varpi_L + 3\varpi_L {^\tau} \varpi_L + \mathrm{Tr}_{L/K}\left(\varpi_L^2 {^\tau} \varpi_L\right) + N_{L/K}\left( \varpi_L \right)^2$.

By contradiction, assume that we have $\omega\left( \varpi_L + {^\tau}\varpi_L + \varpi_L {^\tau} \varpi_L \right) > 2 \omega\left( \varpi_L \right)$ and $\omega\left( \varpi_L + {^\tau}\varpi_L + 3 \varpi_L {^\tau} \varpi_L \right) > 2 \omega\left( \varpi_L \right)$.
Then, by triangle inequality, we get $\omega\left( 2 \varpi_L {^\tau}\varpi_L \right) > 2 \omega\left( \varpi_L \right)$.
Since $p \neq 2$, we have $\omega\left(2 \varpi_L {^\tau}\varpi_L\right) = 2 \omega \left( \varpi_L \right)$ and there is a contradiction.

Hence, we have, at least, 
$\omega\left( \varpi_L + {^\tau}\varpi_L + \varpi_L {^\tau} \varpi_L \right) = 2 \omega\left( \varpi_L \right)$,
or
$\omega\left( \varpi_L + {^\tau}\varpi_L + 3 \varpi_L {^\tau} \varpi_L \right) = 2 \omega\left( \varpi_L \right)$.
So, at least one of the two following equalities $\omega\left(t {^\tau}t -1 \right) = 2 \omega\left( \varpi_L \right)$ or $\omega\left(t' {^\tau}t' -1 \right) = 2 \omega\left( \varpi_L \right)$ is satisfied.
Hence $t$ or $t'$ is suitable.
\end{proof}

Denote by $H(L,L_2)_l = \left\{ (u,v) \in H(L,L_2),\ \frac{1}{2} \omega(v) \geq l \right\}$ the filtered subgroup of $H(L,L_2)$.
Remark that $H(L,L_2)_l$ can be seen as the integral points of a $\mathcal{O}_K$-model of the $K$-group scheme $H(L,L_2)$, namely the group scheme $\mathcal{H}^l$  defined by \cite[4.23]{Landvogt}.
Recall that for any $l \in \mathbb{R}$, we have $H(L,L_2)_l \simeq U_{a,l}$, by definition of the filtration on root groups, through the isomorphism $(u,v) \mapsto x_a(u,v)$.
Recall that we also have an isomorphism $\widetilde{a} : 1 + \mathfrak{m}_L \simeq T^a(K)_b^+$.

\begin{Lem}
\label{lem:commutation:non:reduced:torus:root:group}
Let $l \in \Gamma_a = \frac{1}{2} \mathbb{Z}$.

If $L/L_2$ is unramified, we have
$$U_{a,l+1} \subset \left[ T(K)_b^+ , U_{a,l} \right] \subset U_{a,l+\frac{1}{2}}$$

If $L/L_2$ is ramified, we have
$$U_{a,l+\frac{3}{2}} \subset \left[ T(K)_b^+ , U_{a,l} \right] \subset U_{a,l+\frac{1}{2}}$$
\end{Lem}

\begin{proof}
For any $t \in 1 + \varpi_L \mathcal{O}_L \simeq T(K)_b^+$ and all $(u,v) \in H(L,L_2)_l$, we have:

$$
\left[
\begin{pmatrix}1&-{^\tau}u&-v\\0&1&u\\0&0&1\end{pmatrix} ,
\begin{pmatrix}t&0&0\\0&\frac{{^\tau}t}{t}&0\\0&0&\frac{1}{{^\tau}t}\end{pmatrix}
\right] =
\begin{pmatrix}
	1 & - {^\tau}U & - V\\
	0 & 1 & U\\
	0 & 0 & 1
\end{pmatrix}
$$
where $U = \left( 1 - \frac{{^\tau}t^2}{t} \right) u$ and $V = \left(1- \frac{{^\tau}t^2}{t} \right) v + \left(t {^\tau}t - \frac{{^\tau}t^2}{t} \right) {^\tau} v$.
One can check that $(U,V) \in H(L,L_2)$.
We have:
$$\begin{array}{rcl}
\omega\left( V \right) & \geq & 
	\min \bigg(
		\omega\left(t - {^\tau}t^2 \right) + \omega(v) - \omega(t) ,
		\omega\left(\frac{{^\tau}t}{t}\right) +
		\omega\left(t^2 - {^\tau}t \right)+
		\omega\left( {^\tau} v \right)
	\bigg)\\
	& & \hfill \text{ by the triangle inequality}\\
	& = & 
	\omega\left( v \right) + \omega\left(t^2 - {^\tau}t \right)
		\hfill \text{ because } \tau \text{ preserves the valuation}\\
	& \geq &
	2 l + 1
		\hfill \text{ by lemma \ref{lem:good:element:t}(1)}
\end{array}$$

From this inequality, we deduce $(U,V) \in H(L,L_2)_{l+\frac{1}{2}}$, hence we have ${\left[U_{a,l},T(K)_b^+ \right] \subset U_{a,l+\frac{1}{2}}}$.

Conversely, let $l' \in \frac{1}{2} \mathbb{Z}$.
Let $(U,V) \in H(L,L_2)_{l'}$.
We want elements $t \in 1 + \mathfrak{m}_L$ and $(u,v) \in H(L,L_2)$ such that $[\widetilde{a}(t),x_a(u,v)]=x_a(U,V)$ and so that $\omega(v)$ is as big as possible.

Choose $t$ satisfying the equalities (2) or (3) in Lemma \ref{lem:good:element:t} applied to the extension of local fields $L/L_2$.
Let $u = \frac{t}{t-{^\tau}t^2}U$.
We seek $X,Y \in \mathcal{O}_K\left(t,{^\tau}t\right)$ such that $\left(1- \frac{{^\tau}t^2}{t} \right) v
		+ \left(t {^\tau}t - \frac{{^\tau}t^2}{t} \right) {^\tau} v = V$ where we set $v=X V + Y {^\tau} V$.
It suffices to find $X,Y$ such that:
$$\left\{
\begin{array}{rcl}
\left(1- \frac{{^\tau}t^2}{t} \right) X + \left(t {^\tau}t - \frac{{^\tau}t^2}{t} \right) {^\tau} Y & = & 1\\
\left(1- \frac{{^\tau}t^2}{t} \right) Y + \left(t {^\tau}t - \frac{{^\tau}t^2}{t} \right) {^\tau} X & = & 0
\end{array}
\right.$$

The unique solution of this linear system is:
$$\left\{
\begin{array}{rcl}
X & = & \frac{1}{\left(1 - t {^\tau}t \right)\left(1 - \frac{{^\tau}t^2}{t}\right)}\\
Y & = & \frac{\frac{{^\tau}t^2}{t}}{\left(1 - t {^\tau}t \right)\left(1 - \frac{{^\tau}t^2}{t} \right)}
\end{array}
\right.$$
so that:
$$v = X V + Y {^\tau} V = \frac{V + \frac{{^\tau}t^2}{t} {^\tau}V}{\left( 1 - t {^\tau}t \right) \left( 1 - \frac{{^\tau}t^2}{t} \right)}$$
satisfies $(u,v) \in H(L,L_2)$.

By a matrix computation, and because $t,u,v$ have been chosen for this, we can check that $\left[x_a(u,v),\widetilde{a}(t) \right] = x_a(U,V)$.
Moreover, the valuation gives us $\omega(v) \geq \omega( V ) - \omega(1 - t {^\tau} t ) - \omega( t - {^\tau}t^2 )$ because $\omega\left(V + \frac{{^\tau}t^2}{t} {^\tau}V \right) \geq \omega(V)$.

When $L/L_2$ is unramified, by \ref{lem:good:element:t}(2), this gives us $\omega(v) \geq 2 l' - 2$.
From this inequality, we deduce $(u,v) \in H(L,L_2)_{l'-1}$, 
hence:
$$\left[U_{a,l'-1},T(K)_b^+ \right] \supset U_{a,l'}$$

When $L/L_2$ is ramified, by \ref{lem:good:element:t}(3), this gives us $\omega(v) \geq 2 l' - 3$.
From this inequality, we deduce $(u,v) \in H(L,L_2)_{l'-\frac{3}{2}}$, 
hence:
$$\left[U_{a,l'-\frac{3}{2}},T(K)_b^+ \right] \supset U_{a,l'}$$
\end{proof}

\begin{Rq}\label{rq:refined:ramified:non:reduced:commutation:relation:torus:root:group}These inequalities could be refined, with a deeper study on the arithmetic properties of the local fields.
As an example, when $L / L_2$ is ramified, and $l \not\in \mathbb{Z}$,
we obtain $\left[ T(K)_b^+ , U_{a,l} \right] \subset U_{a,l+1}$.
\end{Rq}

\begin{Lem}[Commutation of opposite root groups]
\label{lem:commutation:non:reduced:opposite:root:groups}
Let $l,l' \in \Gamma_a = \frac{1}{2} \Gamma_L = \frac{1}{2} \mathbb{Z}$ be such that $l+l' > 0$.
Let $(x,y) \in H(L,L_2)_l$ and $(u,v) \in H(L,L_2)_{l'}$.
We have $\left[x_{-a}(x,y),x_a(u,v)\right] = x_{-a}(X,Y) \widetilde{a}(T) x_a(U,V)$ where:
$$\left\{\begin{array}{rcl}
	T & = &	1-{^\tau}u x + vy \\
	U & = & \frac{1}{{^\tau}T} \left(u^2 {^\tau}x - {^\tau}v x - u {^\tau}v {^\tau}y \right)\\
	V & = &	\frac{1}{T} \left( u v {^\tau}x - {^\tau}u {^\tau}v x + v {^\tau}v y \right) \\
	X & = & \frac{1}{T} \left( {^\tau}u x^2 - u y + v x y \right)\\
	Y & = &	\frac{1}{T} \left( {^\tau}x u y - {^\tau}u x {^\tau}y + v y {^\tau}y \right)
\end{array}\right.$$

Moreover, $\omega(V) \geq \lceil 3 l' + l \rceil$ and $\omega(Y) \geq \lceil l' +  3 l \rceil$.

Consequently,
$$\begin{array}{rcl}
\left[U_{-a,l},U_{a,l'}\right] & \subset & U_{-a,\frac{\lceil 3l + l' \rceil}{2}} T^a(K)_b^+ U_{a,\frac{\lceil l + 3l' \rceil}{2}} \\
& \subset & U_{-a,l+\frac{1}{2}} T^a(K)_b^+ U_{a,l' + \frac{1}{2}}
\end{array}$$
\end{Lem}

\begin{proof}
Because $\tau$ preserves $\omega$, we have the following in $H(L,L_2)$:
$$2 \omega(u) = \omega(u {^\tau}u) = \omega(v + {^\tau}v) \geq \omega(v)$$
Hence, we have:
$$\omega(x) + \omega(u) \geq \frac{1}{2} \big(\omega(y) + \omega(v)\big) \geq l + l' > 0$$

By a matrix computation in $\mathrm{SU}(h)$, we have:

$$ \begin{pmatrix}1 & -{^\tau}u & - v \\ 0 & 1 & u \\ 0 & 0 & 1 \end{pmatrix}
\begin{pmatrix}1 & 0 & 0 \\ x & 1 & 0 \\ -y & -{^\tau}x & 1 \end{pmatrix}   =
\begin{pmatrix}1&0&0\\X_0&1&0\\-Y_0&-{^\tau}X_0&1\end{pmatrix}
\begin{pmatrix}T&0&0\\0&\frac{ {^\tau}T}{T}&0 \\ 0& 0 & \frac{1}{{^\tau}T}\end{pmatrix}
\begin{pmatrix}1&-{^\tau}U_0&-V_0\\0&1&U_0\\0&0&1 \end{pmatrix}
$$

where
$$\left\{\begin{array}{rcl}
	T   & = &	1-{^\tau}u x + vy \\
	U_0 & = & 	\frac{1}{{^\tau}T} (u - {^\tau}v x)\\
	V_0 & = &	\frac{1}{T} v \\
	X_0 & = & 	\frac{1}{T} (x-uy)\\
	Y_0 & = &	\frac{1}{T} y
\end{array}\right.$$
Because $\omega({^\tau}u x) \geq \frac{1}{2} \omega(vy) > 0$,
we get $T \in 1 + \mathfrak{m}_L$.
Hence $\frac{1}{T} \in \mathcal{O}_L^\times$ is well-defined.
It follows:
$$
\left[
\begin{pmatrix}1 & 0 & 0 \\ - x & 1 & 0 \\ -{^\tau} y & {^\tau}x & 1 \end{pmatrix} ,
\begin{pmatrix}1 & -{^\tau}u & - v \\ 0 & 1 & u \\ 0 & 0 & 1 \end{pmatrix} 
\right] =
\begin{pmatrix}1&0&0\\X&1&0\\-Y&-{^\tau}X&1\end{pmatrix}
\begin{pmatrix}T&0&0\\0&\frac{ {^\tau}T}{T}&0 \\ 0& 0 & \frac{1}{{^\tau}T}\end{pmatrix}
\begin{pmatrix}1&-{^\tau}U&-V\\0&1&U\\0&0&1 \end{pmatrix}
$$

where
$$\left\{\begin{array}{rcl}
	T & = &	1-{^\tau}u x + vy \\
	U & = & \frac{1}{{^\tau}T} \left(u^2 {^\tau}x - {^\tau}v x - u {^\tau}v {^\tau}y \right)\\
	V & = &	\frac{1}{T} \left( u v {^\tau}x - {^\tau}u {^\tau}v x + v {^\tau}v y \right) \\
	X & = & \frac{1}{T} \left( {^\tau}u x^2 - u y + v x y \right)\\
	Y & = &	\frac{1}{T} \left( {^\tau}x u y - {^\tau}u x {^\tau}y + v y {^\tau}y \right)
\end{array}\right.$$

We have $$\begin{array}{ccl}
\omega(V) & \geq & \min\big( \omega( u v {^\tau}x ) , \omega( {^\tau}u {^\tau}v x ) , \omega(v {^\tau}v y) \big) \\
& \geq & \omega( v ) + \min\big( \omega( u ) + \omega( x ), \omega( v ) + \omega( y ) \big) \\
& \geq & 2 l' + l + l'
\end{array}$$
Because $\omega( V ) \in \mathbb{Z}$, we have in fact $\omega( V ) \geq \lceil 3 l' + l \rceil \geq 2 l' + 1$.

We proceed in the same way to find a lower bound of $\omega(Y)$.
\end{proof}

In order to compute a derived group in terms of root groups, we would like to invert the above equations. Precisely, given a $t \in 1 + \mathfrak{m}_L^{l''}$, we seek elements $(u,v), (x,y) \in H(L,L_2)$ with prescribed valuations $l,l' \in \frac{1}{2} \mathbb{Z}$ such that $t = 1 - {^\tau}u x +vy$.
The existence of such $(u,v),(x,y)$ is not guaranteed if $l''$ is not large enough.
Firstly, we seek an element $(u,v) \in H(L,L_2)_l$ such that $\omega\left(\mathrm{Tr}(u)\right)$ is minimal.

\begin{Lem}
\label{lem:minimal:trace:uniformizer}
Let $L/K$ be a quadratic Galois extension of local fields with a residue characteristic $p \neq 2$ and a discrete valuation $\omega : L^\times \rightarrow \mathbb{Z}$.
There exists a uniformizer $\varpi_L$ in $\mathcal{O}_L$ such that $\mathrm{Tr}_{L/K}(\varpi_L)$ is a uniformizer of $\mathcal{O}_K$.
\end{Lem}

\begin{proof}
If $L/K$ is unramified, we can choose a uniformizer $\varpi_L$ of $\mathcal{O}_L$ in $\mathcal{O}_K$.
Because $p \neq 2$, the element $\mathrm{Tr}_{L/K}(\varpi_L) = 2 \varpi_L$ is a uniformizer in $\mathcal{O}_K$.

If $L/K$ is ramified, let $\varpi'$ be a uniformizer of $\mathcal{O}_L$.
We know that $\omega\left(\mathrm{Tr}_{L/K}(\varpi')\right) \geq \min\left( \omega\left( \varpi'\right), \omega\left( {^\tau}\varpi' \right) \right) = 1$.
This is never an equality because $\Gamma_K = \omega(K^\times) = 2 \mathbb{Z}$.

If $\omega\left(\mathrm{Tr}_{L/K}(\varpi')\right) = 2$, then we set $\varpi_L = \varpi'$.
Otherwise, we set $\varpi_L = \varpi' + N_{L/K}(\varpi')$.
Thus, $\varpi_L$ is a uniformizer because $\omega\left( N_{L/K}(\varpi') \right) = 2 > 1 = \omega(\varpi')$.
Moreover, $\mathrm{Tr}_{L/K}(\varpi_L) = \mathrm{Tr}_{L/K}(\varpi') + 2 N_{L/K}(\varpi')$.
Because $\omega\left( \mathrm{Tr}_{L/K}(\varpi') \right) > \omega\left( 2 N_{L/K}(\varpi') \right) = 2$, we get the result.
\end{proof}

\begin{Lem}
\label{lem:parametrisation:optimized:unipotent}
Assume that $p \neq 2$ and let $l\in \Gamma_L = \mathbb{Z}$.

If $L/L_2$ is unramified, set $\varepsilon = 0$.

If $L/L_2$ is ramified, set
$\varepsilon = \left\{\begin{array}{cl}0 & \text{ if } l \in \Gamma_{L_2} = 2 \mathbb{Z}\\1 & \text{ otherwise}\end{array}\right.$

There exists $u \in L$ such that:
\begin{itemize}
\item[(a)] $\omega(u) = l$;
\item[(b)] $\omega\left(\mathrm{Tr}_{L/L_2}(u)\right) = l + \varepsilon$;
\item[(c)] $\left(u,\frac{1}{2} u {^\tau}u \right) \in H(L,L_2)_{l}$.
\end{itemize}
\end{Lem}

\begin{proof}
Let $\varpi_L$ be a uniformizer of $\mathcal{O}_L$ such that $\varpi_{L_2} = \mathrm{Tr}_{L/L_2}(\varpi_L)$ is a uniformizer of $\mathcal{O}_{L_2}$, such a uniformizer exists by Lemma \ref{lem:minimal:trace:uniformizer}.
Define $u = (\varpi_L)^\varepsilon \cdot (\varpi_{L_2})^{\frac{l- \varepsilon}{\omega(\varpi_{L_2})}}$.

(a) $\omega(u) = \varepsilon \omega(\varpi_L) + \frac{l - \varepsilon}{\omega(\varpi_{L_2})} \omega(\varpi_{L_2}) = l$.

(b) We have:
$$\begin{array}{rcl}
\mathrm{Tr}_{L/L_2}(u) & = & \mathrm{Tr}_{L/L_2}\left((\varpi_L)^{\varepsilon}\right) \cdot (\varpi_{L_2})^{\frac{l- \varepsilon}{\omega(\varpi_{L_2})}} \\
& = &
\left\{\begin{array}{ll}
(\varpi_{L_2})^{\frac{l- \varepsilon}{\omega(\varpi_{L_2})} + \varepsilon} & \text{ if } \varepsilon=1\\
2 (\varpi_{L_2})^{\frac{l- \varepsilon}{\omega(\varpi_{L_2})}}& \text{ if } \varepsilon = 0
\end{array}\right.
\end{array}$$
Hence $\omega\left(\mathrm{Tr}_{L/L_2}(u)\right) = \left( \frac{l- \varepsilon}{\omega(\varpi_{L_2})} + \varepsilon \right) \omega(\varpi_{L_2}) = l- \varepsilon + \varepsilon \omega(\varpi_{L_2}) =  l + \varepsilon$.

(c) We have $N_{L/L_2}(u) = u {^\tau} u = \mathrm{Tr}\left( \frac{1}{2} u {^\tau}u \right)$.
\end{proof}

As a consequence, we got an element $(u,v)$ such that $\mathrm{Tr}_{L/L_2}(u)$ is minimal.
Secondly, we seek an element $(x,y) \in H(L,L_2)_{l'}$ such that $t = 1 - {^\tau}ux + vy$.
This is a quadratic problem.
That is why we recall the following lemma on the existence of square root.

\begin{Lem}
\label{lem:square:root}
Let $L$ be a local field of residue characteristic $p \neq 2$.
For all $a \in \mathfrak{m}_L$, there exists $b \in \mathfrak{m}_L$ such that $(1+b)^2 = 1+a$ and $\omega(a) = \omega(b)$.
\end{Lem}

\begin{proof}
Let $a \in \mathfrak{m}_L$.
By Hensel's Lemma, the polynomial $X^2 - 1 - a$ admits exactly two roots $1+b$ and $-1+b'$ in $\mathcal{O}_L$, with $b, b' \in \mathfrak{m}_L$ since $1$ and $-1$ are two distinct roots in $\kappa_L$ of the polynomial $X^2-1$.
Moreover $\omega(a) = \omega\big( (1+b)^2 - 1 \big) = \omega(b) + \omega(2 + b)$.
Since $p \neq 2$, we have $\omega(2+b) = 0$.
Hence, $\omega(a) = \omega(b)$.
\end{proof}

We provide a solution $(x,y) \in H(L,L_2)_{l'}$ of $t = 1 - {^\tau}ux + vy$ for a suitable value $l''$ such that $t \in 1+ \mathfrak{m}_L^{l''}$.

\begin{Lem}
\label{lem:inversion:relation:non:reduced:torus}
Assume that $p \neq 2$.
Let $l,l'\in \Gamma_a$ be such that $l+l' > 0$ and $l \in \Gamma'_a = \mathbb{Z}$.
Define $\varepsilon \in \{0,1\}$ as in Lemma \ref{lem:parametrisation:optimized:unipotent}.
Define
$$l'' = \max\left( 1 + 2 \varepsilon, \varepsilon + 2l + 2l' \right) \in \mathbb{N}^*$$
For any $w \in \mathfrak{m}_{L}^{l''}$, there exist $(u,v) \in H(L,L_2)_{l}$ and $(x,y) \in H(L,L_2)_{l'}$ such that ${^\tau}u x - v y = w$.
\end{Lem}

\begin{proof}
In order to simplify notation in this proof, we denote by $T$ the field trace operator $\mathrm{Tr}_{L/L_2} : L \rightarrow L_2$.

Let $w \in \left(\mathfrak{m}_{L}\right)^{l''}$.
Choose $u \in L$ satisfying the properties (a),(b) and (c) of Lemma \ref{lem:parametrisation:optimized:unipotent} and set $v = \frac{1}{2} u {^\tau} u$.
We seek an element $(x,y) \in H(L,L_2) \cap ( L_2 \times L) $ such that ${^\tau}x u - v y = w$, which is equivalent to
$$\left\{\begin{array}{l}
y = \frac{- w + {^\tau} u x}{v}\\
x^2 = T(y) = - T\left(\frac{w}{v}\right) + x T\left(\frac{{^\tau}u}{v}\right)
\end{array}\right.$$
because $v \neq 0$ (otherwise property (a) would be contradicted).

Denote $\delta = 4 \frac{T\left(\frac{w}{v}\right)}{T\left(\frac{{^\tau}u}{v}\right)^2}$.
We have $T\left(\frac{{^\tau}u}{v}\right) = 2 \frac{T(u)}{u{^\tau}u}$ by definition of $v = \frac{1}{2}u{^\tau}u \in L_2$ and by $L_2$-linearity of $T$.
Hence $\omega\left(T\left(\frac{{^\tau}u}{v}\right)\right) = \omega\left( T(u) \right) - 2 \omega(u) = -l + \varepsilon$.
We have $\omega\left(T\left(\frac{w}{v}\right)\right) \geq \omega(w) - \omega(v) \geq l'' - 2 l$.
Hence $\omega(\delta) = \omega\left(T\left(\frac{w}{v}\right)\right) - 2 \omega\left(T\left(\frac{{^\tau}u}{v}\right)\right) \geq 
l'' - 2 \varepsilon \geq 1$.
By Lemma \ref{lem:square:root}, there exists $b \in \mathfrak{m}_{L_2}$ such that $(1+b)^2 = 1-\delta$ and $\omega(b) = \omega(\delta)$.
We denote $\sqrt[2]{1-\delta} = 1+b$.
Hence $\sqrt[2]{1-\delta} \in 1 + \delta \mathcal{O}_{L_2}$ is well-defined and $\omega\left( \sqrt[2]{1-\delta} - 1 \right) = \omega(\delta)$.

Set $x = \frac{1}{2} T\left( \frac{{^\tau}u}{v} \right) \left(1 - \sqrt[2]{1-\delta} \right)  \in L_2$ and set $y = \frac{w - {^\tau}u x}{v} \in L$.
We have $x^2 = T(y)$.
Moreover, $\omega(x) = \omega(\delta) + \varepsilon - l$.
We check the valuation of $y$:
$$\begin{array}{rcl}
\omega(y) & \geq & \min\big( \omega(w), \omega(u) + \omega(x) \big) - \omega( v )\\
& = & \min\big(l'', \omega(\delta) + \varepsilon \big) - 2l \\
& \geq & \min\left( l'', l'' -2 \varepsilon + \varepsilon \right) - 2l\\
& = & l'' - \varepsilon - 2l\\
& \geq & 2l'
\end{array}$$

Hence $(u,v) \in H(L,L_2)_{l}$ and $(x,y) \in H(L,L_2)_{l'}$ are suitable.
\end{proof}

Finally, we can combine Lemmas \ref{lem:commutation:non:reduced:torus:root:group}, \ref{lem:commutation:non:reduced:opposite:root:groups} and \ref{lem:inversion:relation:non:reduced:torus} in order to prove Proposition \ref{prop:non:reduced:frattini:rank:one}.

\begin{proof}[Proof of Proposition \ref{prop:non:reduced:frattini:rank:one}]
Up to exchanging $a$ and $-a$, one can suppose $l \in \Gamma'_a = \mathbb{Z} =\Gamma_L$.

By Lemma \ref{lem:commutation:non:reduced:torus:root:group}, we get $U_{-a,-l+1} \subset [H,H]$ and $U_{a,l+\frac{3}{2}} \subset [H,H]$ when $L/L_2$ is unramified; we get $U_{-a,-l+\frac{3}{2}} \subset [H,H]$ and $U_{a,l+2} \subset [H,H]$ when $L/L_2$ is ramified.

Let $t \in T^a(K)_b^{l''}$ and write it $t = \widetilde{a}(1+w)$ where 
$w \in \left(\mathfrak{m}_L\right)^{l''}$.
Set $l_0=l+1 \in \mathbb{Z}$ et $l'_0=-l+\frac{1}{2}$.
By Lemma \ref{lem:inversion:relation:non:reduced:torus}, there exists $(u,v) \in H(L,L_2)_{l_0}$ and $(x,y) \in H(L,L_2)_{l'_0}$ such that $-w = {^\tau}u x - v y$.

We use the commutation relation of opposite root groups \ref{lem:commutation:non:reduced:opposite:root:groups}.
Let:
$$\left\{\begin{array}{rcl}
	T & = &	1+w \\
	U & = & \frac{1}{{^\tau}T} \left(u^2 {^\tau}x - {^\tau}v x - u {^\tau}v {^\tau}y \right)\\
	V & = &	\frac{1}{T} \left( u v {^\tau}x - {^\tau}u {^\tau}v x + v {^\tau}v y \right) \\
	X & = & \frac{1}{T} \left( {^\tau}u x^2 - u y + v x y \right)\\
	Y & = &	\frac{1}{T} \left( {^\tau}x u y - {^\tau}u x {^\tau}y + v y {^\tau}y \right)
\end{array}\right.$$
By Lemma \ref{lem:commutation:non:reduced:opposite:root:groups},
we have $\left[x_{-a}(x,y),x_a(u,v)\right] = x_{-a}(X,Y) \widetilde{a}(T) x_a(U,V)$ with $\omega(V) \geq \lceil 3 l' + l \rceil$ and $\omega(Y) \geq \lceil l' +  3 l \rceil$.

Because $l \in \mathbb{Z}$, we have $\frac{1}{2} \lceil 3l'_0 + l_0 \rceil = -l+\frac{3}{2}$ and $\frac{1}{2} \lceil l'_0 + 3 l_0 \rceil = l+2$.
Hence $x_{-a}(X,Y) \in [T(K)_b^+,U_{-a,l}]$ and $x_{a}(U,V) \in [T(K)_b^+,U_{a,l+\frac{1}{2}}]$ by Lemma \ref{lem:commutation:non:reduced:torus:root:group}.
Because $\widetilde{a}(1+w) = x_{-a}(X,Y)^{-1} \left[ x_{-a}(x,y),x_a(u,v)\right] x_a(U,V)^{-1} \in [H,H]$, we get $T^a(K)_b^{l''} \subset [H,H]$.

We now assume that $\mathrm{char}(K) = p \geq 5$.
It suffices to check that $H^p \subset [H,H]$.
Inside $H / [H,H]$, we have $u^p = 1$ for any $u \in U_{a,l}$ and it is the same for $-a$.
Indeed, the element $x_a\left(u,v\right)^p = x_a\left(pu,pv+\frac{p(p-1)}{2} u {^\tau}u\right)$ is the neutral element in characteristic $p \neq 2$.

Moreover, if $t \in T^{a}(K)_b^+$, write $t = \widetilde{a}(1+w)$ where $w \in \mathfrak{m}_L$.
We have $(1+w)^p = 1 + w^p$ with $\omega(w^p) \geq p \geq 5 \geq l''$.
Hence $t^p \in T^a(K)_b^{l''} \subset [H,H]$.
\end{proof}

In the case of higher rank, we obtain in Proposition \ref{prop:generated:commutator:root:group} some inclusions of the form $U_{a,l_a} \subset [H,H]$ with a suitable value $l_a$, by commuting some root groups corresponding to non-collinear roots.
Hence, it is useful to do a further assumption on subgroups contained in $[H,H]$.

\begin{Prop}
\label{prop:improvement:upper:bound:bounded:torus:non:reduced:case}
If in Proposition \ref{prop:non:reduced:frattini:rank:one} we furthermore assume that $[H,H]H^p$ contains $U_{a,l+1}$ and $U_{-a,-l+\frac{1}{2}}$, then one can take $l'' = 1 + 2 \varepsilon$.
\end{Prop}

\begin{proof}
In the above proof, up to exchanging $a$ and $-a$ so that $l \in \mathbb{Z} + \frac{1}{2}$ and $l' \in \mathbb{Z}$, we can replace the equalities $l_0 = l+1$ and $l'_0 = -l + \frac{1}{2}$ by $l_0 = l+\frac{1}{2} \in \mathbb{Z}$ and $l'_0 = -l$.
Indeed, in this case we obtain $\lceil 3 l'_0 + l_0 \rceil = \lceil -2l + \frac{1}{2} \rceil = -2l+1$, so that $U_{-a,\frac{1}{2} \lceil 3 l'_0 + l_0 \rceil} \subset H^p [H,H]$ by the additional assumption.
In the same way, $\lceil 3 l'_0 + l_0 \rceil = 2l+2$ so that $U_{a,\frac{1}{2} \lceil l'_0 + 3 l_0 \rceil} \subset H^p [H,H]$.
As a consequence, we can conclude as before.
\end{proof}

To conclude this section, we compute the commutation relation between elements of the same root group. This is non-trivial because, in the non-reduced case, the root group is non-commutative.
This will be useful in order to understand the action of a maximal pro-$p$ subgroup on the Bruhat-Tits building.

\begin{Lem}[Computation of the derived group of a valued root group: specificity on the non-reduced case]
\label{lem:non:reduced:derived:valued:root:group}
Let $l,l' \in \Gamma_a = \frac{1}{2} \mathbb{Z}$.
In general, we have $\left[ U_{a,l}, U_{a,l'} \right] \subset U_{2a,\lceil l \rceil + \lceil l' \rceil}$.

If $L/L_2$ is unramified and $p\neq 2$, then $\left[ U_{a,l}, U_{a,l} \right] = U_{2a,2\lceil l \rceil}$.

If $L/L_2$ is ramified and $p\neq 2$, then $\left[ U_{a,l}, U_{a,l} \right] = U_{2a,2\lceil l \rceil +1}$.
\end{Lem}

\begin{proof}
Let $(u,v) , (x,y) \in H(L,L_2)$.
In matrix-wise terms, we have
$$\left[\begin{pmatrix}
1&-{^\tau}x&-y\\
0&1&-x\\
0&0&1\end{pmatrix} ,
\begin{pmatrix}
1&-{^\tau}u&-v\\
0&1&-u\\
0&0&1\end{pmatrix} \right] =
\begin{pmatrix}
1&0&x {^\tau}u - u {^\tau}x\\0&1&0\\0&0&1
\end{pmatrix}$$
We deduce that $\left[x_a(x,y),x_a(u,v)\right] = x_a\left(0,x {^\tau}u - u {^\tau}x\right)$.

If $\omega(y) \geq 2l$, then $\omega(x) \geq \lceil l \rceil$ because $\omega(x) \in \Gamma_L = \mathbb{Z}$.
Likewise, if $\omega(v) \geq 2l'$, then $\omega(u) \geq \lceil l' \rceil$.
Hence $\omega\left( x {^\tau}u - u {^\tau}x \right) \geq \omega(u) + \omega(x) \geq \lceil l \rceil + \lceil l' \rceil$.
We obtain $\left[ U_{a,l}, U_{a,l} \right] \subset U_{2a,\lceil l \rceil + \lceil l' \rceil}$.

Conversely, we show that any element of $U_{2a,2\lceil l \rceil}$ can be written as the commutator of two suitable elements in $U_{a,l}$.
For that, it suffices to show that for any $w \in (L^0)_{2\lceil l \rceil}$, there exist $(u,v), (x,y) \in H(L,L_2)_l$ such that $w=x {^\tau}u - u {^\tau}x$.

We firstly consider the case of a unramified extension $L/L_2$ with $p \neq 2$.
In this case, we have $\Gamma'_{2a} = \Gamma_{2a} = \mathbb{Z}$ by Lemma \ref{lem:sets:of:values:multipliable:root}.
Hence, there exists $\displaystyle \lambda_0 \in (L^0)_0 = \left\{\lambda \in \mathcal{O}_L^\times,\ \lambda + {^\tau}\lambda =0\right\}$.
Let $\varpi \in \mathcal{O}_{L_2}$ be a uniformizer.
Set $x = \lambda_0 \varpi^{\lceil l \rceil}$ and set $y = \frac{1}{2} x {^\tau}x$ so that $(x,y) \in H(L,L_2)_{l}$.
Let $w \in (L^0)_{2 \lceil l \rceil}= \left\{ w_0 \in \left(\mathfrak{m}_L\right)^{2 \lceil l \rceil} ,\  w_0 + {^\tau}w_0 = 0  \right\}$.
Then $u = \frac{w}{x-{^\tau}x} \in L_2$.
Indeed, ${^\tau}u = \frac{{^\tau}w}{{^\tau}x-x} = \frac{-w}{-(x-{^\tau}x)} = u$.
Moreover, $\omega(x-{^\tau}x) = \omega\left((\lambda_0 - {^\tau}\lambda_0) \varpi^{ \lceil l \rceil}\right) = \omega(2 \lambda_0) + \omega(\varpi^{ \lceil l \rceil}) = \lceil l \rceil$ because $p \neq 2$.
Hence $\omega(u) = \omega(w) - \omega(x-{^\tau}x) = \lceil l \rceil$.
Set $v = \frac{1}{2} u {^\tau} u = \frac{u^2}{2}$ so that $(u,v) \in H(L,L_2)_l$.
We have $x {^\tau}u - u {^\tau}x = u (x-{^\tau}x) = w$.

We secondly consider the case of a ramified extension $L/L_2$ with $p \neq 2$.
In this case, $\Gamma'_{2a} = \Gamma_{2a} = 2\mathbb{Z} +1$ by Lemma \ref{lem:sets:of:values:multipliable:root}.
Thus $ U_{2a,2\lceil l \rceil} = U_{2a,2\lceil l \rceil +1}$.
Moreover, there exists $\displaystyle \lambda_0 \in (L^0)_1 = \left\{\lambda \in \mathcal{O}_L,\ \lambda + {^\tau}\lambda =0 \text{ et } \omega(\lambda) = 1\right\}$.
Let $\varpi \in \mathcal{O}_{L_2}$ be a uniformizer.

If $\lceil l \rceil \in 2 \mathbb{Z}$, we set $x = \lambda_0 \varpi^{\frac{\lceil l \rceil}{2}}$ and $y = \frac{1}{2} x {^\tau}x$ so that $(x,y) \in H(L,L_2)_{l}$.

Otherwise, $\lceil l \rceil \in 2 \mathbb{Z} +1 $. We set $x = \lambda_0 \varpi^{\frac{\lceil l \rceil - 1}{2}}$ and $y = \frac{1}{2} x {^\tau}x$ so that $(x,y) \in H(L,L_2)_{l}$.

Let $w \in (L^0)_{2 \lceil l \rceil}= \left\{ w_0 \in \left(\mathfrak{m}_L\right)^{2 \lceil l \rceil} ,\  w_0 + {^\tau}w_0 = 0  \right\}$.
Then, as before, we get $u = \frac{w}{x-{^\tau}x} \in L_2$.
Moreover, $\omega\left((\lambda_0 - {^\tau}\lambda_0\right) = \omega(2 \lambda_0) = 1$ because $p \neq 2$.
Hence, we obtain the inequalities $\omega(x) \geq \lceil l \rceil $ and $ \omega\left( x - {^\tau}x \right) =  \leq \lceil l \rceil + 1$.
Hence $\omega(u) = \omega(w) - \omega(x-{^\tau}x) \geq \lceil l \rceil$.
We set $v = \frac{1}{2} u {^\tau} u = \frac{u^2}{2}$ so that $(u,v) \in H(L,L_2)_l$.
We get $x {^\tau}u - u {^\tau}x = u (x-{^\tau}x) = w$.
\end{proof}

\section{Bruhat-Tits theory for quasi-split semisimple groups} 
\label{sec:bruhat:tits:building}

In Bruhat-Tits theory, a building is attached to a reductive group in two steps.
The first step, in \cite[§4]{BruhatTits2}, corresponds to split and quasi-split groups.
The second step in \cite[§5]{BruhatTits2} is an \'etale descent to the base field.
In order to describe some subgroups in terms of the action on the Bruhat-Tits building, in Section \ref{sec:description:apartment}, we recall how the simplicial structure of the building is defined thanks to the valuation of root groups.
Then, in Section \ref{sec:action:on:ball}, we consider the action of the group $G(K)$ on its Bruhat-Tits building $X(G,K)$.
In this section, $K$ is a local field and $G$ is an almost-$K$-simple simply-connected quasi-split $K$-group.

\subsection{Numerical description of walls and alcoves}
\label{sec:description:apartment}

The Bruhat-Tits building of $(G,K)$ is obtained by gluing together affine spaces, called apartments, having the same given simplicial structure.
This consists in defining the building as $X(G,K) = G(K) \times \mathbb{A} / \sim$, where $\mathbb{A}$ is a suitable affine space, called the standard apartment,  see \cite[§9]{Landvogt}.
The apartments are glued together along hyperplanes called walls, that we will describe as zero sets of affine functions thanks to the sets of values defined in Section \ref{sec:set:of:values}.
In Section \ref{sec:walls}, we recall how we deduce the simplicial structure of an apartment from the definition of walls.
More precisely, we define an \gts{affinisation} of the spherical root system following the Bruhat-Tits method. In Lemma \ref{lem:equivalence:affine:root:systems}, we check that this construction coincide with the affine root system defined by Tits in \cite{TitsCorvallis}.
In Section \ref{sec:description:alcove}, we describe, thanks to the sets of values, a well-chosen alcove, which is the candidate to be a fundamental domain of the action of $G(K)$ on $X(G,K)$.
In Section \ref{sec:counting:alcoves}, we look locally the building near an alcove.

\subsubsection{Walls of an apartment of the Bruhat-Tits building} 
\label{sec:walls}

In \cite[§1]{Landvogt}, we define a simplicial structure for apartments as follows.
Firstly, we let $\mathbb{A} = \mathbb{A}(G,S,K)$ be the unique affine space under $V = X_*(S) \otimes_\mathbb{Z} \mathbb{R}$ together with a suitable group homomorphism $\nu : \mathcal{N}_G(S)(K) \rightarrow \mathrm{Aff}(\mathbb{A})$.

Secondly, each relative root $a \in \Phi \subset X^*(S)$ induces a linear form on $V$ deduced by linearity from the dual pairing $X_*(S) \times X^*(S) \rightarrow \mathbb{Z}$. Hence, up to choice of an origin $\mathcal{O} \in \mathbb{A}$, each relative root induces an affine map on $\mathbb{A}$.

Thirdly, any relative root $a\in \Phi \subset X^*(S)$ can be seen as a linear form on $V=X_*(S) \otimes_\mathbb{Z} \mathbb{R}$, arising from the dual pairing $\langle \cdot , \cdot \rangle : X^*(S) \times X_*(S) \rightarrow \mathbb{R}$.
From this spherical root system (where each root is seen as a linear form), we define an \gts{affinisation}.
Hence, each affine map $\theta(a,l) = a( \cdot - \mathcal{O}) - l : \mathbb{A} \rightarrow \mathbb{R}$, where $a \in \Phi$ and $l \in \mathbb{R}$, determinates a unique \textbf{half-apartment} denoted by:
$$D(a,l) = \{x \in \mathbb{A},\ \theta(a,l)(x) > 0 \}$$
whose border (an affine subspace of codimension one) is denoted by $\mathcal{H}_{a,l} = \{x \in \mathbb{A},\ \theta(a,l)(x) = 0 \}$.
When $l \in \Gamma'_a$, the affine map $\theta(a,l)$ is called an \textbf{affine root}.
In Lemma \ref{lem:equivalence:affine:root:systems}, we will see that the set of affine roots is the affine root system of \cite[1.6]{TitsCorvallis}.

For each affine root $\theta(a,l)$, the corresponding $\mathcal{H}_{a,l}$ is called a \textbf{wall} of $\mathbb{A}$.
The walls induce a structure of poly-simplicial complex on $\mathbb{A}$:
a connected component of $\displaystyle \mathbb{A} \setminus \bigcup_{a \in \Phi,\  l \in \Gamma'_a} \mathcal{H}_{a,l}$ is called an \textbf{alcove}.
It is a simplex of maximal dimension.
More generally, we define an equivalence relation on points on $\mathbb{A}$ by $x \sim y$ if, for any $a \in \Phi$, if the real numbers $a(x)$ and $a(y)$ have the same sign or are both equal to zero.
That means $x \sim y$ if, and only if, $x$ and $y$ always are in the same half-apartment.
An equivalence class is called a \textbf{facet}; alcoves are the facets of maximal dimension.
The set of facets constitutes a partition of $\mathbb{A}$.
Finally, the affine space $\mathbb{A}$ together with the affine root system $\{ \theta(a,l),\ a \in \Phi\text{ and } l \in \Gamma'_a \}$ and the structure of poly-simplicial complex deduced from the walls is called the \textbf{standard apartment}.

\begin{Not}
\label{not:Omega}
For any non-empty bounded subset $\Omega$ of $\mathbb{A}$, according to \cite[§4]{BruhatTits2} and \cite[§5]{Landvogt}, we denote:
\begin{itemize}
\item $f_\Omega(a) = \sup\{-a(x),\ x \in \Omega \}$ for any relative root $a \in \Phi$;
\item $U_{a,\Omega} = U_{a,f_\Omega(a)}$ for any relative root $a \in \Phi$;
\item $\begin{array}{rcl}f'_\Omega(a) & = & \inf\left\{l \in \Gamma'_a,\ l \geq f_\Omega(a) \text{ or } \frac{1}{2} l \geq f_\Omega(\frac{a}{2}) \right\} \\ & = & \sup \{ l \in \mathbb{R},\  U_{a,l} = U_{a,f_\Omega(a)}\}\end{array}$
\item $U_\Omega$ the subgroup of $G(K)$ generated by the groups $U_{a,\Omega}$ where $a\in\Phi$;
\item $N_\Omega = \{ n \in \mathcal{N}_{G}(S)(K),\ \forall x \in \Omega,\ n \cdot x = x\}$;
\item $P_\Omega = U_\Omega \cdot T(K)_b$,  (we recall that $T(K)_b$ normalizes $U_\Omega$);
\item $\widehat{P_\Omega}$ the subgroup of $G(K)$ generated by $U_\Omega$ and $N_\Omega$.
\end{itemize}
Moreover, because $G$ is a (quasi-split) semisimple $K$-group, the group $\widehat{P}_\Omega$ can be realized as the integral points of a suitable model $\mathfrak{G}_\Omega$ of $G$, and we write $\widehat{P}_\Omega = \mathfrak{G}_\Omega^\circ(\mathcal{O}_K)$.
This group is the connected pointwise stabilizer in $G(K)$ of the subset $\Omega \subset X(G,K)$ \cite[4.6.28]{BruhatTits2}.
\end{Not}

From the dual pairing, each relative root $a \in \Phi$ can be realized geometrically in the Euclidean dual space $V^*$.
By \cite[VI.1.4 Prop. 12]{Bourbaki4-6}, there are exactly one or two values for the length of a root if $\Phi$ is reduced; and by \cite[VI.4.14]{Bourbaki4-6} there are three values if $\Phi$ is non-reduced.
We say that a root $a \in \Phi$ is a \textbf{long root} if its length is maximal in its irreducible component, and is a \textbf{short root} otherwise.
More precisely, if $\Phi$ is a reduced non-simply laced root system, the ratio between the length of a long root and the length of a short root is exactly $\sqrt{d'}$ where the integer $d' \in \{1,2,3\}$ has been defined in \ref{not:L:prime} considering the smallest extension of $K$ splitting $G$.

\begin{Prop}
\label{prop:length:splitting:field}
Let $d$, $L'$, $L_d$ as in \ref{not:L:prime}.

(1) If $d=1$, every root $a \in \Phi$ has $L_a = L' = L_d = L_0$ as splitting field (up to isomorphism, in the sense of \ref{def:splitting:field}).

(2) If $d \geq 2$ and $\Phi$ is reduced, every short root has $L'$ as splitting field; every long root has $L_d$ as splitting field.

(3) If $d = 2$ and $\Phi$ is non-reduced, every non-divisible root has $L'$ as splitting field; every divisible root has $L_d$ as splitting field.
\end{Prop}

\begin{proof}

(1) If $d = 1$, then $\Sigma_0 = \Sigma_d = \Sigma_a$ for any root $a \in \Phi$.
Hence, we have the equality of the corresponding fixed fields $L_0 = L_d = L_a = L'$.

Suppose now that $d \geq 2$.
Because $\mathrm{Dyn}(\widetilde{\Delta})$ has a non-trivial symmetry, all the absolute roots have the same length in the geometric realisation in $\widetilde{V}^*$ defined in \ref{not:star:action}.
Let $a$ be a relative root, seen as orbit, which contain several absolute roots. In the geometric realization, the orbit $a$ can be geometrically realized as the orthogonal projection of its absolute roots.
Hence, the length of the orbits having several roots is shorter than that of the orbits having only one root.

Let $a \in \Phi$ be a relative root and let $\alpha \in \widetilde{\Phi}$ be an absolute root so that the relative root $a = \alpha|_S$ is its orbit for the $*$-action.

(2) If $d\geq 2$ and $\Phi$ is reduced.
If $a$ is short, then $\Sigma_0$ fixes $\alpha$ but $\Sigma_d$ does not.
Moreover, we observe that for $d = 6$ (hence $\widetilde{\Phi}$ is of type $D_4$), the stabilizer of $\alpha$ in $\Sigma_d / \Sigma_0 \simeq \mathfrak{S}_3$ has index $3$.
Hence $L_\alpha$ is a separable extension of $L_d$ of degree $3$ if $d \geq 3$ and of degree $2$ otherwise, hence isomorphic to $L'$.
Thus $L' = L_a$.
If $a$ is long, then $\Sigma_d$ is the stabilizer of $\alpha$.
Hence $L_d = L_a$.

(3) If $d = 2$ and $\Phi$ is non-reduced.
If $a$ is divisible, then $a$ is a long root.
Hence $\Sigma_2$ is the stabilizer of $\alpha$.
Thus $L_2 = L_a$.
Otherwise, $a$ is a short root.
Hence $\Sigma_0$ is the stabilizer of $\alpha$.
Thus $L'=L_0 = L_a$.
\end{proof}

\subsubsection{Description of an alcove by its panels} 
\label{sec:description:alcove}

An alcove is the candidate to be a fundamental domain of the action of $G(K)$ on its Bruhat-Tits building $X(G,K)$.

\begin{Def}
\label{def:panel}
A \textbf{panel} is a facet of $X(G,S)$ of codimension $1$.
\end{Def}

We want to describe precisely, thanks to some relative roots and their sets of values, walls bounding a given alcove.
To do this, we may have to consider a dual root system, which appears to be necessary in some ramified cases.

Firstly, we define a dual root system of $\Phi$ by a suitable normalisation of the canonical dual root system in Lie considerations.

\begin{Not}
\label{not:inverse:root:system}
We consider a geometric realization of $\Phi_{\mathrm{nd}}$ in the Euclidean space $\big(V^*,(\cdot | \cdot )\big)$.
For each root $a \in \Phi_{\mathrm{nd}}$, we set $\lambda_a = \frac{\mu^2}{(a|a)} \in \{1,d'\}$ and $a^D = \lambda_a a \in V$ where $\mu$ is the length of a long root, so that $a^D = a$ for any long roots.
The set $\Phi_{\mathrm{nd}}^D = \{ a^D,\ a \in \Phi_{\mathrm{nd}} \}$ is a root system, because it is proportional (by a factor $\frac{\mu^2}{2}$) to the dual root system $\Phi^\vee$ of \cite[VI.1.1 Prop. 2]{Bourbaki4-6}.
In particular, if $\Phi$ is a reduced irreducible root system, then $\Phi^D=\Phi$ if, and only if, it is a simply laced root system (type $A$, $D$, or $E$).
Moreover, by \cite[VI.1.5 Rem.(5)]{Bourbaki4-6}, if $\Delta$ is a basis of $\Phi$, then $\Delta^D = \{a^D,\  a \in \Delta \}$ is a basis of $\Phi_{\mathrm{nd}}^D$.
\end{Not}

Whereas $\Phi^\vee$ and $\Phi^D$ are constructions strictly in terms of Lie theory, we have found it was more convenient to introduce the following root system $\Phi^\delta$ which takes into account the splitting field extensions of root groups.

\begin{Def}\label{def:dual:root:system}
For any non-divisible root $a \in \Phi_{\mathrm{nd}}$, we denote by $\delta_a \in \{1, d'\}$ the order of the quotient group $\Gamma_{L_a} / \Gamma_{L_d}$ (resp. $\Gamma_{L_a} / \Gamma_{L'}$) if $\Phi$ is reduced (resp. non-reduced), by $a^\delta = \delta_a a$ and by $\Phi_{\mathrm{nd}}^\delta = \{ a^\delta,\ a \in \Phi_{\mathrm{nd}} \}$.
We denote by $\Delta^\delta = \{a^\delta,\  a \in \Delta \}$.
We will see below that $\Phi_{\mathrm{nd}}^\delta = \Phi_{\mathrm{nd}}$ or $\Phi^D_{\mathrm{nd}}$.
\end{Def}

\begin{Not}\label{not:highest:dual:root}
In the following, we denote by:
\begin{itemize}
\item $h$ the highest root of $\Phi$ with respect to the chosen basis $\Delta$;
\item $\theta \in \Phi_{\mathrm{nd}}$ the root such that $\theta^\delta$ is the highest root of $\Phi_{\mathrm{nd}}^\delta$ with respect to the basis $\Delta^\delta$.
\end{itemize}
Moreover, if $\Phi$ is non-reduced, we will see below that $\Phi_{\mathrm{nd}}^\delta = \Phi_{\mathrm{nd}}^D = \Phi_{\mathrm{nm}}$, so that $h = 2 \theta$.
\end{Not}

Note that if $a$ is multipliable and $2l \in \Gamma'_{2a}$, it is possible that $\mathcal{H}_{2a,2l} = \mathcal{H}_{a,l}$ be a wall even if $l \not\in \Gamma'_a$.
Moreover, we have $\Gamma_a = \Gamma'_a \cup \frac{1}{2} \Gamma'_{2a}$ in this case.
Otherwise, if $a$ is non-multipliable and non divisible, we have $\Gamma_a = \Gamma'_a$ by Lemma \ref{lem:sets:of:values:non:multipliable:root}.
In fact, the walls of $\mathbb{A}$ are described by the various $a\in \Phi_{\mathrm{nd}}$ and $l \in \Gamma_a$.

According to \cite[4.2.23]{BruhatTits2}, we can classify the scalings to describe the various alcoves for a $K$-simple group $G$.
In a similar way, there exists a classification of (quasi-split) absolutely almost-simple groups over a local field, provided by Tits in \cite[§4]{TitsCorvallis}.
Here, we reduce the discussion to three types of behaviours.

\paragraph*{First case: $\Phi$ is reduced and $L' / L_d$ is unramified.}

These groups are the residually split groups named $A_n$, $B_n$, $C_n$, $D_n$, $E_6$, $E_7$, $E_8$, $F_4$ and $G_2$; and the non-residually split groups named ${^2}A'_{2n-1}$, ${^2}D_{n+1}$, ${^2}E_6$ and ${^3}D_4$ in the Tits tables \cite[4.2, 4.3]{TitsCorvallis}.
These correspond respectively to scalings, classified in \cite[1.4.6]{BruhatTits1}, of type $A_n$, $B_n$, $C_n$, $D_n$, $E_6$, $E_7$, $E_8$, $F_4$ and $G_2$; and $C_n$, $B_n$, $F_4$ and $G_2$.

Let $a$ be a relative root.
Because $\Phi$ is reduced, $\Gamma_a = \Gamma_{L_a}$ by Lemma \ref{lem:sets:of:values:non:multipliable:root}.
Hence, by Proposition \ref{prop:length:splitting:field}, we have $\Gamma_a = \Gamma_{L_d}$.
Because $L'/L_d$ is unramified, we have $\Gamma_{L'} = \Gamma_{L_d}$.
Hence $\Phi^\delta = \Phi$ and $h=\theta$.

In order to simplify notations, we normalize the valuation $\omega$ so that $\Gamma_{L'} = \mathbb{Z} = \Gamma_{L_d}$ and $0^+ = 1$.
By definition of alcoves as connected components, we can define an alcove as the intersection of all the various half-apartments $D(a,l)$ and $D(b,l^+)$ where $a \in \Phi^+$, $b \in \Phi^-$ and $l \in \mathbb{R}^+$.
Because $D(A,l) \subset D(a,l')$ for any $l > l'$, we are in fact considering the finite intersection of all the various half-apartments $D(a,0)$ and $D(b,1)$ where $a \in \Phi^+$ and $b \in \Phi^-$.
We call it \gts{the} fundamental alcove, denoted by $\mathbf{c}_{\mathrm{af}}$.

\begin{tikzpicture}
\draw [thick] (0,-2.1) -- (0,2.1);
\draw [thick] (-1,-2.1) -- (-1,2.1);
\draw [thick] (1,-2.1) -- (1,2);
\draw [thick] (-2,-2.1) -- (-2,2.1);
\draw [thick] (2,-2.1) -- (2,2.1);

\draw [thick] (-2.1,1.097) -- (-0.363,2.1);
\draw [thick] (-2.1,-0.058) -- (1.637,2.1);
\draw [thick] (-2.1,-1.212) -- (2.1,1.212);
\draw [thick] (-1.637,-2.1) -- (2.1,0.058);
\draw [thick] (0.363,-2.1) -- (2.1,-1.097);

\draw [thick] (-2.1,-1.097) -- (-0.363,-2.1);
\draw [thick] (-2.1,0.058) -- (1.637,-2.1);
\draw [thick] (-2.1,1.212) -- (2.1,-1.212);
\draw [thick] (-1.637,2.1) -- (2.1,-0.058);
\draw [thick] (0.363,2.1) -- (2.1,1.097);

\draw [very thick,fill=gray!40] (0,0) -- (1,0.577) -- (0,1.155) -- (0,0);

\fill[pattern color=blue!100, pattern=north east lines]
 (0,2.1) -- (0,0) -- (2.1,1.212) -- (2.1,2.1)-- (0,2.1);
 
\fill[pattern color=red!100, pattern=north west lines]
 (-2.1,-0.058) -- (0,2/1.732) -- (1,1/1.732) -- (1,-2.1) -- (-2.1,-2.1) -- (-2.1,-0.058);

\draw (0.25,1.732/4) node{$\mathbf{c}_\mathrm{af}$};
\draw (-2.1,-2.1) node[above left] {$\displaystyle \bigcap_{b\in \Phi^-} D(b,1)$};
\draw (2.1,2.1) node[below right] {$\displaystyle \bigcap_{a\in \Phi^+} D(a,0)$};

\draw [dashed,very thin,->] (0,0) -- (1/2,1.732/2);
\draw [dashed,very thin,->] (0,0) -- (-1/2,1.732/2);
\draw [dashed,very thin,->] (0,0) -- (1,0);
\draw [dashed,very thin,->] (0,0) -- (-1/2,-1.732/2);
\draw [dashed,very thin,->] (0,0) -- (1/2,-1.732/2);
\draw [dashed,very thin,->] (0,0) -- (-1,0);
\end{tikzpicture}

By \cite[VI.2.2 Prop. 5]{Bourbaki4-6}, its panels are exactly contained inside the walls $\mathcal{H}_{a,0}$, where $a \in \Delta$, and $\mathcal{H}_{-h,1}$.

\begin{Ex}[The apartments and their fundamental alcoves in dimension~$2$]

\noindent
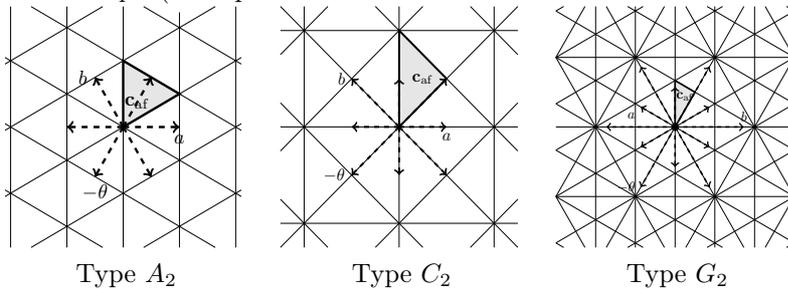
\begin{tabular}{ccc}
\resizebox{3.2cm}{3.2cm}{%
\begin{tikzpicture}
\draw [very thin] (0,-2.1) -- (0,2.1);
\draw [very thin] (-1,-2.1) -- (-1,2.1);
\draw [very thin] (1,-2.1) -- (1,2);
\draw [very thin] (-2,-2.1) -- (-2,2.1);
\draw [very thin] (2,-2.1) -- (2,2.1);

\draw [very thin] (-2.1,1.097) -- (-0.363,2.1);
\draw [very thin] (-2.1,-0.058) -- (1.637,2.1);
\draw [very thin] (-2.1,-1.212) -- (2.1,1.212);
\draw [very thin] (-1.637,-2.1) -- (2.1,0.058);
\draw [very thin] (0.363,-2.1) -- (2.1,-1.097);

\draw [very thin] (-2.1,-1.097) -- (-0.363,-2.1);
\draw [very thin] (-2.1,0.058) -- (1.637,-2.1);
\draw [very thin] (-2.1,1.212) -- (2.1,-1.212);
\draw [very thin] (-1.637,2.1) -- (2.1,-0.058);
\draw [very thin] (0.363,2.1) -- (2.1,1.097);

\draw [very thick,fill=gray!20] (0,0) -- (1,0.577) -- (0,1.155) -- (0,0);
\draw (0.25,1.732/4) node{$\mathbf{c}_\mathrm{af}$};

\draw [dashed,very thick,->] (0,0) -- (1/2,1.732/2);
\draw [dashed,very thick,->] (0,0) -- (-1/2,1.732/2) node[left]{$b$};
\draw [dashed,very thick,->] (0,0) -- (1,0) node[below]{$a$};
\draw [dashed,very thick,->] (0,0) -- (-1/2,-1.732/2) node[below]{$-\theta$};
\draw [dashed,very thick,->] (0,0) -- (1/2,-1.732/2);
\draw [dashed,very thick,->] (0,0) -- (-1,0);
\end{tikzpicture}
} &

\resizebox{3.2cm}{3.2cm}{%
\begin{tikzpicture}
\draw [very thin] (-2,-2.5) -- (-2,2.5);
\draw [very thin] (0,-2.5) -- (0,2.5);
\draw [very thin] (2,-2.5) -- (2,2.5);

\draw [very thin] (-2.5,-2) -- (2.5,-2);
\draw [very thin] (-2.5,0) -- (2.5,0);
\draw [very thin] (-2.5,2) -- (2.5,2);

\draw [very thin] (-2.5,-1.5) -- (-1.5,-2.5);
\draw [very thin] (-2.5,0.5) -- (0.5,-2.5);
\draw [very thin] (-2.5,2.5) -- (2.5,-2.5);
\draw [very thin] (-0.5,2.5) -- (2.5,-0.5);
\draw [very thin] (1.5,2.5) -- (2.5,1.5);

\draw [very thin] (-2.5,1.5) -- (-1.5,2.5);
\draw [very thin] (-2.5,-0.5) -- (0.5,2.5);
\draw [very thin] (-2.5,-2.5) -- (2.5,2.5);
\draw [very thin] (-0.5,-2.5) -- (2.5,0.5);
\draw [very thin] (1.5,-2.5) -- (2.5,-1.5);

\draw [very thick,fill=gray!20] (0,0) -- (1,1) -- (0,2) -- (0,0);
\draw (0.5,1) node[]{$\mathbf{c}_{\mathrm{af}}$};

\draw [dashed, very thick,->] (0,0) -- (-1,1) node[left]{$b$};
\draw [dashed, very thick,->] (0,0) -- (0,1);
\draw [dashed, very thick,->] (0,0) -- (1,1);
\draw [dashed, very thick,->] (0,0) -- (1,0) node[below]{$a$};
\draw [dashed, very thick,->] (0,0) -- (1,-1);
\draw [dashed, very thick,->] (0,0) -- (0,-1);
\draw [dashed, very thick,->] (0,0) -- (-1,-1) node[left]{$-\theta$};
\draw [dashed, very thick,->] (0,0) -- (-1,0);

\end{tikzpicture}
} &

\resizebox{3.2cm}{3.2cm}{%
\begin{tikzpicture}
\draw [very thin] (0,-3) -- (0,3);
\draw [very thin] (-1,-3) -- (-1,3);
\draw [very thin] (-2,-3) -- (-2,3);
\draw [very thin] (1,-3) -- (1,3);
\draw [very thin] (2,-3) -- (2,3);

\draw [very thin] (-1.732,3) -- (1.732,-3);
\draw [very thin] (0.268,3) -- (3,-1.732);
\draw [very thin] (2.268,3) -- (3,1.732);
\draw [very thin] (-0.268,-3) -- (-3,1.732);
\draw [very thin] (-2.268,-3) -- (-3,-1.732);

\draw [very thin] (-3,1.732) -- (3,-1.732);
\draw [very thin] (-3,2.887) -- (3,-0.577);
\draw [very thin] (-1.196,3) -- (3,0.577);
\draw [very thin] (0.804,3) -- (3,1.732);
\draw [very thin] (2.804,3) -- (3,2.887);
\draw [very thin] (3,-2.887) -- (-3,0.577);
\draw [very thin] (1.196,-3) -- (-3,-0.577);
\draw [very thin] (-0.804,-3) -- (-3,-1.732);
\draw [very thin] (-2.804,-3) -- (-3,-2.887);

\draw [very thin] (-3,0) -- (3,0);
\draw [very thin] (-3,1.732) -- (3,1.732);
\draw [very thin] (-3,-1.732) -- (3,-1.732);

\draw [very thin] (-3,-1.732) -- (3,1.732);
\draw [very thin] (-3,-2.887) -- (3,0.577);
\draw [very thin] (-1.196,-3) -- (3,-0.577);
\draw [very thin] (0.804,-3) -- (3,-1.732);
\draw [very thin] (2.804,-3) -- (3,-2.887);
\draw [very thin] (3,2.887) -- (-3,-0.577);
\draw [very thin] (1.196,3) -- (-3,0.577);
\draw [very thin] (-0.804,3) -- (-3,1.732);
\draw [very thin] (-2.804,3) -- (-3,2.887);

\draw [very thin] (-1.732,-3) -- (1.732,3);
\draw [very thin] (0.268,-3) -- (3,1.732);
\draw [very thin] (2.268,-3) -- (3,-1.732);
\draw [very thin] (-0.268,3) -- (-3,-1.732);
\draw [very thin] (-2.268,3) -- (-3,1.732);

\draw [very thick,fill=gray!20] (0,0) -- (0.5,0.866) -- (0,1.155) -- (0,0);
\draw (0.25,0.75) node[]{$\mathbf{c}_{\mathrm{af}}$};

\draw [dashed, very thick,->] (0,0) -- (1.732,0) node[above]{$b$};
\draw [dashed, very thick,->] (0,0) -- (0.866,0.5);
\draw [dashed, very thick,->] (0,0) -- (0.866,1.5);
\draw [dashed, very thick,->] (0,0) -- (0,1) ;
\draw [dashed, very thick,->] (0,0) -- (-0.866,1.5);
\draw [dashed, very thick,->] (0,0) -- (-0.866,0.5) node[below left]{$a$};

\draw [dashed, very thick,->] (0,0) -- (-1.732,0);
\draw [dashed, very thick,->] (0,0) -- (-0.866,-0.5);
\draw [dashed, very thick,->] (0,0) -- (-0.866,-1.5) node[left]{$-\theta$};
\draw [dashed, very thick,->] (0,0) -- (0,-1) ;
\draw [dashed, very thick,->] (0,0) -- (0.866,-1.5);
\draw [dashed, very thick,->] (0,0) -- (0.866,-0.5);
\end{tikzpicture}
} \\
Type $A_2$ &
Type $C_2$ &
Type $G_2$
\end{tabular}

\end{Ex}

\paragraph*{Second case: $\Phi$ is reduced and $L' / L_d$ is ramified.}

These groups are the residually split groups named $B\text{-}C_n$, $C\text{-}B_n$, $F^{I}_4$ and $G^{I}_2$ in the Tits tables \cite[4.2]{TitsCorvallis}.
These correspond respectively to scalings, classified in \cite[1.4.6]{BruhatTits1}, of type $B\text{-}C_n$, $C\text{-}B_n$, $F^{I}_4$ and $G^{I}_2$.

Because $L'/L_d$ is ramified, $d' \in \{2,3\}$, hence $\Phi$ is a non-simply laced root system. 
Moreover, we have $d' \Gamma_{L'} = \Gamma_{L_d}$.
Let $a$ be a relative root.
Because $\Phi$ is reduced, $\Gamma_a = \Gamma_{L_a}$ by Lemma \ref{lem:sets:of:values:non:multipliable:root}.
By Proposition \ref{prop:length:splitting:field},
if $a$ is a long root, $\Gamma_a = \Gamma_{L_d}$;
if $a$ is a short root, $\Gamma_a = \Gamma_{L'}$.
Thus, $\delta_a = \lambda_a$.
Hence $\Phi_{\mathrm{nd}}^\delta = \Phi_{\mathrm{nd}}^D$.

In order to simplify notations, we normalize the valuation $\omega$ so that $\Gamma_{L'} = \mathbb{Z}$.
The intersection of all the various half-apartments $D(a,0)$ and $D(b,0^+)$ where $a \in \Phi^+$ and $b \in \Phi^-$ in exactly an alcove.
If $b \in \Phi^-$ is short, then $\Gamma_b=\Gamma_{L'}$ so that $D(b,0^+) = D(b,1)$;
if $b' \in \Phi^-$ is long, then  $\Gamma_b=\Gamma_{L_d}$ so that $D(b,0^+) =D(b',d')$.
We call it \gts{the} fundamental alcove, denoted by $\mathbf{c}_{\mathrm{af}}$.

Its panels are exactly contained inside the walls $\mathcal{H}_{a,0}$, where $a \in \Delta$, and $\mathcal{H}_{-\theta,1}$.
Indeed, let $a \in \Phi$ and $l \in \mathbb{R}$.
Let $l^D = \delta_a l$ so that for any $x \in \mathbb{A}$:
$$a(x - \mathcal{O}) - l = 0 \Leftrightarrow a^D(x - \mathcal{O}) - l^D= 0$$
By definition, the set $\mathcal{H}_{a,l}$ is a wall of $\mathbb{A}$ if, and only if, $l \in \Gamma_a$; hence if, and only if, $l^D \in \Gamma_{L_d}$.
Thus, the panels of $\mathbf{c}_{\mathrm{af}}$ are contained in the walls $\mathcal{H}_{a^D,l^D}$ described in the first case.
Because the highest root $\theta^D$ is a long root in $\Phi^D$ by \cite[VI.1.8 Prop. 25 (iii)]{Bourbaki4-6}, hence $\theta$ is a short root in $\Phi$ and $\delta_\theta = d'$.

\begin{Rq}
The ramification as the effect of adding some walls in the direction corresponding to short roots.
For instance, if $d=2$ and if the absolute root system $\widetilde{\Phi}$ is of type $A_3$, then the relative root system is of type $C_2$ and we obtain the following picture where we print the \gts{added} walls with dotted lines, and the root system $\Phi^D$ instead of $\Phi$:

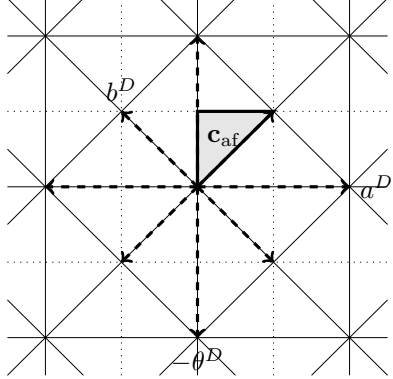
\begin{tikzpicture}
\draw [very thin] (-2,-2.5) -- (-2,2.5);
\draw [dotted] (-1,-2.5) -- (-1,2.5);
\draw [very thin] (0,-2.5) -- (0,2.5);
\draw [dotted] (1,-2.5) -- (1,2.5);
\draw [very thin] (2,-2.5) -- (2,2.5);

\draw [very thin] (-2.5,-2) -- (2.5,-2);
\draw [dotted] (-2.5,-1) -- (2.5,-1);
\draw [very thin] (-2.5,0) -- (2.5,0);
\draw [dotted] (-2.5,1) -- (2.5,1);
\draw [very thin] (-2.5,2) -- (2.5,2);

\draw [very thin] (-2.5,-1.5) -- (-1.5,-2.5);
\draw [very thin] (-2.5,0.5) -- (0.5,-2.5);
\draw [very thin] (-2.5,2.5) -- (2.5,-2.5);
\draw [very thin] (-0.5,2.5) -- (2.5,-0.5);
\draw [very thin] (1.5,2.5) -- (2.5,1.5);

\draw [very thin] (-2.5,1.5) -- (-1.5,2.5);
\draw [very thin] (-2.5,-0.5) -- (0.5,2.5);
\draw [very thin] (-2.5,-2.5) -- (2.5,2.5);
\draw [very thin] (-0.5,-2.5) -- (2.5,0.5);
\draw [very thin] (1.5,-2.5) -- (2.5,-1.5);

\draw [very thick,fill=gray!20] (0,0) -- (1,1) -- (0,1) -- (0,0);
\draw (0.35,0.65) node[]{$\mathbf{c}_{\mathrm{af}}$};

\draw [dashed, very thick,->] (0,0) -- (-1,1) node[above]{$b^D$};
\draw [dashed, very thick,->] (0,0) -- (0,2);
\draw [dashed, very thick,->] (0,0) -- (1,1);
\draw [dashed, very thick,->] (0,0) -- (2,0) node[right]{$a^D$};
\draw [dashed, very thick,->] (0,0) -- (1,-1);
\draw [dashed, very thick,->] (0,0) -- (0,-2) node[below]{$-\theta^D$};
\draw [dashed, very thick,->] (0,0) -- (-1,-1);
\draw [dashed, very thick,->] (0,0) -- (-2,0);
\end{tikzpicture}
\end{Rq}

\paragraph*{Third case: $\Phi$ is non-reduced.}

These groups are named $C\text{-}BC_n$ and ${^2}A'_{2n}$ in the Tits tables \cite[4.2, 4.3]{TitsCorvallis}.
These correspond respectively to scalings, classified in \cite[1.4.6]{BruhatTits1}, of type $C\text{-}BC^{III}_n$ and $C\text{-}BC^{IV}_n$.

Because $\Phi$ is non-reduced, $d = d' = 2$.
In order to simplify notations, we normalize the valuation $\omega$ so that $\Gamma_{L'} = \mathbb{Z}$.
Let $a$ be a non-divisible relative root.
If $a$ is multipliable, by Lemma \ref{lem:sets:of:values:multipliable:root}, we have $\Gamma_a = \frac{1}{2} \Gamma_{L'}$;
if $a$ is non-multipliable, by Lemma \ref{lem:sets:of:values:non:multipliable:root}, and by Proposition \ref{prop:length:splitting:field}, we have $\Gamma_a = \Gamma_{L_a} = \Gamma_{L'}$.
Thus, $\delta_a \Gamma_{a} = \Gamma_{L'}$.

As above, one can see that the intersection of all the various following half-apartments: $D(a,0)$ where $a \in \Phi_{\mathrm{nd}}^+$, $D(b,1)$ where $b \in \Phi_{\mathrm{nd}}^-$ is non-multipliable, and $D(b',\frac{1}{2})$ where $b' \in \Phi^-$ is multipliable, is exactly an alcove.
We call it \gts{the} fundamental alcove, denoted by $\mathbf{c}_{\mathrm{af}}$.
Its panels are exactly contained inside the walls $\mathcal{H}_{a,0}$, where $a \in \Delta$, and $\mathcal{H}_{-\theta,\frac{1}{2}}$.

Indeed, we proceed in the same way as in the previous case, with the reduced root system $\Phi_{\mathrm{nd}}^D$.

\begin{Ex}[$\widetilde{\Phi}$ of type $A_4$ and $\Phi$ of type $BC_2$]~

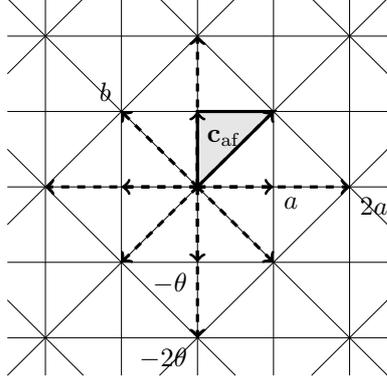
\begin{tikzpicture}
\draw [very thin] (-2,-2.5) -- (-2,2.5);
\draw [very thin] (-1,-2.5) -- (-1,2.5);
\draw [very thin] (0,-2.5) -- (0,2.5);
\draw [very thin] (1,-2.5) -- (1,2.5);
\draw [very thin] (2,-2.5) -- (2,2.5);

\draw [very thin] (-2.5,-2) -- (2.5,-2);
\draw [very thin] (-2.5,-1) -- (2.5,-1);
\draw [very thin] (-2.5,0) -- (2.5,0);
\draw [very thin] (-2.5,1) -- (2.5,1);
\draw [very thin] (-2.5,2) -- (2.5,2);

\draw [very thin] (-2.5,-1.5) -- (-1.5,-2.5);
\draw [very thin] (-2.5,0.5) -- (0.5,-2.5);
\draw [very thin] (-2.5,2.5) -- (2.5,-2.5);
\draw [very thin] (-0.5,2.5) -- (2.5,-0.5);
\draw [very thin] (1.5,2.5) -- (2.5,1.5);

\draw [very thin] (-2.5,1.5) -- (-1.5,2.5);
\draw [very thin] (-2.5,-0.5) -- (0.5,2.5);
\draw [very thin] (-2.5,-2.5) -- (2.5,2.5);
\draw [very thin] (-0.5,-2.5) -- (2.5,0.5);
\draw [very thin] (1.5,-2.5) -- (2.5,-1.5);

\draw [very thick,fill=gray!20] (0,0) -- (1,1) -- (0,1) -- (0,0);
\draw (0.35,0.65) node[]{$\mathbf{c}_{\mathrm{af}}$};

\draw [dashed, very thick,->] (0,0) -- (-1,1) node[above left]{$b$};
\draw [dashed, very thick,->] (0,0) -- (0,1);
\draw [dashed, very thick,->] (0,0) -- (0,2);
\draw [dashed, very thick,->] (0,0) -- (1,1);
\draw [dashed, very thick,->] (0,0) -- (1,0) node[below right]{$a$};
\draw [dashed, very thick,->] (0,0) -- (2,0) node[below right]{$2a$};
\draw [dashed, very thick,->] (0,0) -- (1,-1);
\draw [dashed, very thick,->] (0,0) -- (0,-1) node[below left]{$-\theta$};
\draw [dashed, very thick,->] (0,0) -- (0,-2) node[below left]{$-2\theta$};
\draw [dashed, very thick,->] (0,0) -- (-1,-1);
\draw [dashed, very thick,->] (0,0) -- (-1,0);
\draw [dashed, very thick,->] (0,0) -- (-2,0);
\end{tikzpicture}
\end{Ex}

\subsubsection{Counting alcoves of a panel residue}
\label{sec:counting:alcoves}

Because a maximal pro-$p$ subgroup $P$ fixes an alcove $\mathbf{c}$, it acts on the set of alcoves which are adjacent to $\mathbf{c}$.
We want to describe this set of alcoves.

\begin{Def}
\label{def:panel:residue}
Let $F$ be a panel.
The \textbf{panel residue} with respect to $F$, denoted by $E_{F}$, is the set of the alcoves whose the closure contains $F$.

The \textbf{combinatorial unit ball} centered in $\mathbf{c}$, denoted by $B(\mathbf{c},1)$, is the union of all the panel residues with respect to a panel $F$ in the closure of $\mathbf{c}$.

We say that two alcoves are \textbf{adjacent} if they have a common panel.
\end{Def}

In what follows, we provide a reformulation and a proof of \cite[1.6]{TitsCorvallis}.

\begin{Prop}
\label{prop:first:root:group:quotient}
Let $a \in \Phi$ and $l \in \Gamma_a$.
The group $U_{a,l^+}$ is a normal subgroup of $U_{a,l}$.
We denote by $X_{a,l} = U_{a,l} / U_{a,l^+}$ the quotient group.

If $a$ is non-multipliable, then there exists a canonical $\kappa_{L_{a}}$-vector space structure on $X_{a,l}$ of dimension $1$.

If $a$ is multipliable, then there exists a canonical group homomorphism  $X_{2a,2l} \rightarrow X_{a,l}$;
so that we have the inclusion $[X_{a,l},X_{a,l}] \leq X_{2a,2l}$.
There exists a canonical $\kappa_{L_{a}}$-vector space structure on the quotient group $X_{a,l} / X_{2a,2l}$ of dimension $0$ or $1$.
\end{Prop}

\begin{proof}
Suppose that $a$ is non-multipliable, then $U_a(K)$ is commutative.
Hence $U_{a,l^+}$ is a normal subgroup of $U_{a,l}$ and the quotient group $X_{a,l}$ is commutative.
We define a $\mathcal{O}_{L_a}$-module structure on $X_{a,l}$ by:
$$\forall x \in \mathcal{O}_{L_a},\ \forall y \in L_a\text{ such that }\omega(y) \geq l,\  x \cdot x_a(y) U_{a,l^+} = x_a(xy) U_{a,l^+}$$

For any $x \in \varpi_{L_a} \mathcal{O}_{L_a}$ and any $y \in L_a\text{ such that }\omega(y) \geq l$,
we have $\omega(xy) \geq l^+$,
hence $x X_{a,l} \leq U_{a,l^+}$.
This provides a $\kappa_{L_a} = \mathcal{O}_{L_a} / \varpi_{L_a} \mathcal{O}_{L_a}$-vector space structure on $X_{a,l}$.
We check that this vector space is of dimension $1$: for any $y,y' \in L_a$ such that $\omega(y) = \omega(y') = l$,
since $y$ is invertible, we have $x = y^{-1} y' \in \mathcal{O}_{L_a}$.
Moreover, such elements $y,y'$ exist by definition of $\Gamma_{L_a}$.

Suppose now that $a$ is multipliable.
By Lemma \ref{lem:non:reduced:derived:valued:root:group} applied to $l,l^+ \in \Gamma_a$,
we get that $U_{a,l^+}$ is a normal subgroup of $U_{a,l}$.

The normal subgroup $U_{2a,{2l}^+}$ of $U_{2a,2l}$ is the kernel of the canonical group homomorphism $U_{2a,2l} \rightarrow X_{a,l}$.
Hence we deduce a quotient group homomorphism $X_{2a,2l} \rightarrow X_{a,l}$.
Passing to the quotient the formula of Lemma \ref{lem:non:reduced:derived:valued:root:group}, we get $[X_{a,l},X_{a,l}] \leq X_{2a,2l}$.

In particular, the group $X_{a,l} / X_{2a,2l}$ is commutative.
There exist an $\mathcal{O}_{L_{a}}$-module structure given by:
\begin{multline*}
\forall x \in \mathcal{O}_{L_{a}},\ \forall (y,y') \in H(L_a,L_{2a}) \text{ such that }\omega(y') \geq 2l,\\
x \cdot x_a(y,y') U_{a,l^+} U_{2a,2l} = x_a(xy,x {^\tau}x y') U_{a,l^+} U_{2a,2l}
\end{multline*}

For any $x \in \varpi_{L_{a}} \mathcal{O}_{L_{a}}$ and any $(y,y') \in H(L_a,L_{2a})\text{ such that }\omega(y') \geq 2l$, we have $\omega(x {^\tau}x y') \geq 2 (l^+)$.
This defines a $\kappa_{L_{a}}$-vector-space structure on $X_{a,l} / X_{2a,2l}$.
This vector-space is of dimension at most $1$.
Indeed, if there exist elements $(y,y'),(z,z') \in H(L_a,L_{2a})$ such that $\omega(y') = \omega(z') = 2l$,
then we can set $x = y^{-1} z \in \mathcal{O}_{L_a}$ because $y$ is invertible.
Hence, we have $x_a(z,z')\in x \cdot x_a(y,y') U_{2a,2l}$.
\end{proof}

If $a$ is a non-multipliable root, we set $X_{2a,2l} = 0$ and $\kappa_{L_{2a}} = \kappa_{L_a}$.
Hence, the dimension $d(a,l)=\dim_{\kappa_{L_{2a}}} X_{a,l}/X_{2a,2l}$ has a sense for any root $a \in \Phi$.

\begin{Rq}
Let $F$ be a panel contained in a wall $\mathcal{H}_{a,l}$ corresponding to an affine root $\theta(a,l)$.
Denote $q = \mathrm{Card}(\kappa_{L_{2a}})$.
The panel residue $E_F$ contains $1 + \mathrm{Card}(X_{a,l}) = 1 + q^{d(\frac{a}{2},\frac{l}{2}) + d(a,l) + d(2a,2l)}$ elements.
This is a consequence of Lemma \ref{lem:panel:residue}.
\end{Rq}

The following lemma states that the affine root systems defined in \cite[6.2.6]{BruhatTits1} and in \cite[1.6]{TitsCorvallis} are the same.

\begin{Lem}
\label{lem:equivalence:affine:root:systems}
Let $a \in \Phi$ be a root and $l \in \mathbb{R}$.
Then $d(a,l) > 0$ if, and only if, $l \in \Gamma'_a$.
\end{Lem}

\begin{proof}~

$\begin{array}{rcl}
l \in \Gamma'_a

& \Leftrightarrow & \exists \mathbf{u} \in U_a(K),\ \varphi_a(\mathbf{u}) = l = \sup \varphi_a(\mathbf{u} U_{2a}(K))\\

& \Leftrightarrow & \exists \mathbf{u} \in U_a(K),\ \varphi_a(\mathbf{u}) = l \text{ and } \forall \mathbf{u''} \in U_{2a}(K),\ \varphi_a(\mathbf{u}\mathbf{u''}) < l^+\\

& \Leftrightarrow & U_{a,l} \neq U_{a,l^+} \text{ and }\exists \mathbf{u} \in U_{a,l},\ \forall \mathbf{u}'' \in U_{2a}(K),\ \mathbf{u}\mathbf{u}'' \not\in U_{a,l^+}\\

& \Leftrightarrow & X_{a,l} \neq 0 \text{ and } X_{a,l} \neq X_{2a,2l}\\

& \Leftrightarrow & d(a,l) \neq 0
\end{array}$
\end{proof}

This affine root system is an affinisation of the spherical root system.
It can be obtained by adding affine reflections corresponding to elements $\mathbf{m}(u) = u' u u''$ where for any $u \in U_a(K) \setminus \{1\}$, there exist $u', u'' \in U_{-a}{K}$ uniquely determined such that $\mathbf{m}(u) \in \mathcal{N}_G(S)(K)$.

\subsection{Action on a combinatorial unit ball}
\label{sec:action:on:ball}

We consider a maximal pro-$p$-subgroup $P = P^+_\mathbf{c}$ of $G(K)$.
For any $a \in \Phi$, if there exists a wall $\mathcal{H}_{a,l}$ bounding $\mathbf{c}$, we denote by $F_{\mathbf{c},a}$ the panel of $\mathbf{c}$ contained in $\mathcal{H}_{a,l}$.
Let $E_{\mathbf{c},a} = E_{F_{\mathbf{c},a}} $ be the panel residue of $F_{\mathbf{c},a}$.
We want to study the action of the derived group and of the Frattini subgroup of $P$ on the Bruhat-Tits building $X(G,K)$ of $G$ over $K$.
For this, we consider the action, on each set $E_{\mathbf{c},a}$, of the various valued root groups $U_{a,\mathbf{c}}$ and of the group $T(K)_b^+$.

\begin{Lem}
\label{lem:adjacent:alcoves}
Let $\mathbf{c}_1$ and $\mathbf{c}_2$ be two adjacent alcoves of the apartment $\mathbb{A}$ along a wall directed by a root $a \in \Phi$.
If $b \in \Phi \setminus \mathbb{R} a$,
then $f'_{\mathbf{c}_1}(b) = f'_{\mathbf{c}_2}(b)$ where $f'$ is defined in \ref{not:Omega}.
In particular, we have $U_{b,\mathbf{c}_1} = U_{b,\mathbf{c}_2}$.
\end{Lem}

\begin{proof}
In order that $f'_{\mathbf{c}_1}(b) \neq f'_{\mathbf{c}_2}(b)$, it is necessary and sufficient that there exists a wall directed by $b$ separating the alcoves $\mathbf{c}_1$ and $\mathbf{c}_2$ in two opposed half-apartments.
The alcoves $\mathbf{c}_1$ and $\mathbf{c}_2$ contain a panel contained in a wall directed by $a$.
This wall is the only one separating the alcoves in two opposed half-apartments.
Hence, if $f'_{\mathbf{c}_1}(b) \neq f'_{\mathbf{c}_2}(b)$, then $a$ and $b$ are collinear.
\end{proof}

\begin{Prop}
\label{prop:favourable:geometric:description}
Let $a \in \Phi = \Phi(G,S)$ be a relative root such that there exists a wall $\mathcal{H}_{a,l}$ bounding $\mathbf{c}$.
If $a$ is non-multipliable or if the quadratic extension $L_{a} / L_{2a}$ is ramified, then the Frattini subgroup $\mathrm{Frat}(P)$ fixes $E_{\mathbf{c},a}$ pointwise.

As a consequence, if $\Phi$ is a reduced root system or if the extension $L/L_d$ is ramified,
then $\mathrm{Frat}(P)$ fixes pointwise the simplicial closure $\mathrm{cl}(B(\mathbf{c},1))$ of the combinatorial unit ball.

In general, denoting by $Q_a$ the pointwise stablizer of $E_{\mathbf{c},a}$, we have the group inclusion $\mathrm{Frat}(P) \subset Q_a U_{2a, \mathbf{c}}$.
\end{Prop}

The rest of this section consists in proving the above proposition.

Let $\mathbf{c}'$ be an alcove of $\mathbb{A}$ adjacent to $\mathbf{c}$.
In particular, we have $\mathbf{c}' \in B(\mathbf{c},1)$.
Write $a' + r'$, with $a' \in \Phi$ and $r' \in \Gamma_{a'}$, the affine root directing the wall separating the alcoves $\mathbf{c}$ and $\mathbf{c}'$.
If $a'$ is divisible, we set $a = \frac{1}{2} a'$ and $r = \frac{1}{2} r'$. Remark that we still have $r \in \Gamma_a$ but $a + r$ may or may not be an affine root according to $r$ is an element of $\Gamma_a'$ or not.
Otherwise, we set $a = a'$ and $r = r'$.
We also have the following definition of $r$ by the equality $r = f_\mathbf{c}(a) = f'_\mathbf{c}(a)$ by \cite[7.7]{Landvogt}.
Up to exchanging $a$ and $-a$, one can assume that $f_{\mathbf{c}'}(a) = f_{\mathbf{c}}(a)^+ > f_{\mathbf{c}}(a)$ and that $f_{\mathbf{c}'}(-a) < f_{\mathbf{c}}(-a) = f_{\mathbf{c}'}(-a)^+$.

The group $P$ acts on the finite set of alcoves $E_{\mathbf{c},a}$ and fixes $\mathbf{c}$.
Hence, it acts on the set of alcoves $E'_{\mathbf{c},a} = E_{\mathbf{c},a} \setminus \{\mathbf{c}\}$.
Denote by $Q_a$ the kernel of this action.
We will show that the quotient group $P / Q_a$ is isomorphic to a subgroup of $U_{a,r} / U_{a,r^+}$.

\begin{Lem}
\label{lem:panel:residue}
The group $U_{a,\mathbf{c}}$ acts transitively on the set $E'_{\mathbf{c},a}$.
\end{Lem}

\begin{proof}
By construction of the building, the subgroup $P_\mathbf{c}$ acts transitively on the set of apartments containing $\mathbf{c}$ \cite[9.7 (i)]{Landvogt}.
Because the action preserves the type of facets, we obtain $E_{\mathbf{c},a} = P_\mathbf{c} \cdot \mathbf{c}'$.

Write $P_\mathbf{c} = U_{a,\mathbf{c}} \cdot \prod_{b \in \Phi^+_{\mathrm{nd}} \setminus (a)} U_{b,\mathbf{c}} \cdot U_{-\Phi^+,\mathbf{c}}\cdot T(K)_b$ \cite[7.1.8]{BruhatTits1}.
The group $T(K)_b$ fixes $A$ pointwise \cite[9.8]{Landvogt}, hence it also fixes $\mathbf{c}'$.
For any $b \in \Phi \setminus \mathbb{R} a$, by Lemma \ref{lem:adjacent:alcoves} we have $U_{b,\mathbf{c}} = U_{b,\mathbf{c}'}$.
Hence $U_{b,\mathbf{c}}$ fixes $\mathbf{c}'$.
Since we assumed that $f_{\mathbf{c}'}(-a) < f_{\mathbf{c}}(-a)$, we have $U_{-a,\mathbf{c}} \subset U_{-a,\mathbf{c}'}$.
Hence $U_{-a,\mathbf{c}}$ fixes $\mathbf{c}'$.
As a consequence $E'_{\mathbf{c},a} = U_{a,\mathbf{c}} \cdot \mathbf{c}'$, because the valued root groups $U_{b,\mathbf{c}}$ and the group $T(K)_b$ fix $\mathbf{c}'$.
\end{proof}

\begin{Lem}
\label{lem:element:fixing:panel:residue}
Let $g \in P$ be an element fixing $\mathbf{c}'$.
If $[v,g]$ fixes $\mathbf{c}'$ for any $v \in U_{a,\mathbf{c}}$,
then $g$ fixes $E_{\mathbf{c},a}$.
\end{Lem}

\begin{proof}
Let $\mathbf{c}'' \in E'_{\mathbf{c},a}$.
By Lemma \ref{lem:panel:residue}, there exists an element $v \in U_{a,\mathbf{c}}$ such that $\mathbf{c}'' = v \mathbf{c}'$.
We do the following computation:
$$\begin{array}{rcll}
g \cdot \mathbf{c}''
& = & g v \cdot \mathbf{c}' & \\
& = & v [v^{-1},g] g \cdot \mathbf{c}' & \\
& = & v [v^{-1},g]\cdot \mathbf{c}' & \text{ because } g \text{ fixes } \mathbf{c}'\\
& = & v \mathbf{c}' & \text{ because } [v^{-1},g] \text{ fixes } \mathbf{c}'\\
& = & \mathbf{c}''
\end{array}$$
Since this is true for any $\mathbf{c}'' \in E'_{a,\mathbf{c}}$, we conclude that $g$ fixes $E_{a,\mathbf{c}}$.
\end{proof}

Hence, to show that $g \in [P,P]$ fixes $E_{\mathbf{c},a}$,
it suffices to verify that $[U_{a,\mathbf{c}},g]$ fixes $\mathbf{c}'$.
We are reduced to compute commutators.
Recall that the group $U_{a,f_{\mathbf{c}}(a)^+} = U_{a,\mathbf{c}'}$ fixes $\mathbf{c}'$.

\begin{Lem}\label{lem:groups:fixing:panel:residue}
The following groups:
\begin{enumerate}
\item \label{item:root:group:fixes:panel:residue}
$U_{a,f_{\mathbf{c}}(a)^+}$

\item \label{item:bounded:torus:fixes:panel:residue}
$T(K)_b^+$

\item \label{item:non:collinear:root:group:fixes:panel:residue}
$U_{b,\mathbf{c}}$ where $b \in \Phi \setminus \mathbb{R} a$

\item \label{item:opposite:root:group:fixes:panel:residue}
$U_{-a,\mathbf{c}}$
\end{enumerate}
fix the panel residue $E_{\mathbf{c},a}$.
\end{Lem}

\begin{proof}
(\ref{item:root:group:fixes:panel:residue}) Let $u \in U_{a,f_{\mathbf{c}}(a)^+}$.
Then $u$ fixes $\mathbf{c}'$.
Let $v \in U_{a,\mathbf{c}}$.

If $a$ is non-multipliable,
then $[v,u] = 1$ because the root group $U_a(K)$ is commutative.

If $a$ is multipliable,
by Lemma \ref{lem:non:reduced:derived:valued:root:group}, we know that $[v^{-1},u] \in U_{2a,\lceil f_{\mathbf{c}}(a)^+ \rceil + \lceil f_{\mathbf{c}}(a) \rceil}$.
Since $\lceil f_{\mathbf{c}}(a)^+ \rceil + \lceil f_{\mathbf{c}}(a) \rceil > 2 f_{\mathbf{c}}(a)$,
we deduce that $[v^{-1},u] \in U_{a,f_\mathbf{c}(a)^+} = U_{a,f_{\mathbf{c}'}(a)}$ fixes $\mathbf{c}'$.

Applying Lemma \ref{lem:element:fixing:panel:residue}, we obtain that $u$ fixes $E_{\mathbf{c},a}$. 

(\ref{item:bounded:torus:fixes:panel:residue}) Let $t \in T(K)_b^+$.
The element $t$ fixes $\mathbf{c}'$ because $T(K)_b$ fixes the apartment $A$.
By Lemmas \ref{lem:commutator:relation:reduced:case:torus:unipotent} and \ref{lem:commutation:non:reduced:torus:root:group}, we know that $[T(K)_b^+,U_{a,\mathbf{c}}]\subset U_{a,f_{\mathbf{c}}(a)^+} = U_{a,\mathbf{c}'}$.
Hence $[v,t] \in U_{a,\mathbf{c}'}$ fixes $\mathbf{c}'$ for any $v \in U_{a,\mathbf{c}}$.
We deduce from (\ref{item:root:group:fixes:panel:residue}) that $T(K)_b^+$ fixes $E_{\mathbf{c},a}$.

(\ref{item:non:collinear:root:group:fixes:panel:residue}) Let $g \in U_{b,\mathbf{c}}$ and $v \in U_{a,\mathbf{c}}$.
By Lemma \ref{lem:adjacent:alcoves}, we get $U_{b,\mathbf{c}} = U_{b,\mathbf{c}'}$.
Hence $g \cdot \mathbf{c}' = \mathbf{c}'$.
By quasi-concavity of the functions $f'$ applied in the case where $a$ and $b$ are not collinear, we get by \cite[4.5.10]{BruhatTits2}:
$$[v^{-1},g] \in \prod_{m,n \in \mathbb{N}^*,\ ma + nb \in \Phi} U_{ma+nb,f'_\mathbf{c}(ma+nb)} $$
Applying again Lemma \ref{lem:adjacent:alcoves},
we get $U_{ma+nb,\mathbf{c}} = U_{ma+nb,\mathbf{c}'}$.
Thus $[v,g]$ fixes $\mathbf{c}'$ for any $v$,
hence, by Lemma \ref{lem:element:fixing:panel:residue}, the element $g$ fixes $E_{\mathbf{c},a}$.

(\ref{item:opposite:root:group:fixes:panel:residue}) Let $u \in U_{-a,\mathbf{c}}$ and $v \in U_{a,\mathbf{c}}$.
Since $f_{\mathbf{c}'}(-a) < f_{\mathbf{c}}(-a)$, we get $U_{-a,\mathbf{c}} \subset U_{-a,\mathbf{c}'}$.
Hence $u$ fixes $\mathbf{c}'$.

According to whether $a$ is multipliable or not, we know that $[v,u] \subset U_{-a,f_{\mathbf{c}}(-a)^+} T(K)_b^+ U_{a,f_{\mathbf{c}}(a)^+}$, by applying either Lemma \ref{lem:commutation:non:reduced:opposite:root:groups} or Lemma \ref{lem:commutator:relation:reduced:case:opposite:unipotent}.
The groups $U_{a,f_{\mathbf{c}}(a)^+}$, $T(K)_b^+$, and $U_{-a,f_{\mathbf{c}}(-a)^+} \subset U_{-a,f_{\mathbf{c}}(-a)}$ fix $\mathbf{c}'$.
Thus, the commutator $[v,u]$ fixes $\mathbf{c}'$ because it can be written as the product of three such elements.
Applying lemma \ref{lem:element:fixing:panel:residue}, we conclude that $u$ fixes $E_{\mathbf{c},a}$.
\end{proof}

\begin{proof}[Proof of Proposition \ref{prop:favourable:geometric:description}]
We keep notations introduced below Proposition \ref{prop:favourable:geometric:description}.
In particular, $a$ is a root such that there exists a wall $\mathcal{H}_{a,l}$ bounding the alcove $\mathbf{c} \subset \mathbb{A}$;
the alcove $\mathbf{c}' \in \mathbb{A}$ has the panel $F_{\mathbf{c},a}$ in common with $\mathbf{c}$.
We have the equalities $f'_\mathbf{c}(a)^+ = f'_{\mathbf{c}'}(a) = f'_{\mathbf{c} \cup \mathbf{c}'}(a)$.
Hence $U_{a,f_{\mathbf{c}}(a)^+} = U_{a,\mathbf{c} \cup \mathbf{c}'}$.
For any root $b \in \Phi_{\mathrm{nd}} \setminus \mathbb{R} a$, by Lemma \ref{lem:adjacent:alcoves}, we get $f'_\mathbf{c}(b) = f'_{\mathbf{c}'}(b) = f'_{\mathbf{c} \cup \mathbf{c}'}(b)$.
Hence $U_{b,f_{\mathbf{c}}(b)} = U_{b,\mathbf{c} \cup \mathbf{c}'}$.
Finally, because we have assumed $f'_{\mathbf{c}'}(-a) < f'_{\mathbf{c}}(-a)$, we get the equality of groups
$U_{-a,\mathbf{c} \cup \mathbf{c}'} = U_{-a,f'_{\mathbf{c}}(-a)} \cap U_{-a,f'_{\mathbf{c}'}(-a)} = U_{-a,\max(f'_{\mathbf{c}}(-a),f'_{\mathbf{c}'}(-a))} = U_{-a, \mathbf{c}}$.
From this, we deduce the equality of groups:
$$U_{a,f_\mathbf{c}(a)^+} \left(\prod_{b \in \Phi_\mathrm{nd} \setminus \{a\}} U_{b,\mathbf{c}}\right) T(K)_b^+ U_{-\Phi^+,\mathbf{c}}
= U_{\Phi^+, \mathbf{c} \cup \mathbf{c}'} T(K)_b^+ U_{-\Phi^+, \mathbf{c} \cup \mathbf{c}'}$$

We denote this group by $P_{\mathbf{c} \cup \mathbf{c}'}^+$ because one could show (as in \cite[3.2.9]{Loisel-maximaux}) that it is the (unique because of simply connectedness assumption on $G$) maximal pro-$p$ subgroup of the pointwise stabilizer in $G(K)$ of $\mathbf{c} \cup \mathbf{c}'$.

By Lemma \ref{lem:groups:fixing:panel:residue}, the subgroup $Q_a$ contains the subgroup $P_{\mathbf{c} \cup \mathbf{c}'}^+$.
Firstly, we prove that $P_{\mathbf{c} \cup \mathbf{c}'}^+$ is a normal subgroup of $P$.
We can write $P = U_{a,\mathbf{c}} P_{\mathbf{c} \cup \mathbf{c}'}^+$.
We have the following group inclusions:
\begin{itemize}
\item $[U_{a,\mathbf{c}},U_{a,f_\mathbf{c}(a)^+}] \subset U_{a,f_\mathbf{c}(a)^+} \subset P_{\mathbf{c} \cup \mathbf{c}'}^+$ by Lemma \ref{lem:non:reduced:derived:valued:root:group} or commutativity according to whether the root $a$ is multipliable or not;
\item $[U_{a,\mathbf{c}},T(K)_b^+] \subset U_{a,f_\mathbf{c}(a)^+} \subset P_{\mathbf{c} \cup \mathbf{c}'}^+$ by Lemma \ref{lem:commutation:non:reduced:torus:root:group} or \ref{lem:commutator:relation:reduced:case:torus:unipotent};
\item $[U_{a,\mathbf{c}},U_{-a,\mathbf{c}}] \subset U_{a,f_\mathbf{c}(a)^+} T(K)_b^+ U_{-a,f_\mathbf{c}(-a)^+} \subset P_{\mathbf{c} \cup \mathbf{c}'}^+$ by Lemma \ref{lem:commutation:non:reduced:opposite:root:groups} or \ref{lem:commutator:relation:reduced:case:opposite:unipotent};
\item $[U_{a,\mathbf{c}},U_{b,c}] \subset P_{\mathbf{c} \cup \mathbf{c}'}^+$ for any $b \in \Phi_{\mathrm{nd}} \setminus \mathbb{R} a$ by quasi-concavity \cite[4.5.10]{BruhatTits2}, as in proof of Lemma \ref{lem:groups:fixing:panel:residue} (\ref{item:non:collinear:root:group:fixes:panel:residue}).
\end{itemize}
Hence, $P_{\mathbf{c} \cup \mathbf{c}'}^+$ is a normal subgroup of $P$ and the quotient $P / P_{\mathbf{c} \cup \mathbf{c}'}^+$ is isomorphic to $U_{a,f_{\mathbf{c}}(a)} / U_{a,f_{\mathbf{c}}(a)^+} = X_{a,f_{\mathbf{c}}(a)}$.
Secondly, $Q_a$ is a normal subgroup of $P$ as the kernel of the action of $P$ on $E_{\mathbf{c},a}$.
Hence, the quotient group $P / Q_a$ is a subgroup of $X_{a,f_{\mathbf{c}}(a)}$.

We define a subgroup $Q'_a$ by $Q'_a = Q_a U_{2a, 2f_{\mathbf{c}}(a)}$ if $a$ is multipliable, $L_a / L_{2a}$ is ramified and $f_{\mathbf{c}}(a) \in \Gamma'_a$; and by $Q'_a = Q_a$ otherwise.
We show that the quotient group $P / Q'_a$ can be endowed with a vector space structure.

Firstly, assume that $a$ is non-multipliable or that $L_a / L_{2a}$ is ramified.
Then, by Proposition \ref{prop:first:root:group:quotient}, we know that the quotient group $P / Q'_a = X_{a,f_\mathbf{c}(a)}$ is a  $\kappa_{L_a}$-vector space (of dimension $1$).

Secondly, assume that $a$ is multipliable, that the extension $L_a / L_{2a}$ is unramified and that $f'_{\mathbf{c}}(a) \not\in \Gamma'_a$.
Then, by Proposition \ref{prop:first:root:group:quotient}, we know that $X_{a,f'_{\mathbf{c}}(a)} = X_{2a,2f'_{\mathbf{c}}(a)}$ is a $\kappa_{L_{2a}}$-vector space of dimension $1$ because the quotient space $X_{a,f'_{\mathbf{c}}(a)} / X_{2a,2f'_{\mathbf{c}}(a)}$ is zero by Lemma \ref{lem:equivalence:affine:root:systems}.
Hence $P / Q'_a = X_{a,f'_{\mathbf{c}}(a)}$ is a vector space.

Finally, assume that $a$ is multipliable, that $L_a / L_{2a}$ is unramified and that $f'_{\mathbf{c}}(a) \in \Gamma'_a$.
Then, by Proposition \ref{prop:first:root:group:quotient}, we know that $P / Q'_a \simeq X_{a,f'_{\mathbf{c}}(a)} / X_{2a,2f'_{\mathbf{c}}(a)}$ is a $\kappa_{L_a}$-vector space of dimension $1$.

As a consequence,
on the one hand, the group $P / Q'_a$ is commutative;
hence $[P,P] \subset Q_a'$.
On the other hand, the group $P / Q'_a$ is of exponent $p$;
hence $P^p \subset Q_a'$.
We get $P^p [P,P] \subset Q'_a$.
Because $G(K)$ acts continuously on $X(G,K)$,
the group $Q_a$ is a closed subgroup of $P$ as the kernel of the action of $P$ on $E_{\mathbf{c},a}$.
Moreover, the group $Q_a U_{2a,2f_{\mathbf{c}}(a)}$ is still closed.
Hence $\mathrm{Frat}(P) = \overline{P^p[P,P]} \subset Q'_a$.

If $\Phi$ is a reduced root system or if the extension $L'/L_d$ is ramified,
then for any root $a \in \Phi$ corresponding to a panel of $\mathbf{c}$, we get that $\mathrm{Frat}(P)$ fixes $E_{\mathbf{c},a}$ pointwise and so it fixes the combinatorial ball of radius $1$ centered in $\mathbf{c}$, denoted by $B(\mathbf{c},1)$, which is the union of all the $E_{\mathbf{c},a}$.
By continuity of the action, the group $\mathrm{Frat}(P) = \overline{P^p [P,P]}$ fixes pointwise the simplicial closure of $B(\mathbf{c},1)$.
\end{proof}

\begin{Rq}
\label{remark:action:bounded:torus}
Though the bounded torus $T(K)_b$ fixes pointwise the apartment $A$, its action on the $1$-neighbourhood of this apartement is, in general, non-trivial.
For instance, assume that $\Phi$ is a reduced root system and choose a spherical root $a \in \Phi$ directing a wall bordering the alcove $\mathbf{c}$.
The action of $T(K)_b$ on $E_{\mathbf{c},a}$ corresponds to the action of a subgroup of $\kappa_{L_a}^{\times 2} \subset \kappa_{L_a}^\times$.
The useful term of an element $t \in T(K)_b$ to describe its action on the set of alcoves $E_{\mathbf{c},a} \setminus \{\mathbf{c}'\}$ is $a(t) / \varpi_{L_a} \mathcal{O}_{L_a} \in \kappa_{L_a}^{\times 2}$.
Indeed, let $\mathbf{c}'' \in E_{\mathbf{c},a} \setminus \{\mathbf{c}'\}$ and write it $\mathbf{c}'' = x_a(x) \cdot \mathbf{c}'$ where $\omega(x) = f'_{\mathbf{c}}(a)$.
Then $t \cdot \mathbf{c}'' = t x_a(x) t^{-1} \cdot \mathbf{c}' = x_a(a(t)x) \cdot \mathbf{c}'$.
\end{Rq}

\begin{Cor}[of Proposition \ref{prop:favourable:geometric:description}]
\label{cor:majoration:frattini}
For any non divisible relative root $a \in \Phi_{\mathrm{nd}}$,
\begin{itemize}
\item if $a \not\in \Delta \cup \{ -\theta \}$, we set $V_{a,\mathbf{c}} = U_{a,\mathbf{c}}$;
\item if $a \in \Delta \cup \{ -\theta \}$ is non-multipliable, we set $V_{a,\mathbf{c}} = U_{a,f_{\mathbf{c}}(a)^+}$;
\item if $a \in \Delta \cup \{ -\theta \}$ and if $a$ is multipliable, and either $L_a / L_{2a}$ is unramified or $f'_{\mathbf{c}}(a) \not\in \Gamma'_a$, we set $V_{a,\mathbf{c}} = U_{a,f_{\mathbf{c}}(a)^+}$;
\item if $a \in \Delta \cup \{ -\theta \}$ and if $a$ is multipliable, the extension $L_a / L_{2a}$ is ramified and $f'_{\mathbf{c}}(a) \in \Gamma'_a$, we set $V_{a,\mathbf{c}} = U_{a,f_{\mathbf{c}}(a)^+} U_{2a,2f_{\mathbf{c}}(a)} = U_{a,f_{\mathbf{c}}(a)^+} U_{2a,\mathbf{c}}$.
\end{itemize}

 We have the following:
$$\mathrm{Frat}(P) \leq \prod_{a \in \Phi^-_{\mathrm{nd}}} V_{a,\mathbf{c}} \cdot
T(K)_b^+ \cdot
\prod_{a \in \Phi^+_{\mathrm{nd}} } V_{a,\mathbf{c}} = T(K)_b^+ \prod_{a \in \Phi_{\mathrm{nd}}} V_{a,\mathbf{c}}$$
\end{Cor}

\begin{proof}
Since $\mathrm{Frat}(P) \subset P$, it suffices to check that $\mathrm{Frat}(P) \cap U_a(K) \subset V_{a,\mathbf{c}}$ for any $a \in \Delta \cup \{-\theta\}$.
Let $a \in \Delta \cup \{-\theta\}$.
By Proposition \ref{prop:favourable:geometric:description}, we have the inclusion $\mathrm{Frat}(P) \subset Q_a U_{2a,\mathbf{c}}$ when $a$ is multipliable, the extension $L_a / L_{2a}$ is ramified and $f'_{\mathbf{c}}(a) \in \Gamma'_a$;
we have the inclusion $\mathrm{Frat}(P) \subset Q_a$ otherwise.
In particular, $\mathrm{Frat}(P) \cap U_a(K) \subset V_{a,\mathbf{c}}$.
\end{proof}

\begin{Prop}
\label{prop:combinatorial:ball:fixator}
We assume that $\Phi$ is a reduced root system.
The group $Q = T(K)_b^+ \prod_{a \in \Phi} V_{a,\mathbf{c}}$ is the maximal pro-$p$ subgroup of the pointwise stabilizer in $G(K)$ of $\mathrm{cl}(B(\mathbf{c},1))$.
\end{Prop}

\begin{proof}
Denote by $\mathrm{cl}\left(B(\mathbf{c},1)\right)$ the simplicial closure of the combinatorial ball of radius $1$.
Set $\Omega = \mathrm{cl}(B(\mathbf{c},1)) \cap \mathbb{A}$.
Denote by $\widehat{P}_{B(\mathbf{c},1)}$ (resp. $\widehat{P}_{\Omega}$) the pointwise stabilizer in $G(K)$ of $\mathrm{cl}\left(B(\mathbf{c},1)\right)$ (resp. $\Omega$).
By \cite[9.3 and 8.10]{Landvogt}, we can write
$\widehat{P}_{\Omega} = T(K)_b \prod_{a\in \Phi} U_{a,\Omega}$.

By Proposition \ref{prop:favourable:geometric:description}, we get that $Q$ fixes $\mathrm{cl}(B(\mathbf{c},1))$ pointwise.
Let $g \in \widehat{P}_{B(\mathbf{c},1)} \subset \widehat{P}_{\Omega}$.
Write $g = t \prod_{a \in \Phi} u_a$ where $t \in T(K)_b$ and $u_a \in U_{a,\Omega} = V_{a,\mathbf{c}}$.
By Lemma \ref{lem:groups:fixing:panel:residue}, we know that $u_a$ fixes pointwise $\mathrm{cl}(B(\mathbf{c},1))$.

Let $t \in T(K)_b$ fixing pointwise $\mathrm{cl}(B(\mathbf{c},1))$.
Let $a$ be a root corresponding to a panel of $\mathbf{c}$.
By Lemma \ref{lem:panel:residue}, we write the orbit $E_{\mathbf{c},a}' = U_{a,\mathbf{c}} \mathbf{c'}$.
For any $u \in U_{a,\mathbf{c}}$, the computation $u \cdot \mathbf{c}' = t u \cdot \mathbf{c}' = [t,u] u \mathbf{c}'$ shows that $[t,u] \in V_{a,\mathbf{c}}$.
By Lemma \ref{lem:commutator:relation:reduced:case:torus:unipotent}, we get $a(t) \equiv 1 \mathrm{mod}\:\varpi$.

Because this equality is true for any $a \in \Delta$, we get $t \in T' = {\prod_{a \in \Delta} \widetilde{a}(\pm 1 + \mathfrak{m}_{L_a})}$.
Hence $\widehat{P}_{B(\mathbf{c},1)} \subset T' \prod_{a \in \Phi} V_{a,\mathbf{c}}$.

The index $[T':T(K)_b^+]$ divides $\prod_{a \in \Delta} | \pm 1+ \mathfrak{m}_{L_a} / 1+ \mathfrak{m}_{L_a} | = 2^{|\Delta|}$ which is prime to $p$ since $p \neq 2$.
Hence $Q$ is a subgroup, which has an index prime to $p$, of the profinite group $\widehat{P}_{B(\mathbf{c},1)}$.
Since $Q$ is a pro-$p$-group, we get that it is a maximal pro-$p$ subgroup of $\widehat{P}_{B(\mathbf{c},1)}$.

It remains to show that it is the only one, in other words that $Q$ is normal in $\widehat{P}_{B(\mathbf{c},1)}$.
But since $T(K)_b$ normalises $Q$, this gives the result.
\end{proof}

\section{Computation in higher rank}
\label{sec:commutation:relations:general}

As before, $G$ is an almost-$K$-simple quasi-split simply-connected $K$-group and $P$ is a maximal pro-$p$ subgroup of $G(K)$.
By a geometrical analysis, we provided, in Proposition \ref{prop:combinatorial:ball:fixator}, a description of the Frattini subgroup $\mathrm{Frat}(P)$ as a subgroup of the (unique) maximal pro-$p$ subgroup $Q$ of a well-described stabilizer in $G(K)$.
We now want to provide a large enough subset of $\mathrm{Frat}(P)$, so that this subset generates $Q$, and thus $\mathrm{Frat}(P)$.
We provide unipotent elements of $\mathrm{Frat}(P)$ by finding some values $l_a \in \mathbb{R}$ with $a \in \Phi$ such that the valued root groups $U_{a,l_a}$ are subgroups of $[P,P] \subset \mathrm{Frat}(P)$.
In the rank-$1$ case treated in Section \ref{sec:rank:one:case}, we have already found some values $l_a$. In higher rank, we can improve these values for most of roots; more precisely, for all roots which are not corresponding to panels of the (unique) alcove stabilized by $P$.
In Section \ref{sec:commutation:relations}, we invert most of commutation relations providing bounds of valuations of root groups.
In Section \ref{sec:explicit:computation:with:commutation:relations}, we combine those inversions in the whole root system.

\subsection{Commutation relations between root groups of a quasi-split group} 
\label{sec:commutation:relations}

We consider both the split semisimple $\widetilde{K}$-group $\widetilde{G} = G_{\widetilde{K}}$ and the quasi-split $K$-group $G$.
A \textbf{Chevalley-Steinberg system} of $(G,\widetilde{K},K)$ is the datum of morphisms: $\widetilde{x}_\alpha : \mathbb{G}_{a,\widetilde{K}} \rightarrow \widetilde{U}_\alpha$ parametrizing the various root groups of $\widetilde{G}$, and satisfying some axioms of compatibility, given in \cite[4.1.3]{BruhatTits2}, taking in account the commutation relations of absolute root groups and the $\mathrm{Gal}(\widetilde{K} / K)$-action on root groups.
Note that despite the morphisms parametrize root groups of $\widetilde{G}$, a Chevalley-Steinberg system also depends on the quasi-split group $G$ because of the relations between the $\widetilde{x}_\alpha$ where $\alpha \in \widetilde{\Phi}$.
According to \cite[4.1.3]{BruhatTits2}, a quasi-split group always admits a Chevalley-Steinberg system.

According to \cite[14.5]{Borel}, there exist constants $(c_{r,s;\alpha,\beta})_{r,s \in \mathbb{N}^* ; \alpha,\beta \in \widetilde{\Phi}}$ in $\widetilde{K}$, uniquely determined by the Chevalley-Steinberg system $(\widetilde{x}_\alpha)_{\alpha \in \widetilde{\Phi}}$, so that we have the following relations:
$$[\widetilde{x}_\alpha(u),\widetilde{x}_\beta(v)] = \prod_{r,s\in\mathbb{N}^*} \widetilde{x}_{r \alpha + s \beta}(c_{r,s;\alpha,\beta} u^r v^s)$$
for any non-collinear roots $\alpha, \beta \in \widetilde{\Phi}$ and any parameters $u,v \in \widetilde{K}$.
Moreover $c_{r,s;\alpha,\beta} = 0$ as soon as $r \alpha + s \beta \not\in \widetilde{\Phi}$ which makes the above products finite.
These constants are called the \textbf{structure constants}.
There is some flexibility in the choice of a Chevalley-Steinberg system,
so that we can choose $c_{r,s;\alpha,\beta}$ in $\mathbb{Z} 1_{\widetilde{K}}$ where $1_{\widetilde{K}}$ denotes the identity element of $\widetilde{K}^\times$.
More precisely, because $\widetilde{G}$ is generated by its root groups, it comes from a base change of a $\mathbb{Z}$-reductive group \cite[XXV 1.3]{SGA3}.
In this case, one can determinate the $c_{r,s;\alpha,\beta}$, up to sign, to be some coefficients of a Cartan matrix \cite[XXIII 6.4]{SGA3}. 
More precisely, we have:
\begin{Lem}
\label{lem:structure:constants}
Let $\alpha,\beta \in \widetilde{\Phi}$ be two (non-collinear) roots such that $\alpha + \beta \in \widetilde{\Phi}$.

If $\widetilde{\Phi}$ is of type $A_n$, $D_n$, or $E_n$, then $c_{1,1;\alpha,\beta} \in \{\pm 1_{\widetilde{K}}\}$. 

If $\widetilde{\Phi}$ is of type $B_n$, $C_n$, or $F_4$, then $c_{1,1;\alpha,\beta} \in \{\pm 1_{\widetilde{K}},\pm 2 \cdot 1_{\widetilde{K}}\}$. 

If $\widetilde{\Phi}$ is of type $G_2$, then $c_{1,1;\alpha,\beta} \in \{\pm 1_{\widetilde{K}},\pm 2 \cdot 1_{\widetilde{K}}, \pm 3 \cdot 1_{\widetilde{K}}\}$. 
\end{Lem}

In the quasi-split case, given two non-collinear relative roots $a,b \in \Phi$, there exist commutation relations between the corresponding root groups in terms of the parametrizations $(x_a)_{a\in \Phi}$.
These commutation relations can be completely computed in the irreducible root system $\Phi(a,b) = \Phi \cap \left( \mathbb{R} a \oplus \mathbb{R} b \right)$ of rank $2$.
Hence $\Phi(a,b)$ is of type $A_2$, $C_2$, $BC_2$ or $G_2$, and we can assume that $a$ is shorter or has the same length as $b$.
The various commutation relations are written down in \cite[Annexe A]{BruhatTits2} where Bruhat and Tits consider the angles between roots.
Here, we follow another description in terms of length of roots, as in \cite[§1]{PrasadRaghunathan1}.

We recall that, according to Section \ref{sec:star:action}, the Galois group $\mathrm{Gal}(\widetilde{K}/K)$ acts on the absolute roots $\widetilde{\Phi}$ and that the relative roots $\Phi$ can be seen as the orbits for this action.
We recall that $d'=[L'/L_d]$ has been defined in \ref{not:L:prime} to be the number of absolute roots in a short root seen as an orbit.
We do the following assumptions:

\begin{Hyp}\label{hypothesis:on:residue:characteristic}
We assume that the residue characteristic $p$ of $K$ is such that $p > d'$ and  the following structure constants $c_{1,1;\alpha,\beta}$, where $\alpha, \beta \in \widetilde{\Phi}$, are invertible in $\mathcal{O}_K$.
In other words, this is to say that $p\geq 3$ if the relative root system $\Phi$ of the quasi-split almost-$K$-simple $K$-group $G$ is of type $B_n$, $C_n$ of $F_4$; and that $p\geq 5$ if $\Phi$ is of type $G_2$.
\end{Hyp}

\begin{Prop}
\label{prop:generated:commutator:root:group}
Let $a,b,c \in \Phi$ be relative roots such that $c = a+b$ and, at least, one of the two roots $a,b$ is non-multipliable.
Let $l_a \in \Gamma_a$, $l_b \in \Gamma_b$ and $l_c \in \Gamma_c$ be values such that $l_c = l_b + l_a$.

Let $u \in U_{c,l_c}$.
If Hypothesis \ref{hypothesis:on:residue:characteristic} is satisfied, then there exist elements $v \in U_{a,l_a}$, $v' \in U_{b,l_b}$ and $\displaystyle v''\in \prod_{\substack{r,s \in \mathbb{N}^* \\ r+s \geq 2}} U_{r a + s b, r l_a + s l_b}$ such that 
$u = [v,v'] v''$.
\end{Prop}

\begin{proof}
If $u$ is the identity element, the statement is clear.
From now on, we assume that $u$ is not the identity element.
We choose $\alpha \in a$ and $\beta \in b$.
In this proof, length of root is considered in the irreducible (possibly non-reduced) root system $\Phi(a,b)$ of rank $2$.

In the below various cases, we always follow the same sketch of proof.
Firstly, we recall the splitting field of the roots $a$, $b$ and $c=a+b$ computed in Proposition \ref{prop:length:splitting:field}.
Secondly, we recall the commutation relation between $U_a$ and $U_b$, provided by \cite[A.6]{BruhatTits2} and we draw the relative roots that appear in the writing of this commutation relation.
Thirdly, given a non-trivial unipotent element $u \in U_{c,l_c}$, we use the parametrisation of root groups, defined in Section \ref{sec:parametrization}, to provide suitable elements $v \in U_{a,l_a}$ and $v'\in U_{b,l_b}$.
Finally, we check that $v'' = [v,v']^{-1} u$ is suitable.

\noindent\textbf{Case $d'=1$ or the relative roots $a,b,c$ are long:}

By Proposition \ref{prop:length:splitting:field}, we have $L_a=L_b=L_c=L_d$.

By \cite[A.6]{BruhatTits2}, we have the following commutation relation:
$$\forall y \in L_a,\ z \in L_b,\ [x_a(y),x_b(z)]=\prod_{r,s \in \mathbb{N}^*} x_{r a + s b}(c_{r,s;\alpha,\beta} y^r z^s)$$

There exists a parameter $x \in L_c$ such that $u = x_c(x)$ and $\omega(x) \geq l_c$.
We choose $y \in L_a$ such that $\omega(y) = l_a$.
This is possible because $l_a \in \Gamma_a = \Gamma_{L_a}$ by Lemma \ref{lem:sets:of:values:non:multipliable:root}.
We set $z = c_{1,1;\alpha,\beta}^{-1}xy^{-1} \in L_b$.
Then $\omega(z) = \omega(x) - \omega(y) \geq l_c - l_a = l_b$ satisfies $x = c_{1,1;\alpha,\beta}yz$.
Then, we set $v=x_a(y)$, $v' =x_b(z)$ and $\displaystyle (v'')^{-1} = \prod_{r,s \in \mathbb{N}^*,\ r+s \geq 2} x_{r a + s b}(c_{r,s;\alpha,\beta} y^r z^s)$.
For any pair of non-negative integers $(r,s)$ such that $r+s \geq 2$ and $ra+sb$ is a root, we get $\omega(c_{r,s;\alpha,\beta} y^r z^s) \geq r \omega(y) + s \omega(z) \geq r l_a + s l_b$.
Hence $v'' \in \prod_{r,s \in \mathbb{N}^*; r+s \geq 2} U_{r a + s b, r l_a + s l_b}$.
Thus $[v,v']=u(v'')^{-1}$.

\noindent\textbf{Case $d'= 2$, the roots $a,c$ are short, $b$ is long and non-divisible:}

By Proposition \ref{prop:length:splitting:field}, we have $L_b=L_{2a+b}=L_d$ and $L_a=L_c=L'$.

By \cite[A.6.b]{BruhatTits2}, there exist $\varepsilon_1, \varepsilon_2 \in \{\pm 1\}$ such that we have the following commutation relation:

\begin{tabular}{c >{\centering\arraybackslash} m{.4\linewidth}}
$\begin{array}{rcl}
\forall y \in L_a,\ \forall z \in L_b,\ & & \\
\Big[x_a(y),x_b(z)\Big]& = &
x_{a+b}\Big(\varepsilon_{1}yz\Big) \\
& & x_{2a+b}\Big(\varepsilon_{2}y{^\tau}yz\Big)
\end{array}$
&
\begin{tikzpicture}
\draw [->] (0,0) -- (-1,1) node[left]{$b$};
\draw [->] (0,0) -- (0,1) node[above]{$a+b=c$};
\draw [->] (0,0) -- (1,1) node[right]{$2a+b$};
\draw [->] (0,0) -- (1,0) node[below]{$a$};
\end{tikzpicture}

\end{tabular}

There exists a parameter $x \in L_c$ such that $u = x_c(x)$ and $\omega(x) \geq l_c$.
We choose $z \in L_b$ such that $\omega(z) = l_b$.
This is possible because $l_b \in \Gamma_b =\Gamma_{L_b}$.
We set $y = \varepsilon_1 x z^{-1} \in L' = L_a$.
Then $\omega(y) = \omega(x) - \omega(z) \geq l_c - l_b = l_a$ and $x = \varepsilon_1 y z$.
The root $2a+b$ is non-divisible and we get $\omega(y {^\tau}y z) = 2 \omega(y) + \omega(z) \geq 2 l_a + l_b$.
Then, we set $v=x_a(y)$, $v' =x_b(z)$ and $\displaystyle v'' = x_{2 a + b}(-\varepsilon_2 y {^\tau}y z)$.
Hence $v'' \in U_{2a+b,2 l_a + l_b}$.
Thus $u=[v,v']v''$.

\noindent\textbf{Case $d'= 2$, the roots $a,c$ are short, $b$ is long and divisible:}

By Proposition \ref{prop:length:splitting:field}, we have $L_a=L_c=L'$.

By \cite[A.6.c]{BruhatTits2}, there exist $\varepsilon_1, \varepsilon_2 \in \{\pm 1\}$ such that we have the following commutation relation:

\begin{tabular}{c >{\centering\arraybackslash} m{.3\linewidth}}
$\begin{array}{rcl}
\forall y \in L_a,\ \forall z \in L_b^0,&&\\
\Big[x_a(y),x_{\frac{b}{2}}(0,z)\Big]&=&
x_{a+b}\Big(\varepsilon_{1}yz\Big)\\
&&x_{a+\frac{b}{2}}\Big(0,\varepsilon_{2} y{^\tau}yz\Big)
\end{array}$ &
\begin{tikzpicture}
\draw [->] (0,0) -- (-1,1) node[left]{$a$};
\draw [->] (0,0) -- (0,1);
\draw [->] (0,0) -- (0,2) node[above]{$2a+b$};
\draw [->] (0,0) -- (1,1) node[right]{$a+b=c$};
\draw [->] (0,0) -- (1,0);
\draw [->] (0,0) -- (2,0) node[below]{$b$};
\end{tikzpicture}
\end{tabular}

There exists a parameter $x \in L_c$ such that $u = x_c(x)$ and $\omega(x) \geq l_c$.
By Lemma \ref{lem:sets:of:values:multipliable:root}, we have $l_b \in \Gamma_b = \omega({L'^0}^\times)$.
Hence, we can choose $z \in L_{\frac{b}{2}}^0 = L'^0$ such that $\omega(z) = l_b$.
We set $y = \varepsilon_1 x z^{-1} \in L_a = L'$.
Then $\omega(y) = \omega(x) - \omega(z) \geq l_c - l_b = l_a$ and $x = \varepsilon_1 y z$.
The root $2a+b$ is divisible and we can check that $\omega(\varepsilon_{2} y {^\tau}y z) = 2 \omega(y) + \omega(z) \geq 2 l_a + l_b$.
Then, we set $v=x_a(y)$, $v' =x_b(z)$ and $\displaystyle v'' = x_{a + \frac{b}{2}}(0,-\varepsilon_2 y {^\tau}y z)$.
Thus $u=[v,v']v''$.

\noindent\textbf{Case $d'= 2$, the roots $a,b$ are short, $c$ is long and non-divisible:}

By Proposition \ref{prop:length:splitting:field}, we have $L_a=L_b=L'$ and $L_c=L_d$.

By \cite[A.6.b]{BruhatTits2}, there exists $\varepsilon \in \{\pm 1\}$ such that we have the following commutation relation:

\begin{tabular}{c >{\centering\arraybackslash} m{.3\linewidth}}
$\begin{array}{rcl}
\forall y \in L_a,\ \forall z \in L_b,&&\\
\Big[x_a(y),x_b(z)\Big]&=&
x_{a+b}\Big(\varepsilon (yz + {^\tau}y {^\tau}z)\Big)
\end{array}$
&
\begin{tikzpicture}
\draw [->] (0,0) -- (-1,1);
\draw [->] (0,0) -- (0,1) node[above]{$b$};
\draw [->] (0,0) -- (1,1) node[right]{$a+b=c$};
\draw [->] (0,0) -- (1,0) node[below]{$a$};
\end{tikzpicture}
\end{tabular}

There exists a parameter $x \in L_c$ such that $u = x_c(x)$ and $\omega(x) \geq l_c$.
We choose $z \in L_b = L$ such that $\omega(z) = l_b$.
This is possible because $l_b \in \Gamma_b$.
We set $y = \frac{\varepsilon}{2} x z^{-1} \in L_a = L'$.
This makes sense because $p$ does not divide $d'=2$, hence $2 \in \mathcal{O}_K^\times$.
Then $\omega(y) = \omega(x) - \omega(z) \geq l_c - l_b = l_a$ and $ \varepsilon \mathrm{Tr}( y z ) = \frac{x}{2} + \frac{{^\tau}x}{2} = x$ because $x \in L_d$.
Then, we set $v=x_a(y)$, $v' =x_b(z)$ and $v'' = 1$.
Thus $u=[v,v']v''$.

\noindent\textbf{Case $d'= 2$, the roots $a,b$ are short, $c$ is long and divisible:}

By Proposition \ref{prop:length:splitting:field}, we have $L_a=L_b=L_{\frac{c}{2}}=L'$.

By \cite[A.6.c]{BruhatTits2}, there exists $\varepsilon \in \{\pm 1\}$ such that we have the following commutation relation:

\begin{tabular}{c >{\centering\arraybackslash} m{.3\linewidth}}
$\begin{array}{rcl}
\forall y \in L_a,\ \forall z \in L_b,&&\\
\Big[x_a(y),x_b(z)\Big]&=&
x_{\frac{a+b}{2}}\Big(0,\varepsilon ( yz - {^\tau}y {^\tau}z)\Big)
\end{array}$
&
\begin{tikzpicture}
\draw [->] (0,0) -- (-1,1) node[left]{$a$};
\draw [->] (0,0) -- (0,1);
\draw [->] (0,0) -- (0,2) node[above]{$a+b=c$};
\draw [->] (0,0) -- (1,1) node[right]{$b$};
\draw [->] (0,0) -- (1,0);
\draw [->] (0,0) -- (2,0);
\end{tikzpicture}
\end{tabular}

There exists a parameter $x \in L_{\frac{c}{2}}^0=L'^0$ such that $u = x_{\frac{c}{2}}(0,x)$ and $\omega(x) \geq l_c$.
We choose $z \in L_b = L'$ such that $\omega(z) = l_b$.
This is possible because $l_b \in \Gamma_b$.
We set $y = \frac{\varepsilon}{2} x z^{-1} \in L_a = L'$.
This is possible because $p$ does not divide $d'=2$, hence $2 \in \mathcal{O}_K^\times$.
Then $\omega(y) = \omega(x) - \omega(z) \geq l_c - l_b = l_a$ and $\varepsilon  \left( y z - {^\tau}y {^\tau}z \right) = \frac{x - {^\tau}x}{2} = x$ because $x + {^\tau}x = 0$.
Then, we set $v=x_a(y)$, $v' =x_b(z)$ and $v'' = 1$.
Thus $u=[v,v']v''$.

\noindent\textbf{Case $d'= 2$, the roots $a,b,c$ are short, $a,b$ are non-multipliable:}

By Proposition \ref{prop:length:splitting:field}, we have $L_a=L_b=L_c=L'$.

By \cite[A.6.b]{BruhatTits2}, there exists $\varepsilon \in \{\pm 1\}$ such that we have the following commutation relation:

\begin{tabular}{c >{\centering\arraybackslash} m{.4\linewidth}}
$\begin{array}{rcl}
\forall y \in L_a,\ \forall z \in L_b,&&\\
\Big[x_a(y),x_b(z)\Big]&=&
x_{a+b}\Big(\varepsilon y z\Big)
\end{array}$
&
\begin{tikzpicture}
\draw [->] (0,0) -- (1/2,1.732/2) node[above]{$a+b=c$};
\draw [->] (0,0) -- (-1/2,1.732/2) node[left]{$b$};
\draw [->] (0,0) -- (1,0) node[below]{$a$};
\draw (0,0);
\end{tikzpicture}
\end{tabular}

There exists a parameter $x \in L_c$ such that $u = x_{c}(x)$ and $\omega(x) \geq l_c$.
We choose $z \in L_b = L$ such that $\omega(z) = l_b$.
We set $y = \varepsilon x z^{-1} \in L_a = L'$.
Then $\omega(y) = \omega(x) - \omega(z) \geq l_c - l_b = l_a$ and $x = \varepsilon y z$.
Then, we set $v=x_a(y)$, $v' =x_b(z)$ and $v'' = 1$.
Thus $u=[v,v']v''$.

\noindent\textbf{Case $d'= 2$, the roots $a,b,c$ are short, $b$ is non-multipliable and $a$ is multipliable:}

By Proposition \ref{prop:length:splitting:field}, we have $L_a=L_b=L_c=L'$.

By \cite[A.6.c]{BruhatTits2}, there exist $\varepsilon_1,\varepsilon_2 \in \{\pm 1\}$ such that we have the following commutation relation:

\begin{tabular}{c >{\centering\arraybackslash} m{.2\linewidth}}
$\begin{array}{l}
\forall (y,y') \in H(L_a,L_{2a}),\ \forall z \in L_b,\\
\begin{array}{rcl}
\Big[x_a(y,y'),x_b(z)\Big] & =&
x_{a+b}\Big(\varepsilon_1 y z,y' z {^\tau}z\Big)\\
&& x_{2a+b}\Big(\varepsilon_2 z y'\Big)
\end{array}
\end{array}$
&
\begin{tikzpicture}
\draw [->] (0,0) -- (-1,1) node[left]{$b$};
\draw [->] (0,0) -- (0,1) node[above right]{$a+b=c$};
\draw [->] (0,0) -- (0,2) node[above]{$2a+2b$};
\draw [->] (0,0) -- (1,1) node[right]{$2a+b$};
\draw [->] (0,0) -- (1,0) node[below]{$a$};
\draw [->] (0,0) -- (2,0);
\end{tikzpicture}
\end{tabular}

There exists a parameter $(x,x') \in H(L_c,L_{2c})$ such that $u = x_{c}(x,x')$ and $\omega(x') \geq 2 l_c$.
We choose $z \in L_b$ such that $\omega(z) = l_b$.
This is possible because $l_b \in \Gamma_b$.
We set $y = \varepsilon_1 x z^{-1} \in L$
and $y' = x' z^{-1} {^\tau}z^{-1}$.
Then $y {^\tau}y = y'+{^\tau}y'$ and $\omega(y') = \omega(x') - 2 \omega(z) \geq 2 l_c - 2 l_b = 2 l_a$.
This implies $(y,y') \in H(L_a,L_{2a})_{l_a}$.
Moreover $(x,x') = (\varepsilon_1 y z,y' z {^\tau}z)$.
The root $2a+b$ is non-multipliable, non-divisible, and we can check that $\omega(\varepsilon_2 z y') = \omega(y') + \omega(z) \geq 2 l_a + l_b$.
Then, we set $v=x_a(y,y')$, $v' =x_b(z)$ and $v'' = x_{2a+b}(-\varepsilon_2 x' {^\tau}z^{-1})$.
Thus $u=[v,v']v''$.

\noindent\textbf{Case $d'= 2$, the roots $a,b,c$ are short and $a,b$ are multipliable:}

This case where $a$ and $b$ are both multipliable is the only one excluded by the third assumption.
It is considered in Remark \ref{remark:exclusion:case2D2m}.

From now on, we assume $d'=3$. This occurs only for the trialitarian $D_4$.

\noindent\textbf{Case $d'= 3$, the roots $a,c$ are short and $b$ is long:}

By Proposition \ref{prop:length:splitting:field}, we have $L_a=L_c =L_{2a+b}=L'$ and $L_b=L_{3a+b}=L_{3a+2b}=L_d$.

We denote by $\tau \in \Sigma_d$ an element representing an element of order $3$ in the quotient group $\Sigma_d / \Sigma_0$.
For any $y \in L'$, we denote $\Theta(y) = {^\tau}y {^{\tau^2}}y$ and $\mathrm{N}(y) = y \Theta(y)$.
By \cite[A.6.d]{BruhatTits2}, there exist an integer $\eta \in \{1,2\}$ and four signs $\varepsilon_1, \varepsilon_2, \varepsilon_3, \varepsilon_4 \in \{-1,1\}$ such that we have the following commutation relation:

\begin{tabular}{c >{\centering\arraybackslash} m{.2\linewidth}}
$\begin{array}{rcl}
\forall y \in L_a,\ \forall z \in L_b,&&\\
\Big[x_a(y),x_b(z)\Big] &=&
x_{a+b}\Big(\varepsilon_1 y z\Big)\\
&&x_{2a+b}\Big(\varepsilon_2 \Theta(y) z\Big)\\
&&x_{3a+b}\Big(\varepsilon_3 \mathrm{N}(y) z \Big)\\
&&x_{3a+2b}\Big(\varepsilon_4 \eta \mathrm{N}(y) z^2 \Big)
\end{array}$
&
\begin{tikzpicture}
\draw [->] (0,0) -- (1.732/2,1.732*1.732/2) node[right]{$3a+2b$};
\draw [->] (0,0) -- (-1.732/2,1.732*1.732/2) node[left]{$3a+b$};
\draw [->] (0,0) -- (1.732,0) node[below]{$b$};

\draw [->] (0,0) -- (-1.732/2,1/2) node[below]{$a$};
\draw [->] (0,0) -- (0,1) node[above]{$2a+b$};
\draw [->] (0,0) -- (1.732/2,1/2) node[right]{$a+b=c$};
\end{tikzpicture}
\end{tabular}

There exists a parameter $x \in L_c=L'$ such that $u = x_c(x)$ and $\omega(x) \geq l_c$.
We choose $z \in L_b = L_d$ such that $\omega(z) = l_b$.
This is possible because $l_b \in \Gamma_b$.
We set $y = \varepsilon_1 x z^{-1} \in L_a = L'$.
Then $\omega(y) = \omega(x) - \omega(z) \geq l_c - l_b = l_a$ and $x = \varepsilon_1 y z$.
The root $2a+b$ is short and the parameter $\varepsilon_2 \Theta(y) z \in L'$ satisfies $\omega(\varepsilon_2 {^\tau}y {^{\tau^2}}y z) = 2 \omega(y) + \omega(z) \geq 2 l_a + l_b$.
The root $3a+b$ is long and the parameter $\varepsilon_3 \mathrm{N}(y) z \in L_d$ satisfies $\omega(\varepsilon_3 {^\tau}y {^{\tau^2}}y z) = 3 \omega(y) + \omega(z) \geq 3 l_a + l_b$.
The root $3a+2b$ is long and the parameter $\eta \varepsilon_4 z^2 \mathrm{N}(y) \in L$ satisfies $\omega(\eta \varepsilon_4 z^2 y {^\tau}y {^{\tau^2}}y) = \omega(\eta) + 3 \omega(y) + 2 \omega(z) \geq 3 l_a + 2 l_b$.

Then we set $v=x_a(y)$, $v' =x_b(z)$ and $$ v'' = 
x_{3a+2b}\Big(-\eta \varepsilon_4 \mathrm{N}(y) z^2\Big)
x_{3a+b}\Big(-\varepsilon_3 \mathrm{N}(y)z \Big)
x_{2 a + b}\Big(-\varepsilon_2 \Theta(y) z\Big)$$
Hence $v'' \in U_{2a+b,2 l_a + l_b}U_{3a+b,3 l_a + l_b} U_{3a+2b,3 l_a + 2 l_b}$.
Thus $u=[v,v']v''$

\noindent\textbf{Case $d'= 3$, the roots $a,b$ are short and $c$ is long:}

By Proposition \ref{prop:length:splitting:field}, we have $L_a=L_b =L'$ and $L_c=L_d$.

We denote by $\tau \in \Sigma_d$ an element representing an element of order $3$ in the quotient group $\Sigma_d / \Sigma_0$.
For any $y \in L'$, we denote $\mathrm{Tr}(y) = y + {^\tau}y + {^{\tau^3}}y$.
By \cite[A.6.d]{BruhatTits2}, there exists a sign $\varepsilon \in \{-1,1\}$ such that:

\begin{tabular}{c >{\centering\arraybackslash} m{.4\linewidth}}
$\begin{array}{rcl}
\forall y \in L_a,\ \forall z \in L_b,&&\\
\Big[x_a(y),x_b(z)\Big] &=&
x_{a+b}\Big(\varepsilon \mathrm{Tr}(yz)\Big)\\
\end{array}$
&
\begin{tikzpicture}
\draw [->] (0,0) -- (1.732/2,1.732*1.732/2) node[right]{$a+b=c$};
\draw [->] (0,0) -- (-1.732/2,1.732*1.732/2) node[left]{$ $};
\draw [->] (0,0) -- (1.732,0) node[below]{$ $};

\draw [->] (0,0) -- (-1.732/2,1/2) node[below]{$ $};
\draw [->] (0,0) -- (0,1) node[above]{$a$};
\draw [->] (0,0) -- (1.732/2,1/2) node[right]{$b$};
\end{tikzpicture}
\end{tabular}

There exists a parameter $x \in L_c=L_d$ such that $u = x_c(x)$ and $\omega(x) \geq l_c$.
We choose $z \in L_b = L'$ such that $\omega(z) = l_b$.
This is possible because $l_b \in \Gamma_b$.
We set $y = \frac{\varepsilon}{3} x z^{-1} \in L_a = L$.
This is possible because $p$ does not divide $3=d'$,
hence $3 \in \mathcal{O}_K^\times$.
Then $\omega(y) = \omega(x) - \omega(z) \geq l_c - l_b = l_a$ and $x = \varepsilon \mathrm{Tr}( y z )$.
Then, we set $v=x_a(y)$, $v' =x_b(z)$ and $v'' = 1$.
Thus $u=[v,v']v''$

\noindent\textbf{Case $d'= 3$ and the roots $a,b,c$ are short:}

By Proposition \ref{prop:length:splitting:field}, we have $L_a=L_b=L_c=L'$ and $L_{2a+b}=L_{a+2b}=L_d$.

We denote by $\tau \in \Sigma_d$ an element representing an element of order $3$ in the quotient group $\Sigma_d / \Sigma_0$.
For any $y \in L'$, we denote $\Theta(y) = {^\tau}y {^{\tau^2}}y \in L'$ and $\mathrm{Tr}(y) = y + {^\tau}y + {^{\tau^3}}y \in L_d$ and $\mathrm{N}(y) = y \Theta(y) \in L_d$.
For any $y,z \in L'$, we denote $(y*z) = \Theta(y+z) - \Theta(y) - \Theta(z) = {^\tau}y {^{\tau^2}}z + {^{\tau^2}}y {^\tau}z$.
By \cite[A.6.d]{BruhatTits2}, there exist three signs $\varepsilon_1, \varepsilon_2, \varepsilon_3 \in \{-1,1\}$ such that we have the following commutation relation:

\begin{tabular}{c >{\centering\arraybackslash} m{.3\linewidth}}
$\begin{array}{rcl}
\forall y \in L_a,\ \forall z \in L_b,&&\\
\Big[x_a(y),x_b(z)\Big] &=&
x_{a+b}\Big(\varepsilon_1 (y * z)\Big)\\
&&x_{2a+b}\Big(\varepsilon_2 \mathrm{Tr}\big(\Theta(y)z\big)\Big)\\
&&x_{a+2b}\Big(\varepsilon_3 \mathrm{Tr}\big(y\Theta(z)\big)\Big)
\end{array}$
&
\begin{tikzpicture}
\draw [->] (0,0) -- (1.732/2,1.732*1.732/2) node[right]{$a+2b$};
\draw [->] (0,0) -- (-1.732/2,1.732*1.732/2) node[left]{$2a+b$};
\draw [->] (0,0) -- (1.732,0) node[below]{$ $};

\draw [->] (0,0) -- (-1.732/2,1/2) node[below]{$a$};
\draw [->] (0,0) -- (0,1) node[above]{$a+b=c$};
\draw [->] (0,0) -- (1.732/2,1/2) node[right]{$b$};
\end{tikzpicture}
\end{tabular}

We choose $z \in L_b = L'$ such that $\omega(z) = l_b$,
this is possible because $l_b \in \Gamma_b$.
Because $p$ does not divide $2$,
hence $2 \in \mathcal{O}_K^\times$, we can set:
$$y = \frac{\varepsilon_1}{2} \cdot \frac{\mathrm{Tr}(xz) - 2 x z}{\Theta(z)}
= \frac{\varepsilon_1}{2 \mathrm{N}(z)} \left(z \mathrm{Tr}(xz) - 2 x z^2 \right)$$
so that $(y*z) = \varepsilon_1 x$.
Indeed:
$$\begin{array}{rcl}
(y*z) & = & 
\frac{\varepsilon_1}{2 \mathrm{N}(z)} \left({^\tau}z \mathrm{Tr}(xz) - 2 {^\tau}x {^\tau}z^2 \right) {^{\tau^2}}z
+ \frac{\varepsilon_1}{2 \mathrm{N}(z)} \left( {^{\tau^2}}z \mathrm{Tr}(xz) - 2 {^{\tau^2}}x {^{\tau^2}}z^2 \right) {^\tau}z\\

& = & \frac{\varepsilon_1 \Theta(z)}{2 \mathrm{N}(z)}
\left(\mathrm{Tr}(xz) - 2 {^\tau}x {^\tau}z
+ \mathrm{Tr}(xz) - 2 {^{\tau^2}} x {^{\tau^2}}z \right)\\

& = & \frac{\varepsilon_1}{2z}
\left(
2 x z
\right)
\end{array}$$
Then we have:
$$\begin{array}{rl}\omega(y) = & \omega\big( \mathrm{Tr}(xz) - 2 xz\big) - \omega\big(\Theta(z) \big)\\
\geq &
\min \Big( \omega\big(\mathrm{Tr}(xz)\big) ,\omega(x) + \omega(z) \Big) - 2 \omega(z)\\
\geq & \big(\omega(x) + \omega(z)\big)- 2 \omega(z)\\
= & \omega(x) - \omega(z)\\
\geq & l_c - l_a = l_b
\end{array}$$
In fact, we get $\omega(y) = \omega(x) - \omega(z)$ because we deduce the inequality $\omega(x) \geq \omega(y) + \omega(z)$ from the formula $x=\varepsilon_1 (y * z)$.
The root $2a+b$ is long and we can check that the parameter $\displaystyle \varepsilon_2 \mathrm{Tr}\big(\Theta(y)z\big) \in L_d$ satisfies $\omega\Big(\varepsilon_2 \mathrm{Tr}\big(\Theta(y)z\big)\Big) \geq 2 \omega(y) + \omega(z) = 2 l_a + l_b$.
The root $a+2b$ is long and we can check that the parameter $\displaystyle \varepsilon_3 \mathrm{Tr}\big(y\Theta(z)\big) \in L_d$ satisfies $\omega\Big(\varepsilon_3 \mathrm{Tr}\big(y\Theta(z)\big)\Big) \geq \omega(y) + 2 \omega(z) = l_a + 2 l_b$.
Then, we set $v=x_a(y)$, $v' =x_b(z)$ and
$$v'' = 
x_{a+2b}\Big(-\varepsilon_3 \mathrm{Tr}\big(y\Theta(z)\big)\Big) 
x_{2 a + b}\Big(-\varepsilon_2 \mathrm{Tr}\big(\Theta(y)z\big)\Big)$$
Hence $v'' \in U_{2a+b,2 l_a + l_b}U_{a+2b,l_a + 2 l_b}$.
Thus $u=[v,v']v''$.

All the cases except the excluded one, where $a,b$ both are multipliable, have been treated.
\end{proof}

\begin{Rq}
\label{remark:exclusion:case2D2m}
In the excluded case, by \cite[A.6.c]{BruhatTits2}, there exists a sign $\varepsilon \in \{ \pm 1 \}$ such that we have the following commutation relation:

\begin{tabular}{c >{\centering\arraybackslash} m{.4\linewidth}}
$\begin{array}{rcl}
\forall (y,y') \in H(L_a,L_{2a}),&&\\
\quad \forall (z,z') \in H(L_b,L_{2b}),&&\\
\Big[x_a(y,y'),x_b(z,z')\Big]&=&
x_{a+b}\Big(\varepsilon y z\Big)
\end{array}$
&
\begin{tikzpicture}
\draw [->] (0,0) -- (-1,1);
\draw [->] (0,0) -- (0,1) node[left]{$b$};
\draw [->] (0,0) -- (0,2);
\draw [->] (0,0) -- (1,1) node[right]{$a+b=c$};
\draw [->] (0,0) -- (1,0) node[above]{$a$};
\draw [->] (0,0) -- (2,0);
\end{tikzpicture}
\end{tabular}

There exists a parameter $x \in L_c = L'$ such that $u = x_{c}(x)$ and $\omega(x) \geq l_c$.
The problem is that, for a multipliable root $a\in\Phi$,
the set of values $\Gamma_a$ does not control completely the valuation of the first term $y$ of a parameter $(y,y') \in H(L_a,L_{2a})$.
One can show that, when $l_a \not \in \Gamma'_a$, we get $\omega(y) > l_a$.
Hence the inclusion $[U_{a,l_a},U_{b,l_b}] \subset U_{a+b,l_a+l_b}$ is not, in general, an equality.
\end{Rq}

\subsection{Generation of unipotent elements thanks to commutation relations between valued root groups} 
\label{sec:explicit:computation:with:commutation:relations}

In Corollary \ref{cor:majoration:frattini}, we obtained that $\mathrm{Frat}(P)$ is a subgroup of a pro-$p$ group $Q$ written in terms of valued root groups.
We want to get an equality when it is possible.
It suffices to provide a generating system of the biggest group consisting of $p$-powers and commutators of elements chosen in $P$.
In a general consideration of a compact open subgroup $H$ of $G(K)$, in Section \ref{sec:lower:bounds:positive:roots}, we do an induction on the positive roots from the highest to the simple roots to provide bounds of valued root groups contained in $[H,H]$;
in Section \ref{sec:lower:bounds:negative:roots}, we furthermore consider the length of roots to provide bounds for the whole root system.
In Section \ref{sec:lower:bounds:frattini:subgroup}, we go back to the situation of the Frattini subgroup $\mathrm{Frat}(P) = \overline{P^p [P,P]} \supset [P,P]$.

In order to do an induction on the set of relative roots, the following lemma in Lie combinatorics explains how to get, step by step, all the roots as a linear combination with integer coefficients of the lowest root and the simple roots.

\begin{Lem}
\label{lem:induction:root:system}
Let $\Phi$ be an irreducible root system of rank greater or equal to $2$ and $\Delta$ be a basis of simple roots in $\Phi$, associated to an order $\Phi^+$.
Let $h$ be the highest root for this order.
\begin{enumerate}
\item[(1)]
Let $\beta \in \Phi^+ \setminus (\Delta \cup 2 \Delta)$ be a positive root which is not the multiple of a simple root.
Then, there exists a simple root $\alpha \in \Delta$ and a positive root $\beta' \in \Phi^+$ such that $\beta = \alpha + \beta'$ and the roots $\alpha, \beta'$ are not collinear.
\item[(2)]
Let $\gamma \in \Phi^- \setminus \{-h\}$.
There exists a positive root $\beta \in \Phi^+$ and a negative root $\gamma' \in \Phi^-$ such that $\gamma = \beta + \gamma'$ and the roots $\beta, \gamma'$ are not collinear.
\item[(3)]
Let $\alpha \in \Delta$.
There exists a simple root $\beta \in \Delta$ such that $\alpha + \beta$ is a positive root.
Moreover, the roots $\alpha + \beta \in \Phi^+$ and $-\beta$ are not collinear.
\end{enumerate}
\end{Lem}

\begin{proof}
According to notations of \cite[VI.1.3]{Bourbaki4-6},
we denote by $V$ the $\mathbb{R}$-vector space generated by $\Delta$ containing $\Phi$ and by $(\cdot | \cdot)$ a scalar product which is invariant by the Weyl group.

(1) Let $\beta \in \Phi^+ \setminus \Delta$ be a positive non-simple root.
Because $\Delta$ is a basis of the Euclidean vector space $V$
and $\beta \in \Phi^+$ is in the cone $\mathbb{Z}_{> 0} \Delta$ generated by $\Delta$, there exists $\alpha \in \Delta$ such that $(\alpha | \beta) > 0$.
By \cite[VI.1.3 Corollaire]{Bourbaki4-6}, we get $\beta' = \beta - \alpha \in \Phi$ because we excluded the case where $\alpha = \beta$ assuming $\beta \not\in \Delta$.
Moreover, $\beta'$ is a positive root because its integer coefficients when we write it in the basis $\Delta$ all have the same sign (hence are positive).
Finally, $\beta'$ and $\alpha$ are not collinear because we assumed that $\beta$ is not the multiple of a simple root.
Hence $\beta' = \beta - \alpha$ satisfies assertion (1).

(2) Let $\gamma \in \Phi^- \setminus \{-h,-\frac{h}{2}\}$.
If $(-h | \gamma ) > 0$, then the sum $\beta = h + \gamma \in \Phi^+$ is a positive root.
Moreover, $-h$ and $\beta$ are not collinear because we assumed that $\gamma$ and $h$ are not collinear.
Hence $\beta$ and $\gamma' = - h$ satisfies assertion (2).
Otherwise, we necessarily get the equality $(-h | \gamma ) = 0$ according to \cite[VI.1.8 Proposition 25]{Bourbaki4-6} and there exists a simple root $\alpha \in \Delta$ such that $(\alpha | \gamma) > 0$,
because the roots $\alpha \in \Delta$ form a basis of the Euclidean space $V$ and $-h \neq 0$.
The roots $\gamma$ and $\alpha$ are not collinear because, if they were, we should have $\gamma \in \mathbb{R}_+ \alpha$ according to assumption $(\gamma | \alpha ) > 0$;
and this contradicts $\gamma \in \Phi^-$.
Hence $\gamma'= \gamma-\alpha \in \Phi^-$ is a negative root.
Thus, $\gamma'$ and $\beta = \alpha$ satisfies assertion (2).

Let $\gamma = -\frac{h}{2}$. In particular, this happens only if $\Phi$ is non-reduced. We can apply the same method inside $\Phi_{\mathrm{nd}}$, because the root $-\frac{h}{2}$ is a short root of $\Phi_{\mathrm{nd}}$, hence it cannot be collinear to the highest root of $\Phi_{\mathrm{nd}}$.

(3) Let $\alpha \in \Delta$.
Any $\beta$ connected to $\alpha$ by an edge in $\mathrm{Dyn}(\Delta)$ satisfies (3).
Such a simple root exists because we assumed $\Phi$ to be of rank greater of equal to $2$.
\end{proof}

\begin{Lem}
\label{lem:write:root:positive:coefficients}
Let $\Phi$ be an irreducible root system of rank greater or equal to $2$ and $\Delta$ be a basis of simple roots in $\Phi$, associated to an order $\Phi^+$.
Let $h$ be the highest root for this order.
For any root $\gamma \in \Phi$, there exist non-negative integers $(n_\alpha(\gamma))_{\alpha \in \Delta}$ such that:
$$\gamma = -h + \sum_{\alpha \in \Delta} n_\alpha(\gamma) \alpha$$
\end{Lem}

\begin{proof}
We proceed by induction on height.
If $\gamma = -h$, it is clear.

Induction step: If $\gamma \in \Phi$, by \ref{lem:induction:root:system}, there exists $\beta \in \Phi^+$ and $\gamma' \in \Phi$ such that $\gamma = \gamma' + \beta$.
Hence by induction hypothesis, there exist non-negative integers $(n_\alpha(\gamma'))$ such that $\gamma' = -h + \sum_{\alpha \in \Delta} n_\alpha(\gamma') \alpha$.
According to \cite[VI.1.6 Th\'eor\`eme 3]{Bourbaki4-6}, there exist non-negative integers $(n_\alpha(\beta))$ such that $\beta = \sum_{\alpha \in \Delta} n_\alpha(\beta) \alpha$.
Hence, the property is satisfied by $n_\alpha(\gamma) = n_\alpha(\gamma') + n_\alpha(\beta)$.
\end{proof}

\begin{Def}
Let $f : \Phi \rightarrow \mathbb{R}$ be a map.
We say that the map $f$ is \textbf{concave} if it satisfies the following axioms:
\begin{itemize}
\item[(C0)] $f(2a) \leq 2 f(a)$ for any root $a \in \Phi$ such that $2a \in \Phi$;
\item[(C1)] $f(a+b) \leq f(a) + f(b)$ for any roots $a,b \in \Phi$ such that $a+b \in \Phi$;
\item[(C2)] $0 \leq f(a) + f(-a)$ for any root $a \in \Phi$.
\end{itemize}
\end{Def}

Despite these axioms look like a convexity property, they correspond in fact to a concavity property in terms of valued root groups.

\begin{Ex}
\label{ex:f:is:concave}
For any non-empty subset $\Omega \subset \mathbb{A}$,
the map $f_\Omega : a \mapsto \sup \{ -a(x),\ x \in \Omega \}$ is concave.
Later, we will apply Propositions \ref{prop:generated:commutator:positive:root:groups} and \ref{prop:generated:commutator:negative:root:groups} to values $l_a = f_{\mathbf{c}_{\mathrm{af}}}(a)$.
\end{Ex}

\subsubsection{Lower bounds for positive root groups}
\label{sec:lower:bounds:positive:roots}

Let $(l_a)_{a\in \Phi}$ be any values in $\mathbb{R}$.
We define the following values $(l'_b)_{b \in \Phi^+}$ depending on the $l_a$, to become bounds for the positive root groups.

\begin{Not}
\label{not:positive:bounds}
For any positive root $b \in \Phi^+$, we can write uniquely $b = \sum_{\alpha \in \Delta} n_\alpha(b) \alpha$ where $n_a(b) \in \mathbb{N}$ are nonnegative integers (not all equal to zero).
We define a value $l'_b = \sum_{\alpha \in \Delta} n_\alpha(b) l_\alpha$.
\end{Not}

Thanks to Lemma \ref{lem:induction:root:system}, we do several inductions on various root systems to provide bounds, thanks to Proposition \ref{prop:generated:commutator:root:group}, for the valuations of the valued root groups contained in the Frattini subgroup $\mathrm{Frat}(P)$.
The first step, in terms of positive roots, is the following:

\begin{Prop}\label{prop:generated:commutator:positive:root:groups}
Let $(l_a)_{a \in \Phi}$ be values in $\mathbb{R}$.
Assume that for any simple root $a \in \Delta$, we have $l_a \in \Gamma_a$.

(1) Then $l'_b \in \Gamma_b$ for any positive root $b \in \Phi^+$.

(2) Assume, moreover, that the map $a \mapsto l_a$ is concave.
Then we have $l'_b \geq l_b$ for any positive root $b \in \Phi^+$.

(3) Furthermore, assume that Hypothesis \ref{hypothesis:on:residue:characteristic} is satisfied.
Let $H$ be a (compact open) subgroup of $G(K)$ containing the valued root groups $U_{a,l_a}$ for $a \in \Phi$.
Then for any root $b \in \Phi^+ \setminus \Delta$, the derived group $[H,H]$ contains the valued root group $U_{b,l'_b}$.
\end{Prop}

\begin{proof}

(1) We apply Proposition \ref{prop:length:splitting:field} and Lemmas \ref{lem:sets:of:values:multipliable:root} and \ref{lem:sets:of:values:non:multipliable:root} in the various cases.

\textbf{First case: $\Phi$ is a reduced root system and $L'/L_d$ is unramified.}
For any root $b \in \Phi^+$, the set of values $\Gamma_b$ of $b$ is $\Gamma_{L'} = \Gamma_{L_{d}}$.
Hence, the sum $l'_b = \sum_{\alpha \in \Delta} n_\alpha(b) l_a$ is an element of $\Gamma_{L_{d}} = \Gamma_b$.

\textbf{Second case: $\Phi$ is a reduced root system and $L'/L_d$ is ramified.}
For any long root of $\Phi$, its set of values is the group $d' \Gamma_{L'} = \Gamma_{L_d}$.
For any short root of $\Phi$, its set of values is the group $\Gamma_{L'}$.
Hence, for any short root $b\in \Phi$, the sum $l'_b = \sum_{\alpha \in \Delta} n_\alpha(b) l_\alpha$ is an element of $\Gamma_{L'} = \Gamma_b$.

Let $b \in \Phi$ be a long relative root arising from an absolute root $\beta \in \widetilde{\Phi}$.
Write $\beta = \sum_{\widetilde{\alpha} \in \widetilde{\Delta}} n'_{\widetilde{\alpha}}(\beta) \widetilde{\alpha}$.
Hence $n_\alpha(b) = \sum_{\widetilde{\alpha} \in \alpha} n'_{\widetilde{\alpha}}(\beta)$.
Moreover, $n'_{\widetilde{\alpha}}(\beta)$ is constant along the class $\alpha$ because $\beta$ is $\Sigma_d$-invariant and $\alpha = \Sigma_d \cdot \widetilde{\alpha}$ is an orbit.
Hence, for any short simple root $\alpha$ arising from $\widetilde{\alpha}$ taking in the same irreducible component as $\beta$, we obtain $n_\alpha(b) = d' n'_{\widetilde{\alpha}}(\beta)$.
As a consequence, $n_\alpha(b) l_\alpha = n'_{\widetilde{\alpha}}(\beta) d' l_\alpha \in d' \Gamma_{L'} = \Gamma_{L_d}$.
For any long simple root $\alpha$, we have $l_\alpha \in \Gamma_{L_d}$.
Hence, the sum $l'_b = \sum_{\alpha \in \Delta} n_\alpha(b) l_\alpha$ is an element of $\Gamma_{L_d} = \Gamma_b$.

\textbf{Third case: $\Phi$ is a non-reduced root system.}
The set of values of any multipliable root is $\frac{1}{2} \Gamma_{L'}$.
The set of values of any non-multipliable, non-divisible root is $\Gamma_{L'}$.
For any multipliable root $b \in \Phi^+$, the sum $l'_b$ is an element of $\frac{1}{2} \Gamma_{L'} = \Gamma_b$.
We number by $a_1, \dots, a_{l-1}$ the non-multipliable simple roots and by $a_{l}$ the multipliable simple root.
Any non-multipliable non-divisible root $b \in \Phi^+$ can be written as $b = \sum_{j = 1}^{l} n_j(b) a_j$ with $n_{l} \in \{0,2\}$.
We have $n_j(b) l_{a_j} \in \Gamma_{a_j} = \Gamma_{L}$ and $n_{l}(b) l_{a_{l}} \in 2 \Gamma_{a_{l}} = \Gamma_{L'}$.
Hence the sum $l'_b$ is an element of $\Gamma_{L'} = \Gamma_b$.

(2) For any positive root $b \in \Phi^+$, we apply recursively Lemma \ref{lem:induction:root:system}(1) to $\Phi^+$ in order to write $b = \sum_{j=1}^N a_j$ where $a_j \in \Delta$ are simple roots (possibly with repetitions) and $N \in \mathbb{N}^*$ such that $b_n = \sum_{j=1}^n a_j$ is a (positive) root for any $n \in [1,N]$.
By induction, we get that $l'_{b_n} \geq l_{b_n}$.
Indeed, for any $0 \leq n \leq N-1$, we have $l'_{b_{n+1}} = l'_{b_n} + l_{a_{n+1}} \geq l_{b_n} + l_{a_{n+1}}$ by induction hypothesis;
and from the concavity relation (C1), we end the inequality by
$l_{b_n} + l_{a_{n+1}} \geq l_{b_n + a_{n+1}} = l_{b_{n+1}}$.
Hence, we obtain the inequality $l_b \leq l'_b$.

(3) Consequently, we have the inclusion $U_{b,l'_b} \subset U_{b,l_b}$.
We proceed by decreasing strong induction on height in the root system $\Phi$ relatively to the basis $\Delta$.

\textbf{Basis:} Let $h$ be the highest root of $\Phi$.
For the root group $U_{h,l'_h}$, we know by Lemma \ref{lem:induction:root:system}(1) that there exists a simple root $a\in\Delta$ and a positive root $b \in \Phi^+$ non-collinear to $a$, and non both multipliable, such that $h = a+b$.
Let $u \in U_{h,l'_h}$.
We have the group inclusion $U_{b,l'_b} \subset U_{b,l_b}$.
We know by Proposition \ref{prop:generated:commutator:root:group}, that there exist elements $v \in U_{a,l_a}$, $v' \in U_{b,l'_b}$ and $v'' \in \prod_{r,s\in\mathbb{N}^*;r+s \geq 2} U_{ra+sb,rl_a+sl'_b}$ such that $u = [v,v'] v''$.
But, for any pair of positive integers $(r,s)$ such that $r+s \geq 2$, the character $ra+sb$ is not a root because this would contradict maximality of height of $h$.
Hence $v'' = 1$.
Thus, we get $U_{h,l'_h} \subset [H,H]$.

\textbf{Inductive step:} Let $c \in \Phi^+ \setminus \Delta$.
By Lemma \ref{lem:induction:root:system}(1), we write $c = a+b$ where $a\in\Delta$ and $b \in \Phi^+$.
Let $u \in U_{c,l'_c}$.
We know by Proposition \ref{prop:generated:commutator:root:group}, that there exist elements $v \in U_{a,l_a}$, $v' \in U_{b,l'_b}$ and $v'' \in \prod_{r,s\in\mathbb{N}^*;r+s \geq 2} U_{ra+sb,rl_a+sl'_b}$ such that $u = [v,v'] v''$.
For any pair of positive integers $(r,s)$ such that $r+s \geq 2$,
if the character $ra+sb$ is a root,
then we have $r l_a + s l'_b = l'_{ra+sb}$ by definition of the $l'$.
Moreover, the height of $ra+sb$ is greater than $c$.
By induction hypothesis, the valued root group $U_{ra+sb,l'_{ra+sb}}$ is a subgroup of $[H,H]$, hence $v'' \in [H,H]$.
As a consequence, we get $U_{c,l'_c} \subset [H,H]$.
\end{proof}

\subsubsection{Lower bounds for negative root groups}
\label{sec:lower:bounds:negative:roots}

In order to get an analogous result for negative roots, doing an induction on height no longer works.
In fact, we have to consider length of roots instead of height.
We recall that, in Notation \ref{not:inverse:root:system}, we defined a pure Lie theoretic dual root system $\Phi^D$.

\begin{Lem}
\label{lem:short:root:combinatorial}
Let $\Phi$ be a reduced irreducible non-simply laced root system of rank $l \geq 2$.
Let $\Phi^+$ be an ordering on $\Phi$ and $\theta \in \Phi$ be the short root such that $\theta^D$ is the highest root of $\Phi^D$ in the corresponding ordering.
Then, any short root $c \in \Phi \setminus \{-\theta\}$ can be written $c=a+b$ where $a,b \in \Phi$ are non-collinear roots such that $a\in \Phi$ is short and $b \in \Phi^+$.
In particular, every short root is higher than $- \theta$.
\end{Lem}

\begin{proof}
We provide these roots case by case thanks to an explicit realization of the root system in $\mathbb{R}^l$.
Let $(e_i)_{1 \leq i \leq l}$ be the canonical basis of the Eucliean space $\mathbb{R}^l$.

\textbf{$\Phi$ is of type $B_l$ with $l \geq 2$:}\\
Basis: $a_i = e_i - e_{i+1}$ where $1 \leq i < l$ and $a_l = e_l$\\
Short roots: $\pm e_i$ for $1 \leq i \leq l$ and $\theta = e_1$\\
For any short root $c\in \Phi\setminus \{-\theta\}$,
\begin{itemize}
\item if $c \in \Phi^+$, we write $c = e_i = a + b$ with $1 \leq i \leq l$, $a = -e_j$, $b = e_i + e_j$ and $j \neq i$;
\item if $c\in \Phi^-$, we write $c = - e_i = a + b$ with $1 < i \leq l$, $a = -e_1$ and $b = e_1 - e_i$.
\end{itemize}

\textbf{$\Phi$ is of type $C_l$ with $l \geq 3$:}\\
Basis: $a_i = e_i - e_{i+1}$ where $1 \leq i < l$ and $a_l = 2 e_l$\\
Short roots: $\pm e_i \pm e_j$ where $1 \leq i < j \leq l$ and $\theta = e_1+e_2$\\
For any short root $c\in \Phi\setminus \{-\theta\}$,
\begin{itemize}
\item if $c = e_i \pm e_j$ where $1 \leq i < j \leq l$, we write $c = a + b$  where $a = -e_i \pm e_j$ and $b = 2 e_i$;
\item if $c = -e_i \pm e_j$ where $1 < i < j \leq l$, we write $c = a + b$ where $a = -e_1 - e_i$ and $b = e_1 \pm e_i$;
\item if $c = -e_1 \pm e_j$ where $2 < j \leq l$, we write $c = a + b$ where $a = -e_1 - e_2$ and $b = e_2 \pm e_j$;
\item if $c = -e_1 + e_2$, we write $c = a + b$ where $a = -e_1 - e_3$ and $b = e_2 + e_3$.
\end{itemize}

\textbf{$\Phi$ is of type $F_4$:}\\
Basis: $a_1 = e_2 - e_3$, $a_2 = e_3 - e_4$, $a_3 = e_4$ and $a_4 = \frac{1}{2}(e_1 - e_2 - e_3 -e_4)$\\
Highest root: $h = e_1+e_2 = 2a_1 + 3a_2 + 4a_3 + 2a_4$\\
Short roots: $\pm e_i$ where $1 \leq i \leq 4$ and $\frac{1}{2}(\pm e_1 \pm e_2 \pm e_3 \pm e_4)$ and $\theta = e_1$\\
For any short root $c\in \Phi\setminus \{-\theta\}$,
\begin{itemize}
\item if $c = e_1$, we write $c = a + b$ where $a = \frac{1}{2}(e_1-e_2-e_3-e_4)$ and $b = \frac{1}{2}(e_1+e_2+e_3+e_4)$;
\item if $c = \pm e_i$ where $1 < i \leq 4$, we write $c = a + b$  where $a = \frac{1}{2}(-e_1+\pm e_i - e_j - e_k)$ and $b = \frac{1}{2}(e_1+\pm e_i + e_j + e_k)$ where $\{i,j,k\} = \{2,3,4\}$;
\item if $c = \frac{1}{2}(e_1 \pm e_2 \pm e_3 \pm e_4)$, we write $c = a + b$  where $a = \frac{1}{2}(-e_1 \mp e_2 \pm e_3 \pm e_4)$ et $b = e_1 \pm e_2$;
\item if $c = \frac{1}{2}(-e_1 \pm e_2 \pm e_3 \pm e_4)$, we write $c = a + b$  where $a = -e_1$ and $b = \frac{1}{2}(e_1 \pm e_2 \pm e_3 \pm e_4)$.
\end{itemize}

\textbf{$\Phi$ is of type $G_2$:}\\
Basis: $\alpha$, $\beta$ where $\alpha$ is short and $\beta$ is long\\
Highest root: $h = 3 \alpha + 2 \beta$\\
We have $\theta = 2 \alpha + \beta$.
We summarize the choices for the short roots, except $-\theta$, case by case, in the following table:
$$\begin{array}{|c|c|c|c|c|c|}
\hline
c & 2 \alpha + \beta & \alpha + \beta & \alpha & - \alpha & -\alpha - \beta \\
\hline
a & \alpha & -\alpha & -\alpha - \beta & -2\alpha - \beta & - 2 \alpha - \beta \\
\hline
b & \alpha + \beta & 2 \alpha + \beta & 2 \alpha + \beta & \alpha + \beta & \alpha\\
\hline
\end{array}$$
\end{proof}

We let $(\delta_c)_{c \in \Phi}$, $\Phi_{\mathrm{nd}}^\delta$, $\theta$ and $h$ be defined as in Notation \ref{not:highest:dual:root}.
Let $(l_a)_{a\in \Phi}$ be any values in $\mathbb{R}$.
We define the following values $(l''_c)_{c \in \Phi}$ depending on the $l_a$, to become bounds for all the root groups.

\begin{Not}\label{not:negative:bounds}
For any non-divisible root $c \in \Phi_{\mathrm{nd}}$, thanks to Lemma \ref{lem:write:root:positive:coefficients} applied in the root system $\Phi_{\mathrm{nd}}^\delta$, we write:
$$c^\delta = - \theta^\delta + \sum_{\alpha^\delta\in \Delta^\delta} n'_\alpha(c) \alpha^\delta \in \Phi^\delta$$
with $n'_\alpha(c) \in \mathbb{N}$.
We define $l''_c \in \mathbb{R}$ by:
$$\delta_c l''_c = \delta_{-\theta} l_{-\theta} + \sum_{\alpha \in \Delta} \delta_\alpha n'_\alpha(c) l_\alpha $$
Furthermore, for any multipliable root $c \in \Phi$, we define $l''_{2c} = 2 l''_{c}$.
Note that for any root $c \in \Phi$, there exist integers $n_{\alpha}(c)$ for $\alpha \in \Delta$, uniquely determined by:
$$c = \sum_{\alpha \in \Delta} n_\alpha(c) \alpha$$
This extends Notation \ref{not:positive:bounds}.
\end{Not}

These values overestimate the values of valued root groups contained in the derived group $[H,H]$.
In particular, this proposition provides values even for simple roots, which were not treated in Proposition \ref{prop:generated:commutator:positive:root:groups}.
We can remark on an example that, in general, this values are not optimal for positive non-simple roots.

\begin{Prop}
\label{prop:generated:commutator:negative:root:groups}
Let $(l_a)_{a \in \Phi}$ be values in $\mathbb{R}$.
Assume that for any simple root $a \in \Delta$, we have $l_a \in \Gamma_a$ and that $l_{-\theta} \in \Gamma_{-\theta}$.

(1) We have $l''_c \in \Gamma_c$ for any non-divisible root $c \in \Phi_{\mathrm{nd}} \setminus \{-\theta\}$.

(2) We assume, moreover, that the map $a \mapsto l_a$ is concave.
For any root $c \in \Phi$, we have $l''_c \geq l_c$;
for any positive root $b \in \Phi^+$, we have $l''_b \geq l'_b \geq l_b$.

(3) We assume, moreover, that the irreducible root system $\Phi$ is not of rank $1$ and that Hypothesis \ref{hypothesis:on:residue:characteristic} is satisfied.
Let $H$ be a (compact open) subgroup of $G(K)$ containing the valued root groups $U_{a,l_a}$ for $a \in \Phi$.
If $G$ is a trialitarian $D_4$ (i.e. $\Phi$ of type $G_2$ and $\delta_\theta = 3$), we assume furthermore that $l'_{\theta} + l_{-\theta} \leq \omega(\varpi_{L'})$.
Then the derived group $[H,H]$ contains the valued root groups $U_{c,l''_c}$ for any root $c \in \Phi \setminus \{-\theta\}$.
\end{Prop}

\begin{proof}

(1) If $\Phi$ is a reduced root system,
then $\Phi^{\delta} = \Phi$ if the extension $L' / L_d$ is unramified;
and $\Phi^{\delta} = \Phi^D$ if the extension $L' / L_d$ is ramified.
By Definition \ref{def:dual:root:system}, for any root $c \in \Phi$, the integer $\delta_c$ is the order of the quotient group $\Gamma_c / \Gamma_{L_d}$, so that $\delta_c \Gamma_c = \Gamma_{L_d}$.
Hence, each term $n'_\alpha(c) \delta_\alpha l_\alpha$ and $\delta_{-\theta} l_{-\theta}$ of the sum belongs to the group $\Gamma_{L_d}$.
Thus $\delta_c l''_c \in \Gamma_{L_d} = \delta_c \Gamma_c$,
and we obtain $l''_c \in \Gamma_c$ for any root $c\in\Phi$.

If $\Phi$ is a non-reduced root system, then the set of values of multipliable roots is $\frac{1}{2} \Gamma_{L'}$ by Lemma \ref{lem:sets:of:values:multipliable:root} and the set of values of non-multipliable and non-divisible roots is $\Gamma_{L'}$.
For any non-divisible root $c \in \Phi$, the value $\delta_c l_c$ is an element of $\Gamma_{L'}$, hence so is the sum $l''_c$.
If $c$ is non-multipliable, then $\delta_c = 1$, hence $l''_c \in \Gamma_{L'} = \Gamma_c$.
If $c$ is multipliable, then $\delta_c = 2$ hence $l''_c \in \frac{1}{2}\Gamma_{L'} = \Gamma_c$.

(2) In the following, for any root $c \in \Phi_{\mathrm{nd}}$, we denote by $n_\alpha(c)$ and $n'_\alpha(c)$ the integers defined in Notation \ref{not:negative:bounds}.
We furthermore denote by $n_\alpha^\delta(c)$ the integers uniquely determined by the following writing in basis $\Delta^\delta$:
$c^\delta = \sum_{\alpha \in \Delta} n_\alpha^\delta(c) \alpha^\delta$.
From uniqueness, for any $\alpha \in \Delta$, we deduce that $\delta_\alpha n_\alpha^\delta(c) = \delta_c n_\alpha(c)$ and that $n'_\alpha(c) = n_\alpha^\delta(\theta) - n_\alpha^\delta(c) \geq 0$ (it is a non-negative integer).

Let $b \in \Phi_{\mathrm{nd}}^+$ be a non-divisible positive root.
In $V^* = \mathrm{Vect}(\Phi)$ we have:
$$ \begin{array}{rl} b^\delta = &\displaystyle - \theta^\delta + \theta^\delta + \sum_{\alpha \in \Delta} n_\alpha^\delta(b) \alpha^\delta\\
= &\displaystyle - \theta^\delta + \sum_{\alpha \in \Delta} \left( n_\alpha^\delta(\theta) + n_\alpha^\delta(b) \right) \alpha^\delta
\end{array}$$
By definition of $l''_b, l'_b, l'_\theta$, we get:
$$ \begin{array}{rl} \delta_b l''_b = &\displaystyle \delta_\theta l_{-\theta} + \sum_{\alpha \in \Delta} \left( n_\alpha^\delta(b) + n_\alpha^\delta(\theta) \right) \delta_\alpha l_\alpha\\
= &\displaystyle \delta_\theta l_{-\theta} + \left( \sum_{\alpha \in \Delta} \delta_b n_\alpha(b) l_\alpha \right) + \left( \sum_{\alpha \in \Delta} \delta_\theta n_\alpha(\theta) l_\alpha \right)\\
= &\displaystyle \delta_\theta l_{-\theta} + \delta_b l'_b + \delta_\theta l'_\theta
\end{array}$$
Hence $\delta_b (l''_b - l'_b) = \delta_\theta (l'_\theta + l_{-\theta})$.
According to Proposition \ref{prop:generated:commutator:positive:root:groups}(2), we have $l'_b \geq l_b$ for all positive roots and, in particular, $l'_\theta \geq l_\theta$.
Hence, by axiom (C2), we get $l'_\theta + l_{-\theta} \geq l_\theta + l_{-\theta} \geq 0$.
As a consequence, we get $l''_b \geq l'_b \geq l_b$.

Let $b \in \Phi^+$ be a multipliable root.
Then $l''_{2b} = 2l''_b \geq l'_{2b} = 2 l'_{b} \geq 2 l_b$.
By axiom (C0), we have $2 l_b \geq l_{2b}$, hence $l''_{2b} \geq l_{2b}$.

Let $c \in\Phi_\mathrm{nd}^-$ be a non-divisible negative root.
We want to prove that $l''_c \geq l_c$.
We proceed by induction on height in $\Phi_{\mathrm{nd}}$.

\textbullet\ \textbf{First case: $\Phi_\mathrm{nd}^\delta = \Phi_\mathrm{nd}$.}
Then $\delta_\theta = 1$, $h = \theta$ and $\delta_c = 1$ for any root $c \in \Phi$.
By definition, $l''_{-h} = l''_{-\theta}= l_{-\theta}=l_{-h}$.

If $c \neq - \theta$, by Lemma \ref{lem:induction:root:system}(2), there exist $a \in \Phi_\mathrm{nd}$ and $b \in \Phi_\mathrm{nd}^+$ such that $c = a + b$.
From $c = - \theta + \sum_\alpha n'_\alpha(c) \alpha = - \theta + \sum_\alpha n'_\alpha(a) \alpha + \sum_\alpha n_\alpha(b) \alpha = a + b$, we deduce $n'_\alpha(c) = n'_\alpha(a) + n_\alpha(b)$.
Hence $l''_c = l''_a + l'_b \geq l_a + l'_b$ by induction hypothesis.
By axiom (C1) and because $l'_b \geq l_b$, we get $l''_c \geq l_a + l_b \geq l_{a+b} = l_c$.

\textbullet\ \textbf{Second case: $\Phi_\mathrm{nd}^\delta = \Phi_\mathrm{nd}^D \neq \Phi_\mathrm{nd}$.}
Then $\delta_\theta = d'$.

We firstly do the induction, initialized by $l''_{-\theta} = l_{-\theta}$, on height of short roots.
Assume that $c \neq -\theta$ is a short root in $\Phi_{\mathrm{nd}}$.
By Lemma \ref{lem:short:root:combinatorial}, there exist a short root $a \in \Phi_\mathrm{nd}$ and a positive root $b \in \Phi^+_{\mathrm{nd}}$ such that $c = a+b$. Hence $\delta_a = \delta_c = \delta_\theta$.
We have $\delta_\theta b = \delta_\theta (c-a) = c^\delta -a^\delta = - \theta^\delta + \sum_\alpha \delta_\alpha n'_\alpha(c) + \theta^\delta - \sum_\alpha \delta_\alpha n'_\alpha(a) = \sum_\alpha \delta_\alpha \Big( n'_\alpha(c) - n'_\alpha(a) \Big)$.
Hence $\delta_\theta n_\alpha(b) = \delta_\alpha \big( n'_\alpha(c) - n'_\alpha(a) \big)$ for any $\alpha \in \Delta$.
Hence, we get:
$$
\begin{array}{rl}
\delta_c l''_c = & \delta_\theta l_{-\theta} + \sum_\alpha \delta_\alpha n'_\alpha(c) l_\alpha \\
= & \Big( \delta_\theta l_{-\theta} + \sum_\alpha \delta_\alpha n'_\alpha(a) l_\alpha \Big) + \sum_\alpha \delta_\alpha \Big( n'_\alpha(c) - n'_\alpha(a) \Big) l_\alpha\\
= &\delta_a l''_a + \delta_\theta l'_b
\end{array}
$$
Hence $l''_c = l''_a + l'_b  \geq l_a + l'_b$ by induction hypothesis.
By axiom (C1) and because $l'_b \geq l_b$, we get $l''_c \geq l_a + l_b \geq l_{a+b} = l_c$.

Now we do an induction on height for all roots of $\Phi_{\mathrm{nd}}$.
Basis: consider the lowest root $-h$.
Because $\Phi_\mathrm{nd}$ is non-simply laced, there exist two short roots $a,b  \in \Phi_\mathrm{nd}$ such that $-h = a + b$.
In particular, $\delta_a = \delta_b = \delta_\theta$.
Then:
$$
\begin{array}{rl}
-h = & - \delta_\theta \theta + \sum_\alpha \delta_\alpha n'_\alpha(h) \alpha\\
a = & - \theta + \sum_\alpha \frac{\delta_\alpha}{\delta_a} n'_\alpha(a) \alpha\\
b = & - \theta + \sum_\alpha \frac{\delta_\alpha}{\delta_b} n'_\alpha(b) \alpha\\
(\delta_\theta -2) \theta = & \sum_\alpha \Big( \delta_\alpha n'_\alpha(h) - \frac{\delta_\alpha}{\delta_a} n'_\alpha(a) - \frac{\delta_\alpha}{\delta_b} n'_\alpha(b) \Big) \alpha \\
= & \sum_\alpha (\delta_\theta -2) n_\alpha(\theta) \alpha
\end{array}
$$
Hence, we obtain:
$$\begin{array}{rl}l''_{-h} - l''_a - l''_b = &\displaystyle
\Big( \delta_\theta l_{-\theta} + \sum_\alpha n'_\alpha(-h) \delta_\alpha l_\alpha \Big) - \Big( l_{-\theta} + \sum_\alpha \frac{\delta_\alpha}{\delta_a} n'_\alpha(a) l_\alpha \Big) \\
&- \Big( l_{-\theta} + \sum_\alpha \frac{\delta_\alpha}{\delta_b} n'_\alpha(b) l_\alpha \Big)\\
= & (\delta_\theta -2) l_{-\theta} + \sum_\alpha \Big( \delta_\alpha n'_\alpha(-h) - \frac{\delta_\alpha}{\delta_a} n'_\alpha(a) - \frac{\delta_\alpha}{\delta_b} n'_\alpha(b) \Big) l_\alpha\\
= & (\delta_\theta -2) (l_{-\theta} + l'_\theta)
\end{array}$$
Because $\delta_\theta =d' \geq 2$ and $l_{-\theta} + l'_\theta \geq l_{-\theta} + l_\theta \geq 0$, we have $l''_{-h} \geq l''_{a} + l''_{b}$.
By the case of short roots, we know that $l''_{a} \geq l_a$ and $l''_b \geq l_b$.
Hence, by axiom (C1), we have $l''_{-h} \geq l_a + l_b \geq l_{a+b} = l_{-h}$.

Induction step: we consider the length of a root $c \neq -h$.
The case of short roots has been treated.
Let $c \neq -h \in \Phi_{\mathrm{nd}}$ be a long root and we assume that $l''_{a} \geq l_{a}$ for any root $a$ lower than $c$ in $\Phi_{\mathrm{nd}}$.
We have $c = c^\delta = - \delta_\theta \theta + \sum_\alpha n'_\alpha(c) \delta_\alpha \alpha$.
By Lemma \ref{lem:induction:root:system}, there exist $a \in \Phi_{\mathrm{nd}}$ and $b \in \Phi_{\mathrm{nd}}^+$ such that $c = a+b$.

If $a$ is long, we have $a = a^\delta = - \delta_\theta \theta + \sum_\alpha n'_\alpha(a) \delta_\alpha \alpha$.
Hence, $\delta_\alpha n'_\alpha(c) = \delta_\alpha n'_\alpha(a) + n_\alpha(b)$.
As a consequence, $l''_c = l''_a + l'_b$.
By induction hypothesis, $l''_a \geq l_a$ because $c$ is strictly higher than $a$.
Hence $l''_c \geq l_a + l'_b \geq l_a + l_b \geq l_{a+b} = l_c$ by axiom (C1).

Otherwise, $a$ is a short root, so that $\delta_\alpha = \delta_\theta = d'$.
Hence $a = - \theta + \sum_\alpha \frac{\delta_\alpha}{\delta_\theta} n'_\alpha(a) \alpha$.
We have:
$
0 = a + b -c = (\delta_\theta -1) \theta + \sum_\alpha \Big( \frac{\delta_\alpha}{\delta_\theta} n'_\alpha(a) + n_\alpha(b) - n'_\alpha(c) \delta_\alpha \Big) \alpha
$.
By uniqueness of coefficients, for any $\alpha \in \Delta$, we have $(\delta_\theta -1) n_\alpha(\theta) = \frac{\delta_\alpha}{\delta_\theta} n'_\alpha(a) + n_\alpha(b) - n'_\alpha(c) \delta_\alpha$.
Hence $l''_c-l''_a-l'_b = (\delta_\theta -1) l_{-\theta} + \sum_{\alpha} (\delta_\theta -1) n_\alpha(\theta) l_\alpha = (\delta_\theta -1)(l_{-\theta} + l'_\theta)$.
Because $l_{-\theta} + l'_\theta \geq l_{-\theta} + l_{\theta} \geq 0$ by axiom (C2), we obtain $l''_c \geq l''_a + l'_b$.
By induction hypothesis, $l''_a \geq l_a$.
Hence $l''_c \geq l_a + l_b \geq l_{a+b} = l_c$ by axiom (C1).
This finishes the induction.

Finally if $c$ is a multipliable root, then $l''_{2c} = 2 l''_c \geq 2 l_c \geq l_{2c}$ by axiom (C0).
This finishes the proof of (2).

(3) We now establish inclusions $U_{c,l''_c} \subset [H,H]$ of valued root groups, in the order from the longest roots to the shortest roots.
According to $\Phi$ is a reduced root system or not, there are one, two or three distinct length of roots.

Let $c \neq - \theta$ be a root.
Write it as a sum of two non-collinear roots $c = a + b$.
We want to apply Proposition \ref{prop:generated:commutator:root:group}, with suitable values $l''_a \in \Gamma_a$, $l'_b \in \Gamma_b$ and $\widehat{l_c} \in \Gamma_c$ such that $l''_c \geq \widehat{l_c} = l''_a + l'_b$, to prove that $U_{c,l''_c} \subset [H,H]$.
Because in \ref{prop:generated:commutator:root:group}, there remains a term $v''$, we have to be careful in the order of the steps of this proof. We proceed step by step from the longest length to the shortest length of the roots, and we treat the case, when it happens, of $c = -h \neq - \theta$ separately, at the end.
We denote by $(a,b) = \{r a + s b ,\ r,s \in \mathbb{N} \} \cap \Phi$ and by $\Phi(a,b) = (\mathbb{Z} a + \mathbb{Z} b) \cap \Phi$.
Be careful that in general, $\Phi(a,b) \neq (\mathbb{R} a + \mathbb{R} b) \cap \Phi$.

\textbullet\ \textbf{Case of a divisible root:} Suppose that $c \neq -h$ is a divisible root. Hence $\Phi$ is non-reduced and $\delta_c = \delta_\theta = d'= 2$.
Moreover $2\theta = h$.
By Lemma \ref{lem:induction:root:system} applied to ${\Phi}_{\mathrm{nm}}$, there exist non-collinear roots $a,b \in \Phi_{\mathrm{nm}}$ such that $b \in {{\Phi}_{\mathrm{nm}}}^+$ and $c = a + b$.
Moreover, $a,b$ have to be non-divisible and we have $\delta_a = \delta_b = 1$.
As above, one can show again that $l''_c = 2 l''_{\frac{c}{2}} = l''_a + l'_b$.
By Proposition \ref{prop:generated:commutator:root:group}, for any $u \in U_{c,l''_c}$, there exist elements $v \in U_{a,l''_a}$ and $v' \in U_{b,l'_b}$ such that $u = [v,v']$.
Hence $U_{c,l''_c} \subset [H,H]$.

\textbullet\ \textbf{Case of a non-divisible long root:} Let $c$ be a long root of ${\Phi}_{\mathrm{nd}}$.
Then $\delta_c = 1$ by definition.
Suppose that $c = c^\delta \not\in \{ -\theta, -h \}$.
By Lemma \ref{lem:induction:root:system} applied to ${\Phi}_{\mathrm{nd}}$, there exist non-collinear roots $a,b \in \Phi$ such that $b \in {{\Phi}_{\mathrm{nd}}}^+$ and $c = a + b$.

\noindent \textbf{First subcase: $\Phi(a,b)$ is of type $A_2$.}
We have $(a,b) = \{a,b,a+b\}$ and we have shown in (2) that $l''_c \geq l''_a+l'_b$.
By Proposition \ref{prop:generated:commutator:root:group}, for any $u \in U_{c,l''_c}$, there exist elements $v \in U_{a,l''_a}$ and $v' \in U_{b,l'_b}$ such that $u = [v,v']$.
Hence $U_{c,l''_c} \subset [H,H]$ because $l''_a \geq l_a$ and $l'_b \geq l_b$.

\noindent \textbf{Second subcase: $\Phi(a,b)$ is of type $B_2$ or $G_2$.}
We have $(a,b) = \{a,b,a+b\}$ and $\delta_a = \delta_b = \delta_\theta$ because in this case, necessarily, the long root $c$ is the sum of two short roots.
We have shown that $l''_c \geq l''_a + l'_b$.
By Proposition \ref{prop:generated:commutator:root:group}, for any $u \in U_{c,l''_c}$, there exist elements $v \in U_{a,l''_c-l'_b}$ and $v' \in U_{b,l'_b}$ such that $u = [v,v']$.
Hence $U_{c,l''_c} \subset [H,H]$.

\noindent \textbf{Third subcase: $\Phi(a,b)$ is of type $BC_2$.}
Then $a$ and $b$ are multipliable, and we have $\delta_a = \delta_b = 2$.
If $a \neq -\theta$, we define $a' = a-b \in {\Phi}_{\mathrm{nm}}$ and $b'= 2b \in {\Phi}_{\mathrm{nm}}$.
Then $a'$ is a long non-divisible root and $b'$ is a divisible root.
We have $\delta_{a'} = \delta_c = 1$ and $2a'+b' = 2a$.
Hence $a'= -\delta_\theta \theta + \sum_\alpha n'_\alpha(a') \delta_\alpha \alpha$ and $b' = 2b = \sum_\alpha 2 n_\alpha(b) \alpha$.
For any $\alpha \in \Delta$, we obtain $n'_\alpha(c) \delta_\alpha = n'_\alpha(a') \delta_\alpha + 2 n_\alpha(b)$.
Hence $l''_c = \delta_\theta l_{-\theta} + \sum_\alpha n'_\alpha(c) \delta_\alpha l_\alpha = l''_{a'} + 2 l'_b = l''_{a'} + l'_{b'}$.

We have $-2 \theta + \sum_\alpha n'_\alpha(a') \delta_\alpha \alpha = a' = a + b  = \Big( -\theta + \sum_\alpha \frac{\delta_\alpha}{2} n'_\alpha(a) \alpha \Big) + \sum_\alpha n_\alpha(b) \alpha$.
For any $\alpha \in \Delta$, we obtain $n'_\alpha(a') \delta_\alpha - n_\alpha(\theta) = \frac{\delta_\alpha}{2} n'_\alpha(a) + n_\alpha(b)$.
Hence:
$$\begin{array}{rl}l''_{a'} + l'_b = & \delta_\theta l_{-\theta} + \sum_\alpha \Big(n'_\alpha(a') \delta_\alpha + n_\alpha(b) \Big) l_\alpha\\
= & 2 l_{-\theta} +\sum_\alpha \Big( \frac{\delta_\alpha}{2} n'_\alpha(a) + n_\alpha(\theta) \Big) l_\alpha\\
= & 2 l_{-\theta} + \frac{1}{2} (2 l''_a - 2 l_{-\theta} ) + l'_\theta\\
= & (l_{-\theta} + l'_\theta) + \frac{1}{2} l''_{2a}\end{array}$$
Because $l_{-\theta} + l'_\theta \geq 0$, we get $2 l''_{a'} + l'_{b'} = 2 (l''_{a'} +l'_b) \geq l''_{2a}$.
By Proposition \ref{prop:generated:commutator:root:group}, for any $u \in U_{c,l''_c}$, there exist elements $v \in U_{a',l''_{a'}}$ and $v' \in U_{b,l'_{b'}}$ and $v'' \in U_{2a'+b',2l''_{a'}+l'_{b'}}$ such that $u = [v,v']v''$.
We have already shown, because $2a' + b'=2a \neq -2 \theta$ is a divisible root, that the group $U_{2a'+b',2l''_{a'}+l'_{b'}} \subset U_{2a,l''_{2a}}$ is a subgroup of $[H,H]$.
Hence $U_{c,l''_c} \subset [H,H]$.

If $a = -\theta$, we define $a' = 2a \in {\Phi}_{\mathrm{nm}}$ and $b'= b-a =b + \theta\in {\Phi}_{\mathrm{nm}}^+$.
In the same way, we obtain $l''_c = l''_{a'} + l'_{b'}$ and $l''_{a'} + 2 l'_{b'} = 2 l''_{b} = l''_{b'}$.
By Proposition \ref{prop:generated:commutator:root:group}, for any $u \in U_{c,l''_c}$, there exist elements $v \in U_{a',l''_{a'}}$ and $v' \in U_{b,l'_{b'}}$ and $v'' \in U_{a'+2b',l''_{a'}+2l'_{b'}}$ such that $u = [v,v']v''$.
We have already shown, in the case of a divisible root, that the group $U_{a'+2b',l''_{a'}+l'_{2b'}} = U_{2b,l''_{2b}}$ is a subgroup of $[H,H]$.
Hence $U_{c,l''_c} \subset [H,H]$.

\textbullet\ \textbf{Case of a short root:} Let $c \in {\Phi}_{\mathrm{nd}}$ be a short root of $c \in {\Phi}_{\mathrm{nd}}$.
Then $\delta_c = \delta_{\theta}$ by definition.
Suppose that $c \neq -\theta$ and that $-c^D$ is not the highest root of $\Phi_{\mathrm{nd}}^D$.
By Lemma \ref{lem:short:root:combinatorial} applied to ${\Phi}_{\mathrm{nd}}$, there exist non-collinear roots $a,b \in \Phi$ such that $b \in {{\Phi}_{\mathrm{nd}}}^+$, the root $a$ is short and $c = a + b$.

\noindent \textbf{First subcase: case of two short roots $a$ and $b$.}
We have $\delta_a = \delta_b =\delta_c=\delta_{\theta}$ and we have shown in (2) that $l''_c = l''_a + l'_b$.
The rank $2$ root subsystem $\Phi(a,b)$ is of type $A_2$ or $G_2$.
Moreover, when $\Phi(a,b)$ is of type $G_2$, we have $(a,b) = \{a, b, a+b, 2a+b, a+2b\}$.
By Proposition \ref{prop:generated:commutator:root:group}, for any $u \in U_{c,l''_c}$, there exist elements $v \in U_{a,l''_{a}}$ and $v' \in U_{b,l'_{b}}$ and $v'' \in U_{2a+b,2l''_{a}+l'_{b}} U_{a+2b,l''_a+2l'_b}$ if $\Phi(a,b)$ is of type $G_2$, $v''=1$ if $\Phi(a,b)$ is of type $A_2$, such that $u = [v,v']v''$.

It remains to prove that $v'' \in [H,H]$.
In the $G_2$ case, we have $\delta_{2a+b} = \delta_{a+2b} = 1$.
Moreover, $2a + b = 2 \Big( - \theta + \sum_{\alpha} \frac{\delta_\alpha}{\delta_a} n'_\alpha(a) \alpha \Big) + \sum_\alpha n_\alpha(b) \alpha = - \delta_\theta \theta + \sum_\alpha \Big( 2 \frac{\delta_\alpha}{\delta_a} n'_\alpha(a) + n_\alpha(b) + (\delta_\theta -2) n_\alpha(\theta) \Big) \alpha$.
We have:
$$\begin{array}{rl}
l''_{2a+b} = & \delta_\theta l_{-\theta} + \sum_\alpha \Big( 2 \frac{\delta_\alpha}{\delta_a} n'_\alpha(a) + n_\alpha(b) + (\delta_\theta -2) n_\alpha(\theta) \Big) l_\alpha\\
= &\delta_\theta l_{-\theta} + \frac{2}{\delta_a} (\delta_a l''_a - \delta_\theta l_{-\theta}) + l'_b + (\delta_\theta - 2) l'_\theta\\
= &2 l''_a + l'_b + (\delta_\theta -2) (l_{-\theta} + l'_\theta)
\end{array} $$
In the same way, one can show that $l''_{a+2b} = l''_{a} + 2 l'_b + (\delta_\theta - 1) (l'_\theta + l_{-\theta})$.

If $\delta_\theta = 1$, because $l_{-\theta} + l'_\theta \geq 0$, we get $l''_{2a+b} \leq 2 l''_a + l'_b$ and $l''_{a+2b} = l''_a + 2l'_b$.
Hence, we get $U_{2a+b,l''_{2a+b}} \supset U_{2a+b,2l''_a+l'_b}$ and $U_{a+2b,l''_{a+2b}} = U_{a+2b,l''_a+2l'_b}$.

Otherwise, $\delta_\theta = 3$ and $G$ is a trialitarian $D_4$.
In that case, we assumed that $l_{-\theta} + l'_\theta \leq \omega(\varpi_{L'}) = 0^+ \in \Gamma_{L'}$.
Because $l''_{a+2b}, l''_{2a+b} \in \Gamma_{L_d} = 3 \Gamma_{L'}$, we obtain that $0 \leq (\delta_\theta - 1) (l'_\theta + l_{-\theta}) < 3 \omega(\varpi_{L'}) = 0^+ \in \Gamma_{L_d}$.
The same is for $(\delta_\theta - 1) (l'_\theta + l_{-\theta})$.
Hence, we have the equalities of root groups:
$U_{a+2b, l''_a + 2 l'_b} = U_{a+2b,l''_{a+2b} + (\delta_\theta - 1) (l'_\theta + l_{-\theta})} = U_{a+2b,l''_{a+2b}}$
and $U_{2a+b, 2l''_a + l'_b} = U_{2a+b,l''_{2a+b} + (\delta_\theta - 2) (l'_\theta + l_{-\theta})} = U_{2a+b,l''_{2a+b}}$.

In both cases, because $2a+b$ and $a+2b$ are long and different from $-h$, we have shown that the root groups $U_{2a+b,l''_{2a+b}}$ and $U_{a+2b,l''_{a+2b}}$ are contained in $[H,H]$.
Thus, $v'' \in [H,H]$.
Hence $U_{c,l''_c} \subset [H,H]$.

\noindent \textbf{Second subcase: $a$ is short and $b$ is long.}
We have $\delta_a = \delta_c=\delta_{\theta}$ and $\delta_b =1$.
The rank $2$ root subsystem $\Phi(a,b)$ is of type $B_2$ or $BC_2$.
Precisely, we have $(a,b) = \{a, b, a+b, 2a+b\}$ if $\Phi$ is a reduced root system and $(a,b) = \{a, b, a+b, 2a, 2a+b, 2a+2b\}$ otherwise.
We have $\delta_a = \delta_c = \delta_\theta$ and $\delta_b = \delta_{2a+b} = 1$.
We have $\delta_c c = \delta_\theta \Big( - \theta + \sum_\alpha \big( \frac{\delta_\alpha}{\delta_a} n'_\alpha(a) + n_\alpha(b) \big) \alpha \Big)
= -\delta_\theta \theta + \sum_\alpha \big( \delta_\alpha n'_\alpha(a) + \delta_\theta n_\alpha(b) \big) \alpha$.
Hence $\delta_c l''_c = \delta_a l''_a + \delta_\theta l'_b$.
Thus $l''_c = l''_a + l'_b$.
By Proposition \ref{prop:generated:commutator:root:group}, for any $u \in U_{c,l''_c}$, there exist elements $v \in U_{a',l''_{a'}}$ and $v' \in U_{b,l'_{b'}}$ and $v'' \in U_{2a+b,2l''_{a}+l'_{b}}$ such that $u = [v,v']v''$.

It remains to check that $v'' \in [H,H]$.
We have:
$$\begin{array}{rl}
\delta_{2a+b} (2a+b) = 2a+b = &2 \Big(- \theta + \sum_\alpha \frac{\delta_\alpha}{\delta_a} n'_\alpha(a) \alpha \Big) + \sum_\alpha n_\alpha(b) \alpha\\
= & - \delta_\theta \theta + \sum_\alpha \Big( \delta_\alpha \frac{2}{\delta_a} n'_\alpha(a) + n_\alpha(b) + (\delta_\theta-2) n_\alpha(\theta) \Big) \alpha
\end{array}$$
Hence:
$$\begin{array}{rl}
l''_{2a+b} =& \delta_\theta l_{-\theta} + \frac{2}{\delta_a} (\delta_a l''_a - \delta_\theta l_{-\theta} ) + l'_b + (\delta_\theta -2) l'_\theta\\
= &\delta_\theta l_{-\theta} + 2l''_a -2 l_{-\theta} + l'_b + (\delta_\theta -2) l'_\theta\\
= & 2l''_a + l'_b + (\delta_\theta -2)(l_{-\theta} + l'_{\theta})
\end{array}$$
Because $\delta_\theta \in \{1,2\}$ and $l_{-\theta} + l'_{\theta} \geq 0$, we obtain the inequality $l''_{2a+b} \leq 2 l''_a + l'_b$.
Since $2a+b$ is a long root of ${\Phi}_\mathrm{nd}$,
we have already shown that $U_{2a+b,2l''_{a}+l'_{b}} \subset U_{2a+b,l''_{2a+b}} \subset [H,H]$.
Hence $v'' \in [H,H]$ and it follows that $U_{c,l''_c} \subset [H,H]$.

Now, two cases of roots may remain: the negative root $c$ such that $-c^D$ is the highest root of $\Phi^D$ when $h=\theta$;
and the negative root $c=-h$ when $h \neq \theta$.

\textbullet\ \textbf{The lowest dual root:}
Assume that $c$ is the negative root of $\Phi_{\mathrm{nd}}$ such that $-c^D$ is the highest root of $\Phi_{\mathrm{nd}}^D$ and $h=\theta \neq -c$ (this case appears only if $L'/L_d$ is unramified and $\Phi$ is not a simply laced root system).
In this case, we have $\delta_a=\delta_b = \delta_c=\delta_{\theta}=1$ and
the rank $2$ root subsystem $\Phi(a,b)$ is reduced.
By Lemma \ref{lem:induction:root:system}(2), there exists $a \in \Phi_{\mathrm{nd}}^-$ and $b \in \Phi_{\mathrm{nd}}^+$ such that $c=a+b$.
If $a$ is short, we can proceed as before. Hence we assume that $a$ is a long root, $b$ and $c$ are short roots.

If $\Phi(a,b)$ is of type $B_2$, then $(a,b) = \{a, b, a+b, a+2b\}$ and we have the equalities $l''_{a+b} = l''_a + l'_b$ and $l''_{a+2b} = l''_a + 2 l'_b$.
By Proposition \ref{prop:generated:commutator:root:group}, for any $u \in U_{c,l''_c}$, there exist elements $v \in U_{a,l''_{a}}$ and $v' \in U_{b,l'_{b}}$ and $v'' \in U_{a+2b,l''_{a}+2l'_{b}}$ such that $u = [v,v']v''$.
Since $a+2b$ is a long root of ${\Phi}_\mathrm{nd} = \Phi$,
we have already shown that $U_{a+2b,l''_{a}+2l'_{b}} = U_{a+2b,l''_{a+2b}} \subset [H,H]$.
Hence $U_{c,l''_c} \subset [H,H]$.

If $\Phi(a,b)$ is of type $G_2$, then $(a,b) = \{a, b, a+b, a+2b, a+3b, 2a+3b\}$
We have the equalities $l''_{a+b} = l''_a + l'_b$, $l''_{a+2b} = l''_a + 2 l'_b$ and $l''_{a+3b} = l''_a + 3 l'_b$.
Moreover, we have $l''_{2a+3b} = 2l''_a + 3 l'_b - (l_{-\theta} + l'_\theta) \leq 2l''_a + 3 l'_b$.
By Proposition \ref{prop:generated:commutator:root:group}, for any $u \in U_{c,l''_c}$, there exist elements $v \in U_{a,l''_{a}}$ and $v' \in U_{b,l'_{b}}$ and $v'' \in U_{a+2b,l''_{a}+2l'_{b}} U_{a+3b,l''_{a}+3l'_{b}} U_{2a+3b,2l''_{a}+3l'_{b}}$ such that $u = [v,v']v''$.
Since $a+3b$ and $2a+3b$ are long roots of ${\Phi}_\mathrm{nd} = \Phi$,
we have already shown that $U_{a+3b,l''_{a}+3l'_{b}} = U_{a+3b,l''_{a+3b}} \subset [H,H]$ and that $U_{2a+3b,2l''_{a}+3l'_{b}} \subset U_{2a+3b,l''_{2a+3b}} \subset [H,H]$.
Since $a+2b \neq -\theta$ can be written as the sum of the two short roots $b$ and $a+b$,
we have shown that $U_{a+2b,l''_{a}+2l'_{b}} = U_{a+2b,l''_{a+2b}} \subset [H,H]$.
Hence $U_{c,l''_c} \subset [H,H]$.

\textbullet\ \textbf{The lowest root:}
To conclude, it remains to treat the case, when it appears, of the root $-h \neq -\theta$ where $h$ is the highest root of $\Phi$ (this appears only for $G$ of type ${^2}A_{2l+1}$, ${^2}D_{l+1}$, ${^2}E_6$, ${^3}D_4$ or ${^6}D_4$ with a ramified extension $L' / L_d$).
In this case, we have $\delta_{\theta} > 1$ and $h$ is a long root.
In particular, the integer $(\delta_{\theta} -2)$ is non-negative.
We write $h$ as a sum $h = c = a+b$ of two short roots $a$ and $b$, so that $\delta_a = \delta_b = \delta_\theta$ and $\delta_c=1$.
Moreover $(a,b) = \{a,b,a+b\}$.
We have:
$$\begin{array}{rl}
c = a+b = & \Big( -\theta + \sum_\alpha \frac{\delta_\alpha}{\delta_a} n'_\alpha(a) \alpha \Big) + \Big( -\theta + \sum_\alpha \frac{\delta_\alpha}{\delta_b} n'_\alpha(b) \alpha \Big)\\
= &-2 \theta + \sum_\alpha \Big( \frac{\delta_\alpha}{\delta_\theta} n'_\alpha(a) +\frac{\delta_\alpha}{\delta_\theta} n'_\alpha(b) \Big) \alpha\\
= & - \delta_\theta \theta + \sum_\alpha \Big( \frac{\delta_\alpha}{\delta_\theta} n'_\alpha(a) +\frac{\delta_\alpha}{\delta_\theta} n'_\alpha(b)  + (\delta_\theta - 2) n_\alpha(\theta) \Big) \alpha
\end{array}$$
Hence we obtain:
$$\begin{array}{rl}
l''_c = & \delta_\theta l_{-\theta}
+ \frac{1}{\delta_\theta} (\delta_a l''_a - \delta_\theta l_{-\theta} )
+ \frac{1}{\delta_\theta} (\delta_b l''_b - \delta_\theta l_{-\theta} )
+ (\delta_\theta - 2) l'_\theta\\
= & l''_a + l''_b + (\delta_\theta - 2)(l_{-\theta} + l'_\theta)\\
\geq & l''_a + l''_b
\end{array}$$
By Proposition \ref{prop:generated:commutator:root:group}, for any $u \in U_{c,l''_c} \subset U_{c,l''_a + l''_b}$, there exist elements $v \in U_{a,l''_{a}}$ and $v' \in U_{b,l''_{b}}$ such that $u = [v,v']$.
This finishes the proof.
\end{proof}

\begin{Rq}
Proposition \ref{prop:generated:commutator:positive:root:groups} and Proposition \ref{prop:generated:commutator:negative:root:groups} do not restrict the choice of the basis $\Delta$ but only the choice of values $l_a$.
In fact, the conditions $l_a \in \Gamma_a$ for any $a \in \Delta$ and $l_{-\theta} \in \Gamma_{-\theta}$ limit the available choices for the basis $\Delta$.
\end{Rq}

\begin{Lem}
\label{lem:base:change:trick}
Let $\Phi$ be a non-reduced root system and $\Delta$ be a basis of $\Phi$.
Let $a \in \Delta$ be the multipliable simple root.
Let $\theta$ be the half highest root of $\Phi$ relatively to the basis $\Delta$.
Then $\Delta' = \left( \Delta \cup \{-\theta\} \right) \setminus \{a\}$ is another basis of $\Phi$;
and $-a$ is the half highest root of $\Phi$ relatively to the basis $\Delta'$.
\end{Lem}

\begin{proof}
We consider the following Euclidean geometric realisation of the root system $\Phi = \{\pm e_i,\ 1 \leq i \leq l \} \cup \{\pm e_i \pm e_j,\ 1 \leq i  < j\leq l \} \cup \{\pm 2 e_i,\ 1 \leq i \leq l \}$ where $(e_i)$ denotes the canonical basis of the Euclidean space $\mathbb{R}^l$.
We denote by $a_i = e_i - e_{i+1}$ for any $1 \leq i \leq l-1$ and by $a_l = e_l$.
The set $\Delta = \{a_1,\dots,a_l\}$ is a basis of $\Phi$ and $\theta = e_1 = a_1 + \dots + a_l$ is the half highest root of $\Phi$.

Let $w \in \mathrm{GL}_l(\mathbb{R})$ be the element of the Weyl group $W(\Phi)$ defined by $w(e_i) = - e_{l-i+1}$.
We observe that $w$ stabilises $\Delta \setminus \{a_l\}$,
that $w(-\theta) = a_l$ and that $w(a_l) = -\theta$.

If $D$ is a half-space of $\mathbb{R}^l$ defining the basis $\Delta$,
then $w(D)$ is also a half-space of $\mathbb{R}^l$ and it defines the basis $\Delta' = \left( \Delta \setminus \{a_l \} \right) \cup \{-\theta\}$.
The half highest root of $\Phi$ relatively to $\Delta'$ is then $-a_l$.
\end{proof}

\subsubsection{Lower bounds for valued root groups of the Frattini subgroup}
\label{sec:lower:bounds:frattini:subgroup}

We want to apply Propositions \ref{prop:generated:commutator:positive:root:groups} and \ref{prop:generated:commutator:negative:root:groups} to the maximal pro-$p$ subgroup $P$ corresponding to the fundamental alcove $\mathbf{c}_{\mathrm{af}}$ described in Section \ref{sec:description:apartment}.

\begin{Thm}
\label{thm:minoration:derived:group:with:root:groups}
Assume that the irreducible relative root system $\Phi$ is of rank $l \geq 2$ and that the residue characteristic $p$ of $K$ satisfies Hypothesis \ref{hypothesis:on:residue:characteristic}.
Let $P$ be a maximal pro-$p$ subgroup of $G(K)$ and let $\mathbf{c}$ be the (unique) alcove fixed by $P$.
For any root $a \in \Phi$,
if the wall $\mathcal{H}_{a,f'_{\mathbf{c}}(a)}$ (this notation has been defined in Section \ref{sec:walls}) contains a panel of $\mathbf{c}$,
then we have $[P,P] \supset U_{a,f'_{\mathbf{c}}(a)^+}$;
otherwise, we have  $[P,P] \supset U_{a,f'_{\mathbf{c}}(a)}$.
\end{Thm}

\begin{proof}
We normalize $\Gamma_{L'} = \mathbb{Z}$.
Up to conjugation, we can assume that $\mathbf{c} = \mathbf{c}_{\mathrm{af}}$ is the fundamental alcove, defined in Section \ref{sec:description:alcove}, and bounded by the following walls:
\begin{itemize}
\item $\mathcal{H}_{a,0}$ for all simple roots $a \in \Delta$;
\item $\mathcal{H}_{-\theta,1}$ if $\Phi$ is reduced;
\item $\mathcal{H}_{-\theta,\frac{1}{2}}$ if $\Phi$ is non-reduced.
\end{itemize}

For any root $a \in \Phi$, we have the following value:
\begin{itemize}
\item $f'_{\mathbf{c}}(a) = 0$ if $a \in \Phi^+$;
\item $f'_{\mathbf{c}}(a) = \frac{\delta_{\theta}}{\delta_a} \in \{1, d' \}$ if $a \in \Phi^-$ and $\Phi$ is reduced;
\item $f'_{\mathbf{c}}(a) = \frac{1}{\delta_a} \in \{\frac{1}{2}, 1\}$ if $a \in \Phi_{\mathrm{nd}}^-$ and $\Phi$ is non-reduced.
\end{itemize}

The wall bounding the alcove $\mathbf{c}$ are directed by the relative roots $\Delta \cup \{-\theta\}$.
Hence, for any $a \in \Delta \cup \{-\theta\}$, we get $f_{\mathbf{c}}(a) = f'_{\mathbf{c}}(a) \in \Gamma_a$.
Moreover, $f_{\mathbf{c}}(-\theta)=1$ and $l'_\theta = 0$ so that the sum satisfies $f_{\mathbf{c}}(-\theta) + l'_\theta = 1 = \omega(\varpi_{L'})$.
As a consequence, we can apply Propositions \ref{prop:generated:commutator:positive:root:groups} and \ref{prop:generated:commutator:negative:root:groups} to the group $P$ and the values $l_c = f_{\mathbf{c}}(c)$ where $c \in \Phi$.

For any non-divisible non-simple positive root $b \in \Phi^+_{\mathrm{nd}} \setminus \Delta$,
by Proposition \ref{prop:generated:commutator:positive:root:groups},
we get $l'_b = 0$.
Hence $[P,P] \supset U_{b,0} = U_{b,l_b}$.

For any root $c \in \Phi^- \setminus \{ - \theta, -2\theta \}$,
by Proposition \ref{prop:generated:commutator:negative:root:groups},
we get $\delta_c l''_c = \delta_{\theta} f'_{\mathbf{c}}(- \theta)$.
If $\Phi$ is reduced, then we have $l''_c = \frac{\delta_{\theta}}{\delta_c} = f'_{\mathbf{c}}(c)$.
If $\Phi$ is non-reduced, then we have $l''_c = \frac{1}{\delta_c} = f'_{\mathbf{c}}(c)$ because $\delta_{-\theta} l_{-\theta} = 1$.
Hence $[P,P] \supset U_{c,l_c}$.

We suppose that $\Phi$ is reduced.
Let $a \in \Delta \cup \{-\theta \}$.
Then, by Proposition \ref{prop:frattini:computation:reduced:case},
we know that $[P,P] \supset U_{a,l_a^+}$.

We suppose that $\Phi$ is non-reduced.
Let $a \in \Delta$.
By Proposition \ref{prop:generated:commutator:negative:root:groups},
we get $\delta_a l''_a = \delta_{\theta} f'_{\mathbf{c}}(- \theta)$.
We have $l''_a = \frac{1}{\delta_a} = 0^+=f'_{\mathbf{c}}(a)^+$.
Indeed, if $a$ is mutlipliable, $l''_a = \frac{1}{2}$;
otherwise $l''_a=1$ is the smallest positive value of $\Gamma_a$.
Hence $[P,P] \supset U_{a,l_a^+}$.

Finally, when $\Phi$ is non-reduced,
we can apply Lemma \ref{lem:base:change:trick} to exchange the roles of the multipliable simple root $a \in \Delta$ and the opposite of the half highest root $-\theta$.
We write $\theta = \sum_{b \in \Delta} n_b b$ where $n_b \in \mathbb{N}^*$,
so that $-\theta = \theta+(-2\theta) = n_a a + \sum_{b \in \Delta \setminus \{ a \}} n_b b + 2 (-\theta)$.
Thus, by applying Proposition \ref{prop:generated:commutator:negative:root:groups} to the basis $\displaystyle \Delta' = \left(\Delta \setminus \{a\}\right)\cup \{-\theta \} $,
we get $l''_{-\theta} = 2 l_{-\theta} = 1 = l_{\theta}^+$.
\end{proof}

\begin{Rq}
As an immediate consequence,
the derived group $[P,P]$ contains $U_{c,f'_{B(\mathbf{c},1) \cap \mathbb{A}}(c)}$ for any root $c \in \Phi$.

In the rank $1$ case, we have a lack of rigidity that could make $[P,P]$ smaller than expected.
Typically, Propositions \ref{prop:generated:commutator:positive:root:groups} and \ref{prop:generated:commutator:negative:root:groups} cannot be applied.
\end{Rq}

\begin{Cor}
We assume that $p \neq 2$ and that the structure constant $c_{1,1;\alpha,\beta}$ are in $\mathcal{O}_K^\times$ for all pairs of non-collinear roots $\alpha,\beta$.
For any non-divisible root $a \in \Phi_{\mathrm{nd}}$ and any maximal pro-$p$ subgroup $P$ of $G(K)$,
we write $P \cap U_a(K) = U_{a,l_a}$ where $l_a \in \Gamma_a$.
If $a \in \Delta \cup \{-\theta\}$,
\begin{itemize}
\item if $a$ is a non-multipliable root or if the extension $L_{a} / L_{2a}$ is ramified,
then we have the equality $[P,P] \cap U_a(K) = U_{a,l_a^+}$.
\item if $a$ is multipliable and if the extension $L_{a} / L_{2a}$ is unramified,
then we have the inclusions $U_{a,l_a^+} \subset [P,P] \cap U_a(K) \subset U_{a,l_a^+} U_{2a,2l_a}$.
\end{itemize}

If $a \in \Phi \setminus (\Delta \cup \{-\theta\})$,
then we have the equality $[P,P] \cap U_a(K) = U_{a,l_a}$.
\end{Cor}

\begin{proof}
This results immediately from Theorem \ref{thm:minoration:derived:group:with:root:groups} and Proposition \ref{prop:favourable:geometric:description}.
\end{proof}

\section{Generating set of a maximal pro-\texorpdfstring{$p$}{p} subgroup} 
\label{sec:generating:set}

As before, $G$ is an almost-$K$-simple quasi-split simply-connected $K$-group and $P$ is a maximal pro-$p$ subgroup of $G(K)$.
In Corollary \ref{cor:minimal:number:topological:generators}, we obtain the minimal number of topological generators of the pro-$p$ Sylow $P$ in the various cases. 

In order to give explicit formulas for these numbers, we introduce the following integers.
We denote by $e'$ the ramification index of $L'/L_d$ and by $f'$ its residue degree; we let $m = \log_p(\mathrm{Card}(\kappa_K))$ so that $\kappa_K \simeq \mathbb{F}_{p^m}$.
Moreover, when $G$ is assumed to be almost-$K$-simple instead of absolutely simple, we denote by $e$ the ramification index of $L_d / K$ and by $f$ its residue degree.

\subsection{The Frattini subgroup}
\label{sec:frattini:subgroup}

In order to compute a minimal generating set of the maximal pro-$p$ subgroup $P$, we know by \cite[1.9]{DixonDuSautoyMannSegal} that is suffices to compute a minimal generating set of the $p$-elementary commutative group $P / \mathrm{Frat}(P)$, where $\mathrm{Frat}(P)$ denotes the Frattini subgroup of $P$.
According to \cite[3.2.9]{Loisel-maximaux}, we know that $P = \left( \prod_{a \in \Phi_{\mathrm{nd}}^-} U_{a,\mathbf{c}} \right) T(K)_b^+ \left(\prod_{a \in \Phi_{\mathrm{nd}}^+} U_{a,\mathbf{c}} \right) $ as directly generated product, where $\mathbf{c}$ is a suitable alcove of $X(G,K)$.
Up to conjugation, we can --- and do --- assume that $\mathbf{c} = \mathbf{c}_{\mathrm{af}}$.

We want to describe the Frattini subgroup $\mathrm{Frat}(P)$, in the same way, in terms of valued root groups $U_{a,\widehat{l_a}}$, with suitable values $\widehat{l_a} \in \mathbb{R}$, and a subgroup of $T(K)_b^+$ that we have to determinate.
Since $P$ is a pro-$p$ group, by \cite[1.13]{DixonDuSautoyMannSegal}, we have $\mathrm{Frat}(P) = \overline{P^p [P,P]}$.
Hence $P / \mathrm{Frat}(P)$ is a $\mathbb{Z} / p \mathbb{Z}$ vector space of dimension $d(P)$ that we want to compute explicitly.

\begin{Thm}[{Descriptions of the Frattini subgroup of a maximal pro-$p$ subgroup: the reduced case}]
\label{thm:frattini:descriptions:reduced:case}
We suppose that the relative root system $\Phi$ is reduced and that $p \neq 2$.
If $\Phi$ is of type $G_2$, we require that $p \geq 5$.
Then:

\textbf{Profinite description:}
The pro-$p$ group $P$ is topologically of finite type and, in particular, $\mathrm{Frat}(P) = P^p [P,P]$.
Moreover, when $K$ is of characteristic $p > 0$, we have $P^p \subset [P,P]$.

\textbf{Description by the valued root groups datum:}
For any $a \in \Phi$, we set:
$$V_{a,\mathbf{c}} = \left\{\begin{array}{cl}U_{a,f_{\mathbf{c}}(a)^+}&\text{ if } a \in \Delta \cup \{-\theta\} \\
U_{a,\mathbf{c}} & \text{ otherwise}\end{array}\right.$$
This group depends only on the root $a \in \Phi$ and the alcove $\mathbf{c} \subset \mathbb{A}$, not on the chosen basis $\Delta$.

We have the following writing, as directly generated product:
$$
\mathrm{Frat}(P) = 
\left( \prod_{-a \in \Phi^{+}} V_{-a,\mathbf{c}} \right)
T(K)_b^+
\left( \prod_{a \in \Phi^+} V_{a,\mathbf{c}} \right)
$$

\textbf{Geometrical description:}
The Frattini subgroup $\mathrm{Frat}(P)$ is the maximal pro-$p$ subgroup of the pointwise stabilizer in $G(K)$ of the combinatorial ball centered at $\mathbf{c}$ of radius $1$.
\end{Thm}

\begin{proof}

For any $a \in \Phi$, we let $l_a = f_{\mathbf{c}}(a)$, so that $l_a \in \Gamma_a$ for any $a \in \Delta \cup \{-\theta\}$ and the map $a \mapsto l_a$ is concave.
We define $\widehat{l_a} = \left\{\begin{array}{cl} l_a^+ & \text{ if } a \in \Delta \cup \{ - \theta \}\\
l_a & \text{ otherwise}
\end{array}\right.$.
We define $Q = \prod_{a \in \Phi^-} U_{a,\widehat{l_a}} \cdot T(K)_b^+ \cdot \prod_{a \in \Phi^+} U_{a,\widehat{l_a}}$.
We prove the chain of inclusions $Q \subset P^p [P,P] \subset \mathrm{Frat}(P) \subset Q$.

The inclusion $P^p [P,P] \subset \overline{P^p [P,P]} = \mathrm{Frat}(P)$ is immediate.

By Corollary \ref{cor:majoration:frattini}, we have $\mathrm{Frat}(P) \subset Q$.

If the reduced irreducible root system $\Phi$ is of rank $l \geq 2$, by Theorem \ref{thm:minoration:derived:group:with:root:groups}, we have $\forall a \in \Phi,\ [P,P] \supset U_{a,\widehat{l_a}}$.
If $\Phi$ is of rank $1$, by Proposition \ref{prop:frattini:computation:reduced:case}, we have $\forall a \in \Phi,\ P^p[P,P] \supset U_{a,\widehat{l_a}}$.
Moreover, by Proposition \ref{prop:frattini:computation:reduced:case}, we also have $T^a(K)_b^+ \subset P^p[P,P]$ for any $a \in \Phi$.
Because $G$ is a simply-connected semisimple group, $T(K)_b^+$ is generated by the groups $T^a(K)_b^+$, hence $T(K)_b^+ \subset P^p[P,P]$.
As a consequence, $Q \subset P^p[P,P]$.

Hence, we obtain (2): $Q = \mathrm{Frat}(P) = P^p[P,P]$.

Moreover, if $K$ is of positive characteristic, by Proposition \ref{prop:frattini:computation:reduced:case} one can replace $[P,P]P^p$ by $[P,P]$ so that we get (1): $Q = [P,P]$.

(3) By Proposition \ref{prop:combinatorial:ball:fixator}, we know that $\mathrm{Frat}(P) = Q$ is the maximal pro-$p$ subgroup of the pointwise stabilizer of the combinatorial closure of the combinatorial unit ball centered in $\mathbf{c}$.
\end{proof}

In the case of a non-reduced root system $\Phi$, we have seen that computation of $[P,P]$ is different from the reduced case because of non-commutativity of root groups.
We have to study this case separately.

\begin{Thm}[{Descriptions of the Frattini subgroup of a maximal pro-$p$ subgroup: the non-reduced case}] 
\label{thm:frattini:descriptions:non:reduced:case}
We suppose that $\Phi$ is a non-reduced root system of rank $l \geq 2$, and that $p \geq 5$.
Then:

\textbf{Profinite description:}
The pro-$p$ group $P$ is topologically of finite type and, in particular, $\mathrm{Frat}(P) = P^p [P,P]$.

\textbf{Description by the valued root groups datum:}
Let $a \in \Phi_{\mathrm{nd}}$ be a non-divisible root.
If $a \not \in \Delta \cup \{-\theta\}$, we set 
$V_{a,\mathbf{c}} = U_{a,\mathbf{c}}$.

If $a \in \Delta \cup \{-\theta\}$, we set:
$$V_{a,\mathbf{c}} = \left\{\begin{array}{cl}
U_{a,f_{\mathbf{c}}(a)^+} & \text{ if } a \text{ is non-multipliable} \\
U_{a,f'_{\mathbf{c}}(a)^+} & \text{ if } a \text{ is multipliable and } L'/L_d \text{ is ramified}\\
U_{a,f'_{\mathbf{c}}(a)^+} & \text{ if } a \text{ is multipliable, } L'/L_2 \text{ is unramified and } f'_{\mathbf{c}}(a) \not\in \Gamma'_a \\
U_{a,f'_{\mathbf{c}}(a)^+} U_{2a,2f'_{\mathbf{c}}(a)} & \text{ if } a \text{ is  multipliable, } L'/L_2 \text{ is unramified and } f'_{\mathbf{c}}(a) \in \Gamma'_a \\
\end{array}\right.$$

Then $\displaystyle \mathrm{Frat}(P) = \left(\prod_{a \in \Phi^-_{\mathrm{nd}}} V_{a,\mathbf{c}} \right) T(K)_b^+ \left( \prod_{a \in \Phi_{\mathrm{nd}}^+} V_{a,\mathbf{c}} \right)$.
\end{Thm}

\begin{proof}
Let $\displaystyle Q = \left(\prod_{a \in \Phi^-_{\mathrm{nd}}} V_{a,\mathbf{c}} \right) T(K)_b^+ \left( \prod_{a \in \Phi_{\mathrm{nd}}^+} V_{a,\mathbf{c}} \right)$.
By Corollary \ref{cor:majoration:frattini}, we have $\mathrm{Frat}(P) \subset Q$.

If $\Phi$ is of rank $l \geq 2$, by Theorem \ref{thm:minoration:derived:group:with:root:groups} and Lemma \ref{lem:non:reduced:derived:valued:root:group}, we have $\forall a \in \Phi,\ [P,P] \supset V = \prod_{a\in \Phi_{\mathrm{nd}}} V_{a,\mathbf{c}}$.
For the multipliable simple root $a$, by Proposition \ref{prop:non:reduced:frattini:rank:one} and Proposition \ref{prop:improvement:upper:bound:bounded:torus:non:reduced:case}, because $f_{\mathbf{c}_{\mathrm{af}}}(a) =0$, we have $\varepsilon=0$, and so $T^a(K)_b^{+} \subset [P,P]$.
For any non-multipliable root $a \in \Phi$, 
by Propositions \ref{prop:frattini:computation:reduced:case} and \ref{prop:improvement:upper:bound:bounded:torus:non:reduced:case}, we have $T^a(K)_b^+ \subset [P,P]$.
Hence, $T(K)_b^+$ is a subgroup of $\mathrm{Frat}(P)$.
As a consequence, we have $Q \subset \mathrm{Frat}(P)$.

Moreover, because $Q$ is an open subgroup of $P$ (of finite index), the Frattini subgroup $\mathrm{Frat}(P) = Q$ is open in $P$.
By \cite[1.14]{DixonDuSautoyMannSegal}, we know that $P$ is topologically of finite type.
By \cite[1.20]{DixonDuSautoyMannSegal}, we deduce $\mathrm{Frat}(P) =P^p [P,P]$.
\end{proof}

\subsection{Minimal number of generators}
\label{sec:minimal:number}

\begin{Cor}[{of Theorems \ref{thm:frattini:descriptions:reduced:case} and \ref{thm:frattini:descriptions:non:reduced:case}}]
\label{cor:frattini:quotient}
We assume $p \neq 2$.

If the root system $\Phi$ is reduced, we assume that, at least, $p \neq 3$ or $\Phi$ is not of type $G_2$.
If the root system $\Phi$ is non-reduced, we assume that $p \geq 5$ and that $\Phi$ is not of rank $1$.

Then $P/\mathrm{Frat}(P)$ is isomorphic to the following direct product of $p$-elementary commutative groups: $\prod_{a \in \Phi} U_{a,\mathbf{c}} / V_{a,\mathbf{c}}$, where the groups $V_{a,\mathbf{c}}$ for $a \in \Phi$ are defined in Theorems \ref{thm:frattini:descriptions:reduced:case} and \ref{thm:frattini:descriptions:non:reduced:case}.
\end{Cor}

\begin{proof}
Let $A = \prod_{a \in \Phi} U_{a,\mathbf{c}} / V_{a,\mathbf{c}}$ be the considered direct product of quotient groups.
Let $B = \left(\prod_{a \in \Phi^-} U_{a,\mathbf{c}} \right) \times T(K)_b^+ \times \left(\prod_{a \in \Phi^+} U_{a,\mathbf{c}} \right)$ be the direct product of the valued root groups with respect to $\mathbf{c} = \mathbf{c}_{\mathrm{af}}$, and of the maximal pro-$p$ subgroup of the bounded torus.
Let $C = \left(\prod_{a \in \Phi^-} V_{a,\mathbf{c}} \right) \times \{1\} \times \left(\prod_{a \in \Phi^+} U_{a,\mathbf{c}} \right)$ be the direct product of the valued root groups provided by Theorems \ref{thm:frattini:descriptions:reduced:case} and \ref{thm:frattini:descriptions:non:reduced:case}.

We want to define a surjective group homomorphism $B \rightarrow P / \mathrm{Frat}(P)$.
Let $\pi : P \rightarrow P / \mathrm{Frat}(P)$ be the quotient homomorphism.
For any inclusion $j_a : U_{a,\mathbf{c}} \rightarrow P$ (resp. $j_0 : T(K)_b^+ \rightarrow P$), we define a group homomorphism $\phi_a = \pi \circ j_a : U_{a,\mathbf{c}} \rightarrow P / \mathrm{Frat}(P)$ (resp. $\phi_0 = \pi \circ j_0)$.
Since $P/\mathrm{Frat}(P)$ is commutative, the multiplication map induces a group homomorphism $\mu : B \rightarrow P/\mathrm{Frat}(P)$. 
Applying \cite[3.2.9]{Loisel-maximaux} to $P$, we deduce that the homomorphism $\mu$ is surjective.

By Theorems \ref{thm:frattini:descriptions:reduced:case}(2) and \ref{thm:frattini:descriptions:non:reduced:case}(2), we get $\ker \mu = C$.
Passing to the quotient, we deduce a group isomorphism $B / C \simeq P / \mathrm{Frat}(P)$.
Furthermore, there is a canonical group isomorphism $A \simeq B / C$.
Hence $P / \mathrm{Frat}(P)$ is isomorphic to $A$.
\end{proof}

Since $P / \mathrm{Frat}(P)$ is a $p$-elementary commutative group,
we deduce that so are the quotient groups $U_{a,\mathbf{c}} / V_{a,\mathbf{c}}$.
Hence, we can compute their dimension as $\mathbb{F}_p$-vector space.
According to \cite[1.9]{DixonDuSautoyMannSegal}, we know that the minimal number of elements in a generating set of a pro-$p$ group is $d(P) = \mathrm{dim}_{\mathbb{F}_p}\big(P/\mathrm{Frat}(P)\big)$.
It can also be computed by $d(P) = \mathrm{dim}_{\mathbb{Z}/p \mathbb{Z}}\big(H^1(P, \mathbb{Z} / p \mathbb{Z})\big)$ according to \cite[4.2 Corollaire 5]{SerreCohomologieGaloisienne}.
We apply this to our maximal pro-$p$ subgroup $P$ of $G(K)$.

\begin{Cor}
\label{cor:minimal:number:topological:generators}
As above we assume that $K$ is a non-Archimedean local field of residue characteristic $p$.
We assume that $G$ is an almost-$K$-simple simply-connected quasi-split $K$-group and that $p \neq 2$.
We keep notations of \ref{not:L:prime}.
Let $n$ be the rank of an irreducible subsystem of the absolute root system $\widetilde{\Phi}(G_{\widetilde{K}}, \widetilde{K})$ and $l$ be the rank of the irreducible relative root system $\Phi(G,K)$.
Let $f$ be the residue degree of $L_d / K$ and $m = \log_{p}\big(\mathrm{Card}(\kappa_K)\big)$.

(1) If $\Phi$ is of type $G_2$ or if $\Phi$ is non-reduced, suppose that $p \geq 5$.
If $L'/L_d$ is ramified, then $d(P) = m f (l+1)$;
if $L'/L_d$ is unramified, then $d(P) = m f (n+1)$.

(2) Suppose that $\Phi$ is of type $BC_1$ and that $p \geq 5$.
If $L'/L_d$ is ramified, then $2 m f \leq d(P) \leq 6 m f$;
if $L'/L_d$ is unramified, then $3 m f \leq d(P) \leq 9 m f$.
\end{Cor}

\begin{Rq}[Summary in terms of quasi-split groups classification]
We recall that $f'$ denotes the residue degree of $L'/L_d$ and that there are, case by case, identities between $d$, $l$ and $n$.
In Corollary \ref{cor:minimal:number:topological:generators}, if the quasi-split group is of type ${^d}X_{n,l}$ (with notations of \cite{TitsBoulder}; Tits indices are not necessary in this study because of quasi-splitness assumption), we have $d(P) = m f \xi$ where:

$$\begin{array}{|l|l|c|}
\hline
\text{Type} & \text{(in)equality} & \text{Assumption}\\
\hline
{^1}X_l,\  l \geq 1,\ X \neq G& \xi = l+1 & p \geq 3\\
\hline
{^1}G_2 & \xi = 3 & p \geq 5\\
\hline
{^2}A_{2 l - 1} ,\  l \geq 2& \xi = f' (l-1) + 2 & p \geq 3\\
\hline
{^2}D_{l+1} ,\  l \geq 3 & \xi = l + f' & p \geq 3\\
\hline
{^2}E_{6} & \xi = 3 + 2 f' & p \geq 3\\
\hline
{^{3}}D_{4} \text{ and } {^{6}}D_{4} &\xi = 2 + f' & p \geq 5\\
\hline
{^2}A_{2 l},\  l \geq 2& \xi = f' l +1& p \geq 5\\
\hline
{^2}A_{2}& f' +1 \leq \xi \leq 3 f'+ 3& p \geq 5\\
\hline
\end{array}$$
\end{Rq}

\begin{proof}
According to \cite[3.1.2]{TitsBoulder}, there exists an absolutely simple group $G'$ such that $G = R_{L_d/K}(G')$, so that $G(K) = G'(L_d)$.
Because $\mathrm{Card}(\kappa_{L_d}) = f \mathrm{Card}(\kappa_{K})$, we can assume that $G$ is absolutely simple, so that $\widetilde{\Phi}$ is irreducible and $m = \log_p\big(\mathrm{Card}(\kappa_{L_d})\big)$.

\textbf{(1) Suppose that $\Phi$ is reduced.}
By definition of the groups $V_{a,\mathbf{c}}$ \ref{thm:frattini:descriptions:reduced:case}(2), we have $U_{a,\mathbf{c}} / V_{a,\mathbf{c}} \simeq \left\{\begin{array}{cl}
X_{a,f_{\mathbf{c}}(a)} & \text{ if } a \in \Delta \cup \{-\theta\} \\
0 & \text{ otherwise}
\end{array}\right.$,
where the quotient groups $X_{a,f_{\mathbf{c}}(a)}$ are defined as in Proposition \ref{prop:first:root:group:quotient}.
Applying Corollary \ref{cor:frattini:quotient}, we write $P / \mathrm{Frat}(P) \simeq \prod_{a \in \Delta \cup \{-\theta\}} X_{a,f_{\mathbf{c}}(a)}$.
We know by Proposition \ref{prop:first:root:group:quotient} that the group $X_{a,f_{\mathbf{c}}(a)}$ is a $\kappa_{L_a}$-vector space of dimension $1$.
The finite field $\kappa_{L_a}$ is of order $p^{m f_a}$ where $f_a$ denotes the residue degree of the extension $L_a / L_d$.
Thus, we obtain $\mathrm{dim}_{\mathbb{F}_p}(P / \mathrm{Frat}(P) ) = \sum_{a \in \Delta \cup \{-\theta\}} m f_a$.
It remains to compute $\xi=\sum_{a \in \Delta \cup \{-\theta\}} f_a$.
Let $a \in \Delta \cup \{-\theta\}$.
If $a$ is a long root, then $L_a = L_d$ and $f_a = 1$.
Otherwise $L_a = L'$ and $f_a = f'$.

Suppose that $L'/L_d$ is ramified.
We know that $\theta^D$ is the highest root of $\Phi^D$ with respect to $\Delta^D$.
Hence $-\theta^D$ is a long root of $\Phi^D$ and $-\theta$ is a short root. 
Thus, $L_{-\theta} = L'$, so that $f_{-\theta} = f' = 1$.
We have $f_a = 1$ for any simple root $a \in \Delta$.
Thus $\xi = \mathrm{Card}(\Delta) + f_{-\theta} = l+1$.

Suppose that $L'/L_d$ is unramified.
We know that $\theta$ is the highest root of $\Phi$ with respect to $\Delta$.
Hence, $-\theta$ is a long root and $L_{-\theta} = L_d$, so that $f_{-\theta} = 1$.
We have $f_a = \mathrm{Card}(a)$ where any simple root $a \in \Delta$ is seen as an orbit of absolute simple roots $\alpha \in \widetilde{\Delta}$.
Thus $\xi =f_{-\theta}+ \sum_{a \in \Delta} f_a =1 + \mathrm{Card}(\widetilde{\Delta}) = 1+n$.


\textbf{Suppose that $\Phi$ is non-reduced of rank $l \geq 2$.}

We have a group isomorphism $P / \mathrm{Frat(P)} \simeq \prod_{b \in \Delta \cup \{-\theta\}} U_{b,l_b} / V_{b}$.
We can express each $U_{b,l_b} / V_{b}$ in terms of $X_{b,l}$ (and of $X_{2b,2l}$ if $b\in\{a,-\theta\}$ is a multipliable root).

\textbf{First case: $b$ is non-multipliable.}
In this case, we have $V_{b} = U_{b,f_{\mathbf{c}}(b)^+}$.
By \ref{prop:first:root:group:quotient}, we know that $U_{b,f_{\mathbf{c}}(b)} / U_{b,f_{\mathbf{c}}(b)^+} = X_{b,f_{\mathbf{c}}(b)}$ is a $\kappa_{L_b}$-vector space of dimension $1$, hence a $\mathbf{F}_p$-vector space of dimension $f' m$.

\textbf{Second case: $b$ is multipliable and $L_b / L_{2b}$ is ramified.}
By Lemmas \ref{lem:equivalence:affine:root:systems} and \ref{lem:sets:of:values:multipliable:root}, we know that $U_{b,f_{\mathbf{c}}(b)} / V_b = U_{b,f_{\mathbf{c}}(b)} / U_{b,f_{\mathbf{c}}(b)^+} = X_{b,f_{\mathbf{c}}(b)}$ is a $\kappa_{L_a} \simeq \kappa_{L_d}$-vector space of dimension $1$, hence a $\mathbf{F}_p$-vector space of dimension $m = f' m$.

\textbf{Third case: $b$ is multipliable, $L_b / L_{2b}$ is unramified and $f'_{\mathbf{c}}(b) \not\in \Gamma'_a$.}
By Proposition \ref{prop:first:root:group:quotient} and Lemma \ref{lem:equivalence:affine:root:systems}, we know that $U_{b,f_{\mathbf{c}}(b)} / V_b = U_{b,f_{\mathbf{c}}(b)} / U_{b,f_{\mathbf{c}}(b)^+} = X_{2b,2f_{\mathbf{c}}(b)}$ is a $\kappa_{L_{2b}}$-vector space of dimension $1$, hence a $\mathbf{F}_p$-vector space of dimension $m$.

\textbf{Fourth case: $b$ is multipliable, $L_b / L_{2b}$ is unramified and $f'_{\mathbf{c}}(b) \in \Gamma'_a$.}
By Proposition \ref{prop:first:root:group:quotient}, we know that $U_{b,f_{\mathbf{c}}(b)} / V_b = U_{b,f_{\mathbf{c}}(b)} / \Big( U_{b,f_{\mathbf{c}}(b)^+} U_{2b,2f_{\mathbf{c}}(b)} \Big) = X_{b,f_{\mathbf{c}}(b)} / X_{2b,2f_{\mathbf{c}}(b)}$ is a $\kappa_{L_{b}}$-vector space of dimension $1$, hence a $\mathbf{F}_p$-vector space of dimension $2 m = f' m$.

Furthermore, we note that we have the alternative: either $f_{\mathbf{c}}(a) \in \Gamma'_a$ and $f_{\mathbf{c}}(-\theta) \not\in \Gamma'_{-\theta}$, or $f_{\mathbf{c}}(a) \not\in \Gamma'_a$ and $f_{\mathbf{c}}(-\theta) \in \Gamma'_{-\theta}$.
Hence, the sum of dimensions over $\mathbb{F}_p$ of $U_{a,f_{\mathbf{c}}(a)} / V_a$ and $U_{-\theta,f_{\mathbf{c}}(-\theta} / V_{-\theta}$ is always equal to $(f'+1) f m$.

Since there are $l-1$ non-multipliable simple roots, we get $d(P) = m f' (l -1) + (1+f') = m(lf'+1)$.
Let $\xi$ be such that $d(P)=m \xi$.
If $L'/L_d$ is unramified, then $f' = 2$ and $\xi = 2l + 1 = n + 1$.
If $L'/L_d$ is ramified, then $f' = 1$ and $\xi = l + 1$.

\textbf{(2) Suppose that $\Phi$ is non-reduced of rank $1$.}
In this case, we cannot apply Theorem \ref{thm:frattini:descriptions:non:reduced:case} and its Corollary.
Let $H = U_{-a,\frac{1}{2}} T(K)_b^+ U_{a,0}$ be a maximal pro-$p$ subgroup of $G(K) \simeq \mathrm{SU}(h)(K)$, so that $\varepsilon = 0$.
Let $l'' = \max(1,3)=3$.

Suppose that $L/L_2$ is unramified.
By Lemma \ref{lem:non:reduced:derived:valued:root:group}, by Lemma \ref{lem:commutation:non:reduced:torus:root:group} and by Proposition \ref{prop:non:reduced:frattini:rank:one}, we have:
$$U_{-2a,2} U_{-a,\frac{3}{2}} T(K)_b^{l''} U_{a,1} U_{2a,0} \subset [H,H]H^p \subset U_{-2a,2}, U_{-a,1} T(K)_b^+ U_{a,\frac{1}{2}} U_{2a,0}$$
One the one hand, thanks to computation with the quotient groups $X_{a,l}$, we get the $L_2$-vector spaces $U_{a,0} / U_{a, \frac{1}{2}} U_{2a,0} \simeq X_{a,0} / X_{2a,0}$ of dimension $d(a,0) = 2$ and $U_{-a,\frac{1}{2}} / U_{-2a,2}, U_{-a,1} \simeq X_{-a,\frac{1}{2}}$ of dimension $d(-a,\frac{1}{2} + d(-2a,1) = 0 + 1 = 1$.
Hence $d(H) \geq 3 m$.
On the other hand, $U_{a,0} / U_{a,1} U_{2a,0}$ have to be isomorphic to a subgroup of $X_{a,0}/ X_{2a,0} \oplus X_{a,\frac{1}{2}} / X_{2a,1}$, of dimension $d(a,0) + d(a,\frac{1}{2}) = 2$ as $\kappa_{L_2}$-vector space.
In the same way, $U_{-a,\frac{1}{2}} / U_{-2a,2} U_{-a,\frac{3}{2}}$ is isomorphic to a subgroup of $X_{-a,\frac{1}{2}} \oplus X_{-a,1} / X_{-2a,-2}$, of dimension $d(-a,\frac{1}{2}) + d(-2a,1) + d(-a,1) = 0 + 1 +2=3$.
Finally, $T(K)_b^+ / T(K)_b^{l''}$ is of dimension $2(l''-1) = 4$.
Thus $d(H) \leq m ( 5 + 4 ) = 9 m $.

Suppose that $L/L_2$ is ramified.
By Lemma \ref{lem:non:reduced:derived:valued:root:group}, by Lemma \ref{lem:commutation:non:reduced:torus:root:group} and by Proposition \ref{prop:non:reduced:frattini:rank:one}, we have:
$$U_{-2a,3} U_{-a,2} T(K)_b^{l''} U_{a,\frac{3}{2}} U_{2a,1} \subset [H,H]H^p \subset U_{-2a,3}, U_{-a,1} T(K)_b^+ U_{a,\frac{1}{2}} U_{2a,1}$$
One the one hand, thanks to computation with the quotient groups $X_{a,l}$, we get the $L_2$-vector spaces $U_{a,0} / U_{a, \frac{1}{2}} U_{2a,1} \simeq X_{a,0}$ of dimension $d(a,0) + d(2a,0) = 1 + 0$ and $U_{-a,\frac{1}{2}} / U_{-2a,3}, U_{-a,1} \simeq X_{-a,\frac{1}{2}}$ of dimension $d(-a,\frac{1}{2} + d(-2a,1) = 0 + 1 = 1$.
Hence $d(H) \geq 2 m$.
On the other hand, $U_{a,0} / U_{a,\frac{3}{2}} U_{2a,1}$ have to be isomorphic to a subgroup of $X_{a,0} \oplus X_{a,\frac{1}{2}} / X_{2a,1} \oplus X_{a,1} / X_{2a,2}$, of dimension $d(a,0) + d(2a,0) + d(a,\frac{1}{2}) + d(a,1) = 1+0+0+1=2$ as $\kappa_{L_2}$-vector space.
In the same way, $U_{-a,\frac{1}{2}} / U_{-2a,3} U_{-a,2}$ is isomorphic to a subgroup of $X_{-a,\frac{1}{2}} \oplus X_{-a,1} \oplus X_{-a,\frac{3}{2}} / X_{2a,3}$, of dimension $d(-a,\frac{1}{2}) + d(-2a,1) + d(-a,1) + d(-2a,2) + d\left(-a,\frac{3}{2}\right)= 0+1+1+0+0=2$.
Finally, $T(K)_b^+ / T(K)_b^{l''}$ is of dimension $(l''-1) = 2$.
Thus $d(H) \leq m ( 4 + 2 ) = 6 m$.
\end{proof}

\begin{Rq}[Generating set in terms of root groups]
A generating set of $P / \mathrm{Frat}(P)$ always come from a topologically generating set of $P$.
Hence, when the relative root system $\Phi$ is reduced, a system of generators of $P$ is given by:
$$\Big\{ x_a(\lambda_i),\ 1 \leq i \leq m \text{ and } a \in \Delta\Big\} \cup \Big\{\{ x_{-\theta}(\lambda_i \varpi_{L'}),\ 1 \leq i \leq m\Big\}$$
where $(\lambda_i)_{1\leq i \leq m}$ is a family of elements of $\mathcal{O}_{L_d}$ 
such that $(\lambda_i \mathcal{O}_{L_d} / \mathfrak{m}_{L_d})_{1 \leq i \leq m}$ is a basis of $\kappa_{L_d}$;
the root $\theta$ is chosen as in Section \ref{sec:description:apartment};
and $\varpi_{L'}$ is a uniformizer of $\mathcal{O}_{L'}$.
\end{Rq}


\bibliographystyle{beta} 
\bibliography{biblio} 

Benoit Loisel,\newline
CMLS, École polytechnique, CNRS, Université Paris-Saclay, 91128 Palaiseau Cedex, France\newline
benoit.loisel@polytechnique.edu

\end{document}